\documentclass[a4paper]{amsart}

\usepackage[all]{xy}

\usepackage{amsmath}

\usepackage{amssymb}

\usepackage{latexsym}

\usepackage{amsthm}

\usepackage{mathrsfs}

\usepackage{float}

\usepackage{mathdots} %need ``mathdots.sty''

\usepackage{geometry}

\usepackage{enumerate}

\usepackage{mathtools}

\newcommand{\A}{\mathbb{A}}

\newcommand{\cB}{\mathcal{B}}

\newcommand{\C}{\mathbb{C}}
\newcommand{\cC}{\mathcal{C}}
\newcommand{\fc}{\mathfrak{c}}
\newcommand{\D}{\mathbb{D}}

\newcommand{\fd}{\mathfrak{d}}
\newcommand{\E}{\mathbb{E}}
\newcommand{\cE}{\mathcal{E}}
\newcommand{\ue}{\underline{e}}
\newcommand{\F}{\mathbb{F}}
\newcommand{\cF}{\mathcal{F}}
\newcommand{\ff}{\mathfrak{f}}

\newcommand{\cI}{\mathcal{I}}

\newcommand{\cJ}{\mathcal{J}}

\newcommand{\cL}{\mathcal{L}}

\newcommand{\fM}{\mathfrak{M}}

\newcommand{\cO}{\mathcal{O}}

\newcommand{\cP}{\mathcal{P}}
\newcommand{\p}{\mathfrak{p}}
\newcommand{\Q}{\mathbb{Q}}

\newcommand{\R}{\mathbb{R}}
\newcommand{\cS}{\mathcal{S}}

\newcommand{\fu}{\mathfrak{u}}
\newcommand{\V}{\mathbb{V}}

\newcommand{\bW}{\mathbb{W}}
\newcommand{\cW}{\mathcal{W}}
\newcommand{\bX}{\mathbb{X}}

\newcommand{\bY}{\mathbb{Y}}

\newcommand{\Z}{\mathbb{Z}}

\newcommand{\Aff}{\operatorname{Aff}}
\newcommand{\Aut}{\operatorname{Aut}}

\newcommand{\Hom}{\operatorname{Hom}}
\newcommand{\GL}{\operatorname{GL}}

\renewcommand{\Im}{\operatorname{Im}}
\newcommand{\Ind}{\operatorname{Ind}}
\newcommand{\Irr}{\operatorname{Irr}}

\newcommand{\JL}{\operatorname{JL}}
\newcommand{\Lie}{\operatorname{Lie}}
\newcommand{\ord}{\operatorname{ord}}

\newcommand{\SL}{\operatorname{SL}}
\newcommand{\Gal}{\operatorname{Gal}}
\newcommand{\diag}{\operatorname{diag}}
\newcommand{\Mp}{\operatorname{Mp}}
\newcommand{\Res}{\operatorname{Res}}
\newcommand{\Sym}{\operatorname{Sym}}
\newcommand{\SO}{\operatorname{SO}}
\newcommand{\Sp}{\operatorname{Sp}}
\newcommand{\std}{\operatorname{std}}

\newcommand{\rank}{\operatorname{rank}}

\newcommand{\supp}{\operatorname{supp}}
\newcommand{\tru}{\bigtriangleup}
\newcommand{\trd}{\bigtriangledown}

\newtheorem{df}{Definition}[section]
\newtheorem{thm}[df]{Theorem}
\newtheorem{prop}[df]{Proposition}
\newtheorem{lem}[df]{Lemma}
\newtheorem{cor}[df]{Corollary}
\newtheorem{conj}[df]{Conjecture}
\newtheorem{egn}[df]{Example}
\newtheorem{rem}[df]{Remark}
\newtheorem{hyp}[df]{Hypothesis}

\makeatletter
    
    \@addtoreset{equation}{section}
\makeatother

\title{Formal degrees and the local theta correspondence: The quaternionic case}

\author{Hirotaka KAKUHAMA}

\email{hkaku@math.kyoto-u.ac.jp}
\address{Department of Mathematics, Kyoto University, Kitashirakawa Oiwake-cho, Sakyo-ku, Kyoto 606-8502, Japan}

\date{}

\begin{document}

\maketitle

\begin{abstract}
In this paper, we determine a constant occurring in a local analogue of the Siegel-Weil formula, and describe the behavior of the formal degrees under the local theta correspondence for a quaternionic dual pair of almost equal rank over a non-Archimedean local field of characteristic $0$. 
As an application, we prove the formal degree conjecture of Hiraga, Ichino and Ikeda for the non-split inner forms of $\Sp_4$ and ${\rm GSp}_4$.
\end{abstract}

\setcounter{tocdepth}{1}
\tableofcontents

%%%%%%%%%%%%%%%%%%%%%%%%%%%%%%%%%%%%  text   %%%%%%%%%%%%%%%%%%%%%%%%%%%%%%%%%%

\section{
            Introduction
           }

The principal aim of this paper is to describe the behavior of the formal degrees under the local theta correspondence.
This is related to two important topics in the representation theory of $p$-adic reductive groups.
One is the formal degree conjecture of Hiraga, Ichino, and Ikeda \cite{HII08}, which is an explicit formula of the formal degree in terms of the Langlands parameters. Here, by the Langlands parameter we mean a pair $(\phi, \eta)$ where $\phi$ is an L-parameter and $\eta$ is an irreducible representation of the S-group (see \S\ref{PC intro}). The other is the behavior of the Langlands parameters under the local theta correspondence. This has not been formulated yet, but we observe how $\dim \eta$ changes under the theta correspondence associated with a quaternionic dual pair of almost equal rank (Proposition \ref{dim irr S}). 
Moreover, by admitting conjectual properties on Langlands parameters containing the formal degree conjectural, we infer the behavior of the formal degrees under the local theta correspondence (\S\ref{FDCFDCFDC}).
Although the local Langlands correspondence is assumed in these two topics, Gan and Ichino pointed out that only analytic invariants are needed to describe the behavior of the formal degrees under the local theta correspondence associated with a non-quaternionic dual pair of almost equal rank, and computed it \cite{GI14}.
In this paper, we extend their formula to a quaternionic dual pair and prove it unconditionally (Theorem \ref{fd theta1}). This agrees with the observation in \S\ref{FDCFDCFDC}.
As an application, we prove the formal degree conjecture for the non-split inner forms of ${\rm Sp}_4$ and ${\rm GSp}_4$.

We prove Theorem \ref{fd theta1} by induction.
As in \cite{GI14}, the induction steps are attained by using a formula of Heiermann \cite{Hei04}.
However, it is difficult to prove the base case by case-by-case discussions similar to \cite{GI14}. More precisely, it seems difficult to find enough examples of quaternionic dual pairs $(H,G)$ and square-integrable irreducible representations $\pi$ of $G$ such that we can know the formal degree $\deg \pi$ of $\pi$, the formal degree $\deg \sigma$ of the theta correspondence $\sigma$ of $\pi$, and the standard local $\gamma$-factor $\gamma(s, \sigma\boxtimes\chi, \psi)$ with a quadratic character $\chi$ at the same time even in low-rank cases.  
To avoid the difficulty, we analyze the local analogue of the Siegel-Weil formula, and we obtain a relation between the constant in the local Siegel-Weil formula and the local zeta value for enough cases.
Here, the constant in the local Siegel-Weil formula appears in an expression of the ratio of the formal degrees of irreducible representations corresponding to each other by the local theta correspondence. 
Hence, to establish the description of the behavior of the formal degrees under the local theta correspondence, we compute some local zeta values. On the other hand, a general formula of the local zeta value is obtained by reversing the above discussion. 
For a quaternionic dual pair $(G(V), G(W))$ of almost equal rank, we denote by $\alpha_1(W)$ the local zeta value, by $\alpha_2(V,W)$ the constant in the local Siegel-Weil formula, and by $\alpha_3(V,W)$ the constant appearing in the behavior of the formal degree under the theta correspondence. 
The results obtained in this paper are summarized as follows.

\subsection{The constant $\alpha_1(W)$}\label{alpha_1intro}
%a certain local zeta value
Let $F$ be a non-Archimedean local field of characteristic $0$, let $\epsilon = \pm 1$, let $W$ be an $n$-dimensional right $(-\epsilon)$-Hermitian space equipped with a $(-\epsilon)$-Hermitian form $\langle \ , \ \rangle$ (see \S\ref{set and not}) over a division quaternion algebra $D$ over $F$, and let $G(W)$ be the unitary group of $W$. We denote by $W^\Box$ the doubled space which is the vector space $W\oplus W$ equipped with the $(-\epsilon)$-Hermitian form $\langle \ , \ \rangle^\Box = \langle\ , \rangle \oplus (-\langle\ , \ \rangle)$, by $W^\tru$ the diagonal subset of $W^\Box$, and by $P(W^\tru)$ the parabolic subgroup preserving $W^\tru$. For a character $\omega$ of $F^\times$, we denote by $I(s,\omega)$ the representation of $G(W^\Box)$ induced by the character $\omega_s\circ\Delta$ of $P(W^\tru)$, which is given by $\omega_s(\Delta(p)) = \omega(N(p|_{W^\tru}))^{-1}|N(p|_{W^\tru})|^{-s}$. (Here, we denote by $N(x)$ the reduced norm of $x \in {\rm End}(W^\tru)$.) We denote by $1_{F^\times}$ the trivial character of $F^\times$. Then we define
\[
\alpha_1(W) = Z^W(f_\rho^\circ, \xi^\circ).
\]
Here  
\begin{itemize}
\item $Z^W(\  , \ )$ is the doubling zeta integral (see \S\ref{doubling}),
\item $f_s^\circ$ is the $K(\underline{e'}^\Box)$-invariant section of $I(s, 1_{F^\times})$ so that $f_s^\circ(1) = 1$ where $K(\underline{e'}^\Box)$ is a special maximal compact subgroup of the unitary group $G(W^\Box)$ of $W^\Box$, which depends on the choice of a basis $\underline{e}$ for $W$ (see \S\ref{doubling}),
\item $\xi^\circ$ is the coefficient of the trivial representation of $G(W)$ so that $\xi^\circ(1) = 1$, and
\item $\rho = n - \frac{\epsilon}{2}$.
\end{itemize}
This invariant is technically important because it appears in a certain local functional equation, which relates the zeta integral with the intertwining operator (see Lemma \ref{ratio}). In this paper, we first compute $\alpha_1(W)$ directly for some $W$ (Proposition \ref{alpha1 min}), and finally, we complete the formula for the remaining cases as a corollary of Theorem \ref{FTintro} (Proposition \ref{alpha_1 comp}). 
We also note that by determining $\alpha_1(W)$, we can compute the constant given by the scalar multiplication appearing in a formula of the zeta integral for a certain section (see \S\ref{App}), which has not been computed yet. 

\subsection{The constant $\alpha_2(V,W)$}
%local analogue of SW formula

Let $V$ be an $m$-dimensional $\epsilon$-Hermitian space, and let $( \ , \ )$ be an $\epsilon$-Hermitian form on $V$, let $\psi\colon F \rightarrow \C^\times$ be an additive non-trivial unitary character, and let $\omega_\psi^\Box$ be the Weil representation of $G(V) \times G(W^\Box)$. We realize $\omega_\psi^\Box$ on the Schwartz space $\cS(V\otimes W^\trd)$ where $W^\trd$ is the anti-diagonal subset of $W^\Box$. We assume that $2n-2m = 1+\epsilon$. Then, we define the local theta integral
\[
\cI(\phi, \phi') = \int_{G(V)} (\omega_\psi^\Box(h,1)\phi, \phi') \: dh
\]
for $\phi, \phi' \in \cS(V\otimes W^\trd)$. Here, we denote by $( \ , \ )$ the normalized $L^2$-inner product on $\cS(V\otimes W^\trd)$ (as in Proposition \ref{inner prods}). Moreover, we define another map $\cE\colon \cS(V\otimes W^\trd)^2\rightarrow \C$ as follows. For $\phi \in \cS(V\otimes W^\trd)$, we define $F_\phi \in I(-\frac{1}{2}, \chi_V)$ by $F_\phi(g) = [\omega_\psi^\Box(g)\phi](0)$, and we choose $F_\phi^\dagger \in I(\frac{1}{2}, \chi_V)$ so that $M(\frac{1}{2}, \chi_V)F^\dagger = F_\phi$ where $M(s, \chi_V)$ is an intertwining operator (see \S\ref{doubling}). Then, the map $\cE$ is defined by
\[
\cE(\phi, \phi') = \int_{G(W)} F_\phi^\dagger(\iota(g,1)) \overline{F_{\phi'}(\iota(g,1))} \: dg
\]
for $\phi, \phi' \in \cS(V\otimes W^\trd)$. Here, $\iota\colon G(W)^2 \rightarrow G(W^\Box)$ is given by the natural action of $G(W)^2$ on $W^\Box$. Then, the constant $\alpha_2(V,W)$ is defined as a non-zero constant satisfying $\cI = \alpha_2(V,W) \cdot \cE$ (see Lemma \ref{locSW2}). Then, we have

\begin{thm} \label{SWintro}
Choose the basis $\underline{e} = (e_1, \ldots, e_m)$ for $V$ as in \S\ref{doubling}. Then,
\begin{align*}
\alpha_2(V,W) &= |N(R(\underline{e})|^\rho\cdot\prod_{i=1}^{n-1}\frac{\zeta_F(1-2i)}{\zeta_F(2i)}\\
&\times\begin{cases}
2\chi_V(-1)^n\gamma(1-n, \chi_V, \psi)^{-1}\epsilon(\frac{1}{2}, \chi_V, \psi) & (-\epsilon = 1), \\
\epsilon(\frac{1}{2}, \chi_W, \psi) & (-\epsilon = -1). 
\end{cases}
\end{align*}
Here, $R(\underline{e}) = ((e_i, e_j))_{i,j} \in \GL_m(D)$.
\end{thm}

To prove this theorem, we will first prove it in the case where either $V$ or $W$ is non-zero anisotropic (\S\S\ref{Min1}--\ref{Min2}). In this case, we can express $\alpha_2(V,W)$ using $\alpha_1(W)$, and thus Theorem \ref{SWintro} is derived from the formula of $\alpha_1(W)$.
The remaining cases will be proved after completing the proof of Theorem \ref{FTintro} (\S\ref{det alpha}).

\subsection{The constant $\alpha_3(V,W)$}\label{a3intro}
%the ratio of formal degree
Let $G$ be a reductive group over $F$, let $A$ be the maximal $F$-split torus of the center $G$, and $\pi$ be a square-integrable irreducible representation of $G$. We choose a canonical Haar measure $dg$ on $G(F)/A(F)$ depending only on $G$ and a fixed non-trivial additive character $\psi$ of $F$ as in \cite[\S8]{GG99}. If $G = G(W)$, the measure is given in \S\ref{Haar1}. Then, we define the formal degree of $\pi$ as the positive real number $\deg \pi$ satisfying
\[
\int_{G(F)/A(F)} (\pi(g)x_1, x_2)\overline{(\pi(g)x_3, x_4)} \:dg= \frac{1}{\deg \pi}(x_1,x_3)\overline{(x_2,x_4)}
\]
for $x_1, x_2, x_3, x_4 \in \pi$. Here, $( \ , \ )$ is a non-zero $G(F)$-invariant Hermitian form. Suppose that $\theta_\psi(\pi, V)$ is non-zero and is square-integrable. We denote its central character by $c_{\theta_\psi(\pi, V)}$.
Then, as in \cite[p.~597]{GI14}, we can prove that
\[
\frac{\deg\pi}{\deg\theta_\psi(\pi, V)}\cdot c_{\theta_\psi(\pi, V)}(-1)\cdot\gamma^V(0, \theta_\psi(\pi, V)\times\chi_W, \psi)^{-1}
\]
does not depends on $\pi$ whenever $\pi$ is square irreducible and $\theta_\psi(\pi, V) \not=0$. We denote it by $\alpha_3(V, W)$.
Then, our main theorem is stated as follows: 

\begin{thm} \label{FTintro}
We have
\[
\alpha_3(V,W) =\begin{cases}
\epsilon(\frac{1}{2}, \chi_V, \psi)^{-1} & (-\epsilon = 1), \\
\frac{1}{2}\chi_W(-1)^m\epsilon(\frac{1}{2}, \chi_W, \psi)^{-1} & (-\epsilon = -1).
\end{cases}
\]
\end{thm}

When either $V$ or $W$ is anisotropic, we prove this theorem by expressing $\alpha_3(V,W)$ using $\alpha_2(V,W)$ (more precisely, see Proposition \ref{alpha23}). In general, we use induction on $\dim W$ to compute $\alpha_3(V,W)$ (\S\ref{ind_arg}). 

\subsection{Langlands parameters and the local theta correspondence}
\label{PC intro}

Let $G$ be a reductive group over $F$, and let $\pi$ be an irreducible representation of $G$, let $\Gamma$ be the Galois group of $F^s/F$ where $F^s$ denotes the separate closure of $F$, let $W_F$ be a Weil group of $F$, let $L_F = W_F\times \SL_2(\C)$ be the Langlands group of $F$, let $\widehat{G}$ be the Langlands dual group of $G$, let $Z(\widehat{G})$ be the center of $\widehat{G}$, let $\widehat{G}_{\rm ad}$ be the adjoint group $\widehat{G}$, let $\widehat{G}_{\rm sc}$ be the simply connected covering of $\widehat{G}_{\rm ad}$, and let ${}^L\!G$ be the $L$-group of $G$.
The Langlands parameter of $\pi$ is given by a pair $(\phi, \eta)$ where 
\begin{itemize}
\item $\phi\colon L_F \rightarrow {}^L\!G$ is the $L$-parameter of $\pi$,
\item $\eta$ is an irreducible representation of the component group $\widetilde{\cS}_\phi = \pi_0(\widetilde{S}_\phi)$ of $\widetilde{S}_\phi$ where $\widetilde{S}_\phi$ is the preimage of $S_\phi \coloneqq   C_\phi/Z(\widehat{G})^{\Gamma} \subset \widehat{G}_{\rm ad}$ in $\widehat{G}_{\rm sc}$.
\end{itemize}
Here, we denote by $C_\phi$ the centralizer of $\Im \phi$ in $\widehat{G}$. The group $\widetilde{\cS}_\phi$ is called the S-group of $\phi$.
See \S\ref{FDCFDCFDC} for the discussion on how we define $\eta$ from $\pi$.
Now we consider the pair $(G(W), G(V))$ with $2n-2m = 1+\epsilon$ again. 
Then, we have the following:
\begin{prop}
Assume Hypothesis \ref{Lpar} and Conjecture \ref{GRFDC} hold.
Let $\pi$ be a tempered irreducible representation of $G(W)$ and let $(\phi, \eta)$ be its Langlands parameter. Then $\theta_\psi(\pi, V)$ is nonzero if and only if $\std\circ\phi$ contains $\chi_V\boxtimes 1_{\SL_2(\C)}$ as representations of $W_F\times\SL_2(\C)$. Here, $\std$ is the standard embedding of ${}^L\!G$ into $\GL_N(\C)$ and $1_{\SL_2(\C)}$ is the trivial representation of $\SL_2(\C)$. Suppose that $\theta_\psi(\pi, V)$ is nonzero, and we denote it by $\sigma$. We denote by $(\phi', \eta')$ its Langlands parameter. Then we have
\[
\std\circ\phi \cong (\std\circ\phi'\otimes\chi_V\chi_W^{-1})\oplus (\chi_V\boxtimes1_{\SL_2(\C)})
\]
as representations of $W_F\times \SL_2(\C)$, and we have
\[
\frac{\dim \eta}{\dim \eta'} = \begin{cases} 1 & (\epsilon = 1), \\ 
1 & (\epsilon = -1, {\phi'}^\varepsilon = \phi'), \\ 2 & (\epsilon =-1, {\phi'}^\varepsilon \not= \phi').\end{cases}
\]
Here, $\varepsilon$ is the generator of ${\rm Out}(\widehat{G(V)})$.
\end{prop}

Using this proposition, we verify that our main theorem (Theorem \ref{FTintro}) is consistent with Hypothesis \ref{Lpar} and Conjecture \ref{GRFDC}  (\S\ref{FDCFDCFDC}). 

\subsection{Formal degree conjecture for non-split inner forms of $\Sp_4$ and ${\rm GSp}_4$}\label{fdc_intro}

Let $G$ be a reductive group over $F$, let $\pi$ be a square-integrable irreducible representation of $G(W)$, and let $(\phi, \eta)$ be its Langlands parameter. We denote by $A$ the maximal $F$-split torus of the center of $G$, and we put $C_\phi' = \widehat{G/A} \cap C_\phi$. Then, the formal degree conjecture of Hiraga, Ichino, and Ikeda asserts that
\[
\deg\pi = \frac{\dim \eta}{\# C_\phi'} \cdot |\gamma(0, {\rm Ad}\circ\phi, \psi)|.
\]
where ${\rm Ad}$ is the adjoint representation of ${}^L\!G$ on $\Lie(\widehat{G}_{\rm ad})$ \cite{HII08}.
This conjecture has been proved for reductive groups over Archimedean local fields and for the inner forms of $\GL_n$ by themselves \cite{HII08}. It has also been proved for ${\rm SO}_{2n+1}$, ${\rm Mp}_{2n}$, ${\rm U}_n$, and ${\rm Sp}_4$ (\cite{ILM17}, \cite{BP18}, \cite{GI14}). 
Moreover, Gross and Reeder reformulated it by using the Eular-Poinc\'{a}re measure on $G$ \cite{GR10}. Note that the absolute value does not appear in their reformulation. 
For the non-split inner forms of ${\rm GSp}_4$ and $\Sp_4$, the Langlands correspondence is constructed by Gan and Tantono \cite{GT14} and Choiy \cite{Cho17} respectively. We prove the conjecture for these groups by using Theorem \ref{FTintro}:

\begin{thm}
Let $F$ be a local field of characteristic $0$. Then, the formal degree conjecture holds for the non-split inner forms of $\Sp_4$ and ${\rm GSp}_4$.
\end{thm}

\subsection{Structure of this paper}
Now, we explain the structure of this paper.
In \S\S\ref{quaternion}--\ref{set and not}, we set up the notations for fields, quaternion algebras, and $\pm\epsilon$-Hermitian spaces. In \S\ref{basis for WV}, we define some symbols which are referred to when we take bases for $\pm\epsilon$-Hermitian spaces.
In \S\ref{BTtheory}, we recall the Bruhat-Tits theory for quaternionic unitary groups and define the Iwahori subgroup. In \S\ref{Haar measures}, we explain the normalization of Haar measures on reductive groups and certain nilpotent groups, and we give some volume formulas. In \S\ref{doubling and gamma}, we explain the doubling method and recall the definition of the doubling $\gamma$-factor. Moreover, we compute the constant $\alpha_1(W)$ for some cases. 
In \S\ref{Local Weil reps}, we set up and explain the doubling method and the Weil representations. In \S\ref{Loc theta corr}, we set up the theta correspondence.
In \S\S\ref{SS local SW formula}--\ref{main theorem}, \ref{det alpha}-\ref{FDC}, we state our main results, and in \S\S\ref{Min1}--\ref{ind_arg}, we prove these results. More precisely, \S\S\ref{Min1}--\ref{Min2} are devoted to the computation of $\alpha_2(V,W)$ when either $V$ or $W$ is anisotropic, \S\ref{gamma under theta} is a preliminary for \S\ref{loc Rallis inn prod formula} which associates $\alpha_2(V,W)$ with $\alpha_3(V,W)$, and \S\S\ref{planc mes}--\ref{accidental isom} are preliminaries for \S\ref{ind_arg} in which we verify the commutativity of $\alpha_3(V,W)$ with the parabolic inductions.
Finally, in Appendix \S\ref{App}, we give a formula for doubling zeta integrals of certain sections as an application of the formula of $\alpha_1(W)$. 
Note that this corrects the errors in \cite[Proposition 8.3]{Kak20}.

\subsection*{Acknowledgements} The author would like to thank A.Ichino for suggesting this problem, and for useful discussions, W.T.Gan for his useful comments, and H.Atobe for his valuable comments on L-parameters. 
%The author also would like to thank the anonymous referee for many helpful comments. 
This research was supported by JSPS KAKENHI Grant Number JP20J11509.

%%%%%%%%%%%%%%%%%%%%%%%%%%%%%%%%%    S2      %%%%%%%%%%%%%%%%%%%%%%%%%%%%%%%%%%

\section{
            Quaternion algebras over local fields
           }\label{quaternion}

Let $F$ be a non-Archimedean local field of characteristic $0$, and let $D$ be a quaternion algebra over $F$. In this paper except \S\ref{gamma under theta} and \ref{loc_theta_Planc}, we assume that $D$ is division. 
We denote by $\cO_F$ the valuation ring of $F$, by $\varpi_F$ a uniformizer of $\cO_F$, by $\ord_F\colon F^\times \rightarrow \Z$ the additive valuation normalized so that $\ord_F(\varpi_D) = 1$, by $q$ the cardinality of $\cO_F/\varpi_F$, by $| \ |_F$ the absolute value normalized so that $|\varpi_F|_F = q^{-1}$, by $*\colon D \rightarrow D$ the canonical involution of $D$, by $N_D\colon D \rightarrow F$ the reduced norm, $T_D\colon D \rightarrow F$ the reduced trace, by $\ord_D = \ord_F\circ N_D$ the additive valuation of $D$,
by $| \ |_D = | \ |_F\circ N_D$ the absolute value of $D$, and by $\cO_D$ the valuation ring of $D$.

\begin{lem}\label{quatb}
There exist two elements $\delta$ and $\varpi_D$ of $D$ so that the subfield $F(\delta)$ is unramified over $F$, $T_D(\delta) = T_D(\varpi_D) = 0$, $\ord_D(\delta) = 0$, $\ord_D(\varpi_D) = 1$, and $\delta\varpi_D + \varpi_D\delta = 0$.
\end{lem}

\begin{proof}
Take $d \in F$ so that $\ord_F(d) = 0$ and $F(\sqrt{d})$ is unramified over $F$.
 To prove the claim, it suffices to show that the quaternion algebra $(d, \varpi_F/ F)$ is isomorphic to $D$. Since the $2$-torsion subgroup of the Brauer group of $F$ is isomorphic to $\Z/2\Z$, it remains to show that $(d, \varpi_F/ F)$ is division. This is obtained by the fact that $\varpi_F$ is not contained in the image of the norm map of the unramified extension $F(\sqrt{d})/F$. Hence we have the lemma.
\end{proof}

%%%%%%%%%%%%%%%%%%%%%%%%%%%%%%%%%%%%%%%%%%%%%%%%%%%%%%%%%%%%%%%%%%%%%%%%%%%%%%%%%%%%%%

\section{
		$\epsilon$-Hermitian spaces and their unitary groups
            }\label{set and not}

Let $\epsilon \in \{\pm1\}$. Now, we consider the following:
\begin{itemize}
\item a pair $(W, \langle \ , \ \rangle )$ where $W$ is a left free $D$-module of rank $n$, and $\langle \ , \ \rangle$ is a map $W \times W \rightarrow D$ satisfying
\[
\langle ax, by \rangle = a \langle x,y \rangle b^*, \ \langle y, x\rangle = -\epsilon \langle x, y \rangle
\]
for $x, y \in W$ and $a,b \in D$,
\item a pair $(V, (\ ,\ ))$ where $V$ is a right free $D$-module of rank $m$, and $(\ , \ )$ is a map $V \times V \rightarrow D$ satisfying
\[
(v_1a, v_2b) = a^* (x, y) b, \ (y, x) = \epsilon (x, y)^* 
\]
for $x, y \in V$ and $a,b \in D$.
\end{itemize}
We call them an $n$-dimensional right $\epsilon$-Hermitian space and an $m$-dimensional left $(-\epsilon)$-Hermitian space respectively if they are non-degenerate. 
Then, we define $G(W)$ by the group of the left $D$-linear automorphisms $g$ of $W$ and
\[
\langle x \cdot g, y \cdot g \rangle = \langle x, y\rangle
\]
for all $x,y \in W$. We also define $G(V)$ by the group of the right $D$-linear automorphisms $g$ of $V$ and
\[
(g\cdot x, g\cdot y) = (x, y)
\]
for all $x, y \in V$. Put $\bW = V \otimes_F W$ and define $\langle\langle \ , \ \rangle\rangle$ by
\[
\langle\langle x_1\otimes y_1 ,x_2\otimes y_2 \rangle\rangle = T_D((x_1,y_1)\langle x_2,y_2\rangle^*)
\]
for $x_1, y_1 \in V$ and $x_2, y_2 \in W$. Then, $\langle\langle\ , \ \rangle\rangle$ is a symplectic form on $\bW$, and the $(G(W), G(V))$ is a reductive dual pair in $\Sp(\bW)$. We define
\[
l = l_{V,W} = 
\begin{cases} 2n - 2m - 1 & (\epsilon = 1), \\ 2n - 2m +1 & (\epsilon = -1). \end{cases}
\]
We define the characters $\chi_V$ and $\chi_W$ of $F^\times$ by
\[
\chi_V(a) = \begin{cases} 1_{F^\times} & (\epsilon = 1),  \\ (a,\fd(W))_F & (\epsilon = 1) \end{cases}
\mbox{ and }
\chi_W(a) = \begin{cases} (a, \fd(V))_F & (\epsilon = -1), \\ 1_{F^\times} & (\epsilon = -1) \end{cases}
\]
for $a \in F^\times$.

%%%%%%%%%%%%%%%%%%%%%%%%%%%%%%%%%    S4      %%%%%%%%%%%%%%%%%%%%%%%%%%%%%%%%%

\section{
            Bases for $W$ and $V$
            }\label{basis for WV}

In this section, we discuss bases for $W$, which we will consider in this paper. The discussion for $V$ goes the same line as that of $W$. For a basis $\ue = \{e_1, \ldots, e_n\}$ for $W$, we define 
\[
R(\ue) \coloneqq   (\langle e_i, e_j\rangle)_{ij} \in \GL_n(D).
\]
Denote by $W_0$ the anisotropic kernel of $W$, and put $n_0 = \dim_DW_0$, $r = \frac{1}{2}(n-n_0)$. We assume that 
\[
W_0 = \sum_{i=r+1}^{r+n_0} e_i D, 
\]
both 
\[
X = \sum_{i=1}^re_iD \mbox{ and } \sum_{i=r+n_0+1}^ne_iX^*
\]
are totally isotropic subspaces of $W$, and 
\begin{align}\label{formform}
R(\underline{e}) =\begin{pmatrix} 0 & 0 & J_r \\ 0 & R_0 & 0 \\ -\epsilon J_r & 0 & 0 \end{pmatrix}
\end{align}
where 
\[
J_r = \begin{pmatrix}  & & 1 \\ & \iddots & \\ 1 & & \end{pmatrix},
\]
and $R_0 \in \GL_{n_0}(D)$. By this basis, we regard $G(W)$ as a subgroup of $\GL_n(D)$. It is known that
\[
\begin{cases}
n_0 \leq 1 & (-\epsilon = 1),\\
n_0 \leq 3 & (-\epsilon = -1).
\end{cases}
\]
Moreover, in the case $-\epsilon = -1$, it is known that $n_0 = 2$ if and only if $n$ is even and $\chi_V\not=1_{F^\times}$, $n_0 = 3$ if and only if $n$ is odd and $\chi_V = 1_{F^\times}$.
(Cf. \cite[\S10, Example 1.8 (ii) and Theorem 3 .6]{Sch85}.)

%%%%%%%%%%%%%%%%%%%%%%%%%%%%%%%%%%%%%%%%%%%%%%%%%%%%%%%%%%%%%%%%%%%%%%%%%%%%%%

\section{
            Bruhat-Tits theory
            }\label{BTtheory}

In this section, we recall the definition and construction of the Iwahori subgroups of quaternionic unitary groups.
Before giving the definition, we discuss the apartments.

\subsection{Apartments}\label{apartment}
Take a basis $\underline{e}$ as in \S\ref{basis for WV}. Put $I = \{ e_1, \ldots, e_r\}$, $I_0 = \{e_{r+1}, \ldots, e_{n-r}\}$, and $I^*=\{e_{n-r+1}, \ldots, e_{n}\}$. We denote by $S$ the maximal $F$-split torus
\[
 \{\diag(x_1, \ldots, x_r, 1 , \ldots, 1, x_r^{-1}, \ldots, x_1^{-1}) \mid x_1, \ldots, x_r \in F^\times\}
\]
of $G(W)$. We denote by $Z_{G(W)}(S)$ the centralizer of $S$ in $G(W)$, by $N_{G(W)}(S)$ the normalizer of $S$ in $G(W)$, by $\cW= N_{G(W)}(S)/Z_{G(W)}(S)$ the relative Weyl group with respect to $S$, by $\Phi$ the relative root system of $G(W)$ with respect to $S$, by $X^*(S)$ the group of algebraic characters of $S$, by $E^\vee$ the vector space $X^*(S)\otimes_\Z\R$, and by $E$ the $\R$ dual space of $E^\vee$. Moreover, we define the bilinear map $\langle \ , \ \rangle\colon E\times E^\vee\rightarrow \R$ by $\langle y, \eta \rangle = \eta(y)$ for $y \in E^\vee$ and $\eta \in E$. Then, we can define the map $\mu\colon Z_{G(W)}(S) \rightarrow E$ by
\[
[\mu(z)](a') = -\ord_F(a'(z))
\]
for $a' \in X^*(S)$. Then, there is a unique morphism $\nu\colon N_{G(W)}(S) \rightarrow \Aff(E)$ so that the following diagram is commutative:
\[
\xymatrix{
1\ar[r] & Z_{G(W)}(S) \ar[r]\ar[d]_\mu & N_{G(W)}(S) \ar[r]\ar[d]_\nu & \cW \ar[r]\ar[d] & 1 \\
0 \ar[r]& E \ar[r] & \Aff(E) \ar[r]& \Aut(E) \ar[r]& 1  }.
\]
For $a \in \Phi$, we denote by $X_a$ the root subgroup in $G(W)$. 
Let $u \in X_a\setminus\{1\}$. Then one can prove that $X_{-a}\cdot u\cdot X_{-a} \cap N_{G(W)}(S)$ consists of an unique element. We denote it by $m_a(u)$. 
We define a map $\varphi_a\colon X_a\setminus\{1\} \rightarrow \R$ by
\[
m_a(u)(\eta) = \eta - (\langle a, \eta \rangle+\varphi_a(u))a^\vee
\]
for all $\eta \in E$. We put $\Phi_{\rm aff}$ the affine root system
\[
 \{(a,t) \mid a \in \Phi, \ t =\varphi_a(u) \mbox{ for some $u \in X_a\setminus\{1\}$}\} \subset \Phi\times\R,
\]
and by $E_{a,t}$ the subset $\{ \eta \in E \mid [m_a(u)](\eta)=\eta\}$ where $u \in X_a$ so that $\varphi_a(u)=t$. We call a connected component of 
\[
E \setminus \bigcup_{(a, t)\in \Phi_{\rm aff}} E_{a,t}
\]
a chamber of $E$. For $i \in I$ (resp $i \in I^*)$, we define $a_i \in X^*(S) \subset E^\vee$ by $a_i(x) = N_D(x_i)$ (resp. $a_i(x) = N_D(x_{n+1-i})^{-1}$) for 
\[
x = \diag(x_1, \ldots, x_r, 1, \ldots, 1, x_r^{-1}, \ldots, x_1^{-1}) \in S.
\]

Now we describe $\varphi_a$ explicitly following \cite[\S10]{BT72}. The root system of $G(W)$ with respect to $S$ is divided into
\[
\Phi = \Phi_1^+\cup \Phi_1^- \cup \Phi_2^+\cup\Phi_2^- \cup \Phi_3^+\cup\Phi_3^- \cup \Phi_4^+\cup\Phi_4^-
\]
where
\begin{align*}
\Phi_1^+ &= \{ a_i - a_j \mid 1\leq j  < i \leq r \}, \\
\Phi_2^+ &= \{ a_i \mid i=1, \dots, r\}, \\
\Phi_3^+ &= \{ a_i + a_j \mid 1\leq j < i  \leq r \}, \\
\Phi_4^+ &= \{ 2a_i \mid i=1, \ldots, r\},
\end{align*}
and $\Phi_k^- = -\Phi_k^+$ for $k=1,2,3,4$. Let $a=a_i-a_j \in \Phi_1^+\cup\Phi_1^-$. For $x \in D$, we define $u_a(x) \in X_a$ by
\[
e_k \cdot u_a(x) =
\begin{cases}
e_k & (k\not=i,n-i), \\
e_i + x\cdot e_j & (k=i), \\
e_{n-i} + x^*\cdot e_{n-j} & (k=n-i).
\end{cases}
\]
Let $a=a_i \in \Phi_2^+$. For $c=(c_1, \ldots, c_{n_0}) \in W_0=D^{n_0}$ and $d \in D$ with $(d^*-\epsilon d) + \langle c, c \rangle=0$, we define $u_a(c,d) \in X_a$ by
\[
e_k\cdot u_a(c,d) =
\begin{cases}
e_k &  (k\not=i, r+1, \ldots, r+n_0), \\
e_i+\sum_{t=1}^{n_0}c_te_{r+t} + de_{n-i} & (k=i), \\
e_k + \alpha_{k-r}c_{k-r}^*e_{n-i} & (k=r+1, \ldots, r+n_0).
\end{cases}
\]
Let $a=-a_i \in \Phi_2^-$. For $c=(c_1, \ldots, c_{n_0}) \in W_0=D^{n_0}$ and $d \in D$ with $(d-\epsilon d^*)+ \langle c,c\rangle=0$, we define $u_a(c,d) \in X_a$ by
\[
e_k\cdot u_a(c,d) = \begin{cases}
e_k & (k \not= r+1, \ldots, r+n_0, n-i), \\
e_k - \alpha_{k-r}c_{k-r}^*e_i & (k=r+1, \ldots, r+n_0), \\
de_i + \sum_{t=1}^{n_0}c_te_{r+t} + e_{n-i} & (k=n-i).
\end{cases}
\]
Let $a=(a_i+a_j) \in \Phi_3^+$. For $x \in D$, we define $u_a(x) \in X_a$ by
\[
e_k\cdot u_a(x) =\begin{cases}
e_k & (k \not=i,j), \\
e_i + x\cdot e_{n-i} & (k=i), \\
e_j +\epsilon x^* e_{n-j} & (k=j). \\
\end{cases}
\]
Let $a \in \Phi_3^-$. For $x \in D$, we define $u_a(x) \coloneqq   {}^t\!u_{-a}(x)^* \in X_a$. 
Finally, let $a = \pm 2a_i \in \Phi_4^\pm$. For $d \in D$ with $d^*-\epsilon d = 0$, we define $u_{a}(d)\coloneqq   u_{\pm a_i}(0,d) \in X_{2a}$. 

\begin{lem}
For $a \in \Phi$, we have
\begin{itemize}
\item $\varphi_a (u_a(x)) = \ord_D(x)$ for $x \in D$ if $a \in \Phi_1^+\cup\Phi_1^-\cup\Phi_3^+\cup\Phi_3^-$,
\item $\varphi_a (u_a(c,d)) = \frac{1}{2}\ord_D(d)$ for $c \in D^{n_0}$ and $d\in D$ with $(d^* - \epsilon d) \pm \langle c,c\rangle = 0$ if $a\in \Phi_2^\pm$,  
\item $\varphi_a (u_a(d)) =\ord_D(d)$ for $d \in D$ with $d^*-\epsilon d = 0$ if $a \in \Phi_4^+\cup\Phi_4^-$. 
\end{itemize}
\end{lem}

%%%%%%%%%%

\subsection{Iwahori subgroups}\label{iwahori}

Before stating the definition of the Iwahori subgroup, we explain a map of Kottwitz.
Let $F^{\rm ur}$ be the maximal unramified extension of $F$, let $F^s$ be the separable closure of $F$,  let $I = {\rm Gal}(F^s/F^{\rm ur})$ be the inertia group of $F$, and let ${\rm Fr}$ be a Frobenius element. 
Then, Kottwitz defined a surjective map 
\[
\kappa_{W}\colon G(W) \rightarrow \Hom(Z(\widehat{G(W)})^I, \C^\times)^{{\rm Fr}}
\]
(see \cite[\S7.4]{Kot97}). Here, we denote by $\widehat{G(W)}$ the Langlands dual group of $G(W)$, by $Z(\widehat{G(W)})^I$ the $I$-invariant subgroup of the center of $\widehat{G(W)}$, and by $\Hom(Z(\widehat{G(W)})^I, \C^\times)^{{\rm Fr}}$ the ${\rm Fr}$-invariant subgroup of $\Hom(Z(\widehat{G(W)})^I, \C^\times)$. Then, an Iwahori subgroup of $G(W)$ is defined to be a subgroup consisting of the elements $g$ of $G(W)$ which preserves each point of a chamber of the building and $\kappa_W(g) = 1$.
Now we describe an Iwahori subgroup of $G(W)$. Let $\cC$ be a chamber in $E$ so that 
\begin{itemize}
\item for any root $a \in \Phi(S,G(W))$ with $X_a \subset B$, $\langle a, \cC\rangle \subset \R_{>0}$,
\item the closure $\overline{\cC}$ of $\cC$ contains the origin $0 \in E$.
\end{itemize}
Then, the Iwahori subgroup associated with the chamber $\cC$ is given by 
\[
\cB\coloneqq   \{ g \in G(W) \mid \kappa_{W}(g) = 1 \mbox{ and } g \cdot p = p \mbox{ for all $p \in \cC$}\}.
\]
By the construction of the map $\kappa_W$, the following diagram is commutative:
\[
\xymatrix{Z_{G(W)}(S) \ar[d]\ar[rr]^-{\kappa_{Z_{G(W)}(S)}} & & \Hom(Z(\widehat{Z_{G(W)}(S)})^I, \C^\times)^{{\rm Fr}} \ar[d] \\ G(W)\ar[rr]^-{\kappa_{W}} & & \Hom(Z(\widehat{G(W)})^I, \C^\times)^{{\rm Fr}} }
\]
where the vertical maps are (induced from) the natural embeddings. Hence, we have:
\begin{lem}\label{decomp iwahori}
\[ 
\cB = Z_{G(W)}(S)_1 \cdot \prod_{a \in \Phi^+}X_{a,0} \cdot \prod_{a \in \Phi^-}X_{a,\frac{1}{2}}
\]
where $Z_{G(W)}(S)_1$ is the set of matrices
\[
 \begin{pmatrix} a & 0 & 0 \\ 0 & g_0 & 0 \\ 0 & 0 & {a^*}^{-1} \end{pmatrix}   \ (a = \diag(a_1, \ldots, a_r), g_0 \in G(W_0))
\]
such that $a_i \in \cO_D^\times$ for $i=1, \ldots, r$, and $\kappa_{W_0}(g_0) = 1$. 
Here, we denote by $X_{a,t}$ the subset
\[
\{ u \in X_a \mid \varphi_a(u) \geq t \}
\]
of $X_a$ for $t \in \R$.
\end{lem}
%%%%%%%%%%%%%%%%%%%%%%%%%%%%%%%%%   S6      %%%%%%%%%%%%%%%%%%%%%%%%%%%%%%%%%%%%

\section{
            Haar measures
            }\label{Haar measures}

In this section, we explain how we choose Haar measures in this paper for reductive groups and unipotent groups. Let $\psi\colon F\rightarrow \C^\times$ be a non-trivial additive character of $F$. For a reductive group, Gan and Gross constructed a Haar measure $dg$ depending only on the  group $G$ and the non-trivial additive character $\psi$ \cite[\S8]{GG99}. (In \cite{GG99}, it is denoted by $\mu_G$.) For a unipotent group, it is useful to consider the ``self-dual measures'' $du$ with respect to $\psi$. 
In both cases, we denote by $|X|$ the volume of $X$ for a measurable set $X$.

\subsection{Measures on reductive groups}\label{Haar1}
Let $G$ be a connected reductive group, and let $G'$ be the quasi-split inner form of $G$. Moreover, let $S$ be a maximal $F$-split torus of $G$, let $S'$ be a maximal $F$-split torus of $G'$, let $T'$ be the centralizer of $S'$ in $G'$ (it becomes a torus over $F^s$), and let $\cW(T',G')$ be the Weyl group of $G'$ with respect to $T'$. Put $E' \coloneqq   X^*(T') \otimes \Q$. Then the space $E'$ can be regarded as a graded $\Q[\Gamma]$-module
\[
E' = \oplus_{d \geq 1} E_d'
\]
as follows: consider a $\cW(T',G')$-invariant subalgebra $R = {\Sym^\bullet}(E')^{\cW(T',G')}$ of symmetric algebra ${\Sym^\bullet}(E')$. We denote by $R_+$ the ideal consisting of the elements of positive degrees. Then, there is a $\Q[\Gamma]$-isomorphism $E' \cong R_+/R_+^2$. Then, the grading of $E'$ is the one deduced from the natural grading of $R_+/R_+^2$.

Let $\Psi\colon G' \rightarrow G$ be an inner isomorphism defined over $F^{\rm ur}$. We may assume that the torus $\Psi(S')$ is a maximal $F^{\rm ur}$ split torus containing $S$.
Then the automorphism $\Psi^{-1} \circ {\rm Fr}(\Psi)$ preserves the torus $T'$ and the action agrees with that by a Weyl element $w_G \in W(T', G')^I$. 
We denote by  $\fM$ the motive
\[
 \oplus_{d \geq 1} E_d'(d-1)
\]
of $G$ (see \cite{Gro97}), and by $a(\fM)$ the Artin invariant
\[
 \sum_{d\geq 1} (2d-1)\cdot a(E_d')
\]
of $\fM$ (see \cite{GG99}).
Then, the Haar measure $dg$ is normalized so that the volume of the Iwahori subgroup $\cB$ is given by
\begin{align}\label{iwahori_v}
|\cB| = q^{-\mathfrak{N}-\frac{1}{2}a(\fM)} \cdot \det(1-{\rm Fr}\circ w_G; E'(1)^I).
\end{align}
Here, we put
\[
\mathfrak{N} = \sum_{d\geq1}(d-1)\dim_\Q {E_d'}^I.
\]

Now, consider the case $G=G(W)$ where $W$ is an $n$-dimensional $(-\epsilon)$-Hermitian space over $D$. 
Then, we have the following:

\begin{prop}\label{Iwa_vol}
\begin{enumerate}
\item Suppose that $-\epsilon=1$. Then, we have \label{Iwa_vol1}
\[
|\cB| = (1-q^{-1})^{\lfloor \frac{n}{2}\rfloor}\cdot(1+q^{-1})^{\lceil \frac{n}{2} \rceil} \cdot q^{-n^2}
\]
where $\cB$ is an Iwahori subgroup of $G(W)$.
\item Suppose that $-\epsilon = -1$. Then, we have \label{Iwa_vol2}
\[
|\cB| = \begin{cases}
(1-q^{-2})^{\frac{n}{2}}\cdot q^{-n^2+n} &(n_0=0), \\
(1-q^{-2})^{\frac{n-1}{2}} \cdot q^{-n^2+n-\frac{1}{2}} & (n_0=1, \chi_W \mbox{ is ramified}),\\
(1-q^{-2})^{\frac{n-1}{2}} \cdot (1+q^{-1}) \cdot q^{-n^2+n} & (n_0=1, \chi_W \mbox{ is unramified}),\\
(1-q^{-2})^{\frac{n-2}{2}} \cdot (1+q^{-1})\cdot q^{-n^2+n-\frac{1}{2}} & (n_0=2, \chi_W  \mbox{ is ramified}), \\
(1-q^{-2})^{\frac{n-2}{2}}\cdot (1+q^{-2}) \cdot q^{-n^2+n} & (n_0=2, \chi_W \mbox{ is unramified}), \\
(1-q^{-2})^{\frac{n-3}{2}} \cdot (1+q^{-1} + q^{-2} + q^{-3}) \cdot q^{-n^2+n} & (n_0=3).
\end{cases}
\]
\end{enumerate}
\end{prop}

\begin{proof}
Let $\mathcal{A}$ be a $\cO_F$-scheme so that the fibered product $\mathcal{A}\times_{{\rm Spec} \cO_F}{\rm Spec} F$ is isomorphic to $\Psi(S')$ over $F$.
Then, we have 
\begin{align}\label{formula1}
\det(1- {\rm Fr}\circ w_G; E'(1)^I) &= q^{-\dim_FS'} \# \mathcal{A}(\cO_F/\varpi_F).
\end{align}
In our case, we may assume that the torus $\Psi(S')$ is isomorphic to
\begin{align*}
&{\rm Res}_{L_2/F}(\mathbb{G}_m)^{\frac{n-n_0}{2}}\\
&\times\begin{cases}
1 & (n_0=0, n_0=1 \mbox{ with } \chi_W \mbox{ ramified}),\\
\ker N_{L_2/F} & (n_0=1 \mbox{ with } \chi_W \mbox{ unramified}, n_0=2 \mbox{ with } \chi_W \mbox{ ramified}),\\
\ker N_{L_4/L_2} & (n_0=2 \mbox{ with } \chi_W \mbox{ unramified}), \\
\ker N_{L_4/F} &(n_0=3),
\end{cases}
\end{align*}
where $L_d$ denotes the unramified extension field of $F$ of $[L_d:F] = d$, and $N_{L/K}$ denotes the norm map ${\rm Res}_{L/F}\mathbb{G}_m^\times \rightarrow {\rm Res}_{K/F}\mathbb{G}_m^\times$ associated with a field extension $L/K$.
Hence, by \eqref{formula1}, we have 
\begin{align*}
&\det(1-{\rm Fr}\circ w_G; E'(1)^I) = (1-q^{-2})^{\frac{n-n_0}{2}}  \\
&\times 
\begin{cases}
1 &(n_0=0, n_0=1 \mbox{ with $\chi_W$ ramified}), \\
(1+q^{-1})  & (n_0=1, \mbox{ with $\chi_W$ unramified}, n_0 = 1 \mbox{ with $\chi_W$ ramified}),\\
(1+q^{-2}) & (n_0=2,  \mbox{ with $\chi_W$ unramified}),\\
(1+q^{-1} + q^{-2}+ q^{-3}) & (n_0=3).
\end{cases}
\end{align*}
We define a grading and a $\Gamma$-action on the polynomial ring $\Q[X, Y]$ by
\begin{align*}
&\deg X^k = k, \ \deg Y^l = nl \ (k,l = 0,1, \ldots), \mbox{ and}\\
&\sigma \cdot f(X,Y) = f(X, \eta_W(\sigma)Y) \mbox{ for } f(X,Y) \in \Q[X,Y], \sigma\in \Gamma.
\end{align*}
Here, $\eta_W$ is a character on $\Gamma$ associated with $\chi_W$ via the local class field theory. Then we have that  $E'$ is isomorphic to
\[
\begin{cases}
 \Q X^2+ \Q X^4 + \cdots + \Q X^{2n} & (-\epsilon = 1), \\
 \Q X^2+ \Q X^4 + \cdots + \Q X^{2n-2} + \Q Y & (-\epsilon = -1)
 \end{cases}
\]
as a graded $\Q[\Gamma]$-module. 
Hence, we have
\[
\mathfrak{N} = \begin{cases}n^2 & (-\epsilon = 1), \\ n^2-n & (-\epsilon = -1 \mbox{ with $\chi_W$ unramified}),\\ n^2-2n+1 & (-\epsilon = -1 \mbox{ with $\chi_W$ ramified}), \end{cases}
\]
and
\[
a(\fM) = \begin{cases} 0 & (\chi_W \mbox{ is unramified}), \\ 2n-1 & (\chi_W \mbox{ is ramified}). \end{cases}
\]
By computing the right-hand side of \eqref{iwahori_v}, we have the claim.
\end{proof}

If $G(W)$ is anisotropic, then $\cB = \ker \kappa_{W}$ (see \S\ref{iwahori}). Hence, its total volume is given by the following corollary: 

\begin{cor}\label{vol of GW}
Suppose that $W$ is anisotropic.
\begin{enumerate}
\item If $-\epsilon = 1$ and $n=1$, then we have $|G(W)| = q^{-1}(1+q^{-1})$.
\item If $-\epsilon = -1$, then we have 
\[
|G(W)| = \begin{cases}
1+q^{-1} & (n=1 \mbox{ with $\chi_W$ unramified}), \\
2q^{-\frac{1}{2}} & (n=1 \mbox{ with $\chi_W$ ramified}), \\
2q^{-2}\cdot(1+q^{-2}) & (n=2 \mbox{ with $\chi_W$ nontrivial and unramified}), \\
2q^{-\frac{5}{2}}\cdot(1+q^{-1}) & (n=2 \mbox{ with $\chi_W$ ramified}), \\
2q^{-6}\cdot (1+q^{-1})(1+q^{-2}) & (n=3).
\end{cases}
\]
\end{enumerate}
\end{cor}

\begin{proof}
Since the Kottwitz map $\kappa_W$ is surjective,
\begin{align*}
[G(W): \cB] &=\#(X^*(Z(\widehat{G})^I)^{{\rm Fr}}) \\
&=\begin{cases}
1 & (n=1 \mbox{ with $\chi_W$ unramified}), \\
2 & (\mbox{ otherwise})
\end{cases}
\end{align*}
where $I$ is the inertia group of $F$,  and ${\rm Fr}$ is a Frobenius element of $F$. 
Hence we have the claim.
\end{proof}

\subsection{Measures on unipotent groups}\label{Haar2}

Take a basis $\underline{e}$ and regard $G(W)$ as a subgroup of $\GL_n(D)$ as in \S\ref{basis for WV}. Let
\[
\ff: 0=X_0 \subset X_1 \subset \cdots \subset X_{k-1} \subset X_k = X 
\]
be a flag consisting of totally isotropic subspaces. We put $r_i = \dim_D X_i/X_{i-1}$ for $i=1, \ldots, k$. Moreover, we put
\[
\fu_{r'} = \{z \in {\rm M}_{r'}(D)\mid {}^t\!z^* -\epsilon z = 0 \}
\]
for a positive integer $r'$. We denote by $P$ the parabolic subgroup of all $p \in G(W)$ satisfying $X_i\cdot p \subset X_i$ for $i=0, \ldots, k$, and by $U(P)$ the unipotent radical of $P$. Moreover, we denote by $U_i(P)$ the subgroup 
\[
 \{ u \in U(P) \mid X\cdot(u-1) \subset X_i \}
\]
for $i=1, \ldots, k$. Then, for $i=1, \ldots, k$, we have the exact sequence
\begin{align}\label{mes unip1}
1 \rightarrow U_{i-1}(P) \rightarrow U_{i}(P) \rightarrow \prod_{j=(i+2)/2}^{i} {\rm M}_{r_j, r_{i+1-j}}(D)
\rightarrow 0
\end{align}
if $i$ is even, and the exact sequence
\begin{align}\label{mes unip2}
1 \rightarrow U_{i-1}(P) \rightarrow U_{i}(P) 
\rightarrow \fu_{r_{(i+1)/2}}\times\prod_{j=(i+3)/2}^{i} {\rm M}_{r_{i+1-j},r_j}(D)  
\rightarrow 0
\end{align}
if $i$ is odd. Here, the first maps are the inclusions and the second maps are given by
\[
u=\begin{pmatrix} 1 &  &  &  &  &  \\
                     0 & 1 &  &  &  &  \\
                     z_1 & 0 & 1 & &  &  \\
                     * & z_2 & \ddots&\ddots & &  \\
                     \vdots& \ddots & \ddots&0&1&  \\
                     *&\cdots &* & z_i & 0 & 1 
\end{pmatrix} \mapsto (z_{\lceil (i+1)/2 \rceil}J_{\lceil (i+1)/2 \rceil}, \ldots, z_iJ_i),.
\]
for $u \in U_{i}(P)$. 
We define a measure $dz$ on $\fu_{r'}$ to be the self-dual Haar measure with respect to a pairing
\begin{align}\label{pairingDD}
\fu_r \times \fu_r \rightarrow \C\colon (z,z')\mapsto \psi(T_D(z\cdot {}^t\!{z'}^*)),
\end{align}
and we define a measure $dx$ on ${\rm M}_{r',r''}(D)$ to be the self-dual Haar measure with respect to a pairing
\begin{align}\label{pairingD}
{\rm M}_{r',r''}(D) \times {\rm M}_{r',r''}(D) \rightarrow \C^\times
\colon (x , x') \mapsto \psi(T_D(x \cdot {}^t\!{x'}^*)).
\end{align}
Then, the Haar measure $du$ on $U_i(P)$ is defined inductively by the exact sequences \eqref{mes unip1} and \eqref{mes unip2} for $i=1, \ldots, k$. 

In the rest of this section, we compute the volumes $|\fu_{r'} \cap {\rm M}_{r'}(\cO_D)|$ and $|{\rm M}_{r', r''}'(\cO_D)|$ with respect to the self-dual measures above. To compute them, we observe latices of $\fu_r$ and ${\rm M}_{r',r''}(D)$ with $r = r' = r'' = 1$.

\begin{lem}\label{pairingDDlem}
\begin{enumerate}
\item \label{pairingD1}
Suppose that $r' = r'' = 1$. Then, the dual lattice $\cO_D^*$ of $\cO_D$ with respect to the pairing \eqref{pairingD} is given by $\varpi_D^{-1}\cO_D$.
\item \label{pairingD2}
Suppose that $\epsilon = 1$ and $r=1$. Then, the dual lattice of $\cO_D \cap \fu_1$ with respect to the pairing \eqref{pairingDD} is given by $\frac{1}{2}\delta\cO_F + \varpi_D^{-1}\cO_{F(\delta)}$.
\item \label{pairingD3}
Suppose that $\epsilon = -1$ and $r=1$. Then, the dual lattice of $\cO_D \cap \fu_1$ with respect to the pairing \eqref{pairingDD} is given by $\frac{1}{2}\cO_F$.
\end{enumerate}
\end{lem}  

\begin{proof}
Since the order of $\psi$ is zero, there exists $a \in \cO_F^\times$ such that $\psi(\varpi_F^{-1}a) \not= 1$. The assertion \eqref{pairingD3} is well-known, thus we only prove \eqref{pairingD1} and \eqref{pairingD2}. We admit the existence of an element $b \in \cO_{F(\delta)}^\times$ satisfying $b + b^* = a$ at once. We take two elements $\delta, \varpi_D$ as in Lemma \ref{quatb}.
If $x \in D^\times$ satisfies $\ord_D(x) < -1$, then $x^{-1}\varpi_F^{-1}b \in \cO_D$, and we have
\[
\psi(T_D(x\cdot x^{-1}\varpi_F^{-1}b)) = \psi(\varpi_F^{-1}a) \not=1.
\]
Thus we have that $\cO_F^*$ is contained in $\varpi_D^{-1}\cO_D$. On the other hand, $\psi(T_D(\varpi_D^{-1}\cO_D)) = 1$ since $T_D(\varpi_D^{-1}\cO_D) \subset \cO_F$. Hence we have \eqref{pairingD1}. 
Suppose that $\epsilon = 1$ and $r=1$. An element $x$ of $\fu_1$ can be written in the form $x = \delta \cdot x_1 + \varpi_F \cdot x_2$ where $x_1 \in F$ and $x_2 \in F(\delta)$. If $x_1 \not\in \frac{1}{2}\cO_F$, then $\delta^{-1}(2x_1)^{-1}\varpi_F^{-1}a \in \cO_D$, and
\[
\psi(T_D(\delta x\cdot \delta^{-1}(2x_1)^{-1}\varpi_F^{-1}a)) = \psi(\varpi_F^{-1}a) \not=1.
\]
If $x_2 \not\in \cO_{F(\delta)}$, then $x_2^{-1}b \in \cO_{D}$, and 
\[
\psi(T_D(\varpi_D^{-1}x_2\cdot x_2^{-1}b)) = \psi(\varpi_F^{-1}a) \not= 1.
\]
Thus we have that the dual lattice of $\cO_D \cap \fu_1$ is contained in $\frac{1}{2}\delta\cO_F + \varpi_D^{-1}\cO_{F(\delta)}$. On the other hand, the subset $(\frac{1}{2}\delta\cO_F + \varpi_D^{-1}\cO_{F(\delta)})\cdot \cO_D$ is contained in the subset $\frac{1}{2}\cO_F + \varpi_D^{-1}\cO_D$ on which $\psi\circ T_D$ vanishes. Hence we have \eqref{pairingD2}. 

It remains to show that there exists an element $b \in \cO_{F(\delta)}^\times$ satisfying $b + b^* = a$. Put 
\begin{align*}
\mathcal{X} &= \{T^2 - uT + v \mid u, v \in (\cO_F/\varpi_F)^\times\}, \mbox{ and} \\
\mathcal{Y} &= \{(T-x)(T-y) \mid x,y \in (\cO_F/\varpi_F)^\times, x+y \not= 0\}.
\end{align*}
Then we have $\mathcal{X} \supset \mathcal{Y}$ and
\[
\#\mathcal{X} = (q-1)^2 > \frac{1}{2}q(q-1)-(q-1) = \# \mathcal{Y}.
\]
This enequation implies that $\mathcal{X}$ possesses at least one irreducible polynomial $h(T)$. Take $c \in \cO_{F(\delta)}^\times$ so that its image $\overline{c} \in \cO_{F(\delta)}/\varpi_F$ satisfies $h(\overline{c}) = 0$. Then, by the definition of $\mathcal{X}$, we have $c + c^* \in \cO_F^\times$. Thus, putting $b = c(c + c^*)^{-1} a$, we have $b+b^* = a$. This completes the proof of lemma \ref{pairingDDlem}.
\end{proof}

Let $r, r'$ and $r''$ be arbitrary positive integers again. Then, by Lemma \ref{pairingDDlem}, we have the following:

\begin{cor}
\begin{enumerate}
\item We have
\[
|\fu_{r}\cap {\rm M}_{r}(\cO)| = \begin{cases} |2|^{\frac{1}{4}r(r + 1)}q^{-\frac{1}{2}r(r + 1)} &(\epsilon = 1),  \\ |2|^{\frac{1}{4}r(r + 1)}q^{-\frac{1}{2}r(r - 1)} & (\epsilon = -1). \end{cases}
\]
\item We have $|{\rm M}_{r',r''}(\cO_D)| = q^{-r'r''}$.
\end{enumerate}
\end{cor}

%&&&&&&&&&&&&&&%%%%%%%%%%%%%%%     S7     %%%%%%%%%%%%%%%%%%%%%%%%%%%%%%%%%%

\section{
            Doubling method and local $\gamma$-factors
            }\label{doubling and gamma}

In this section, we explain the doubling method, and we recall the analytic definition of the local standard $\gamma$-factor (\S\ref{local gamma factor}). The doubling method also appears in the formulation of the local Siegel-Weil formula (\S\ref{SS local SW formula} below) and the local analogue of the Rallis inner product formula (\S\ref{loc Rallis inn prod formula} below). 
Let $W$ be a $(-\epsilon)$-Hermitian space over $D$. In this section, we also define the local zeta value $\alpha_1(W)$, which depends on $W$ and its basis $\ue$. In \S\ref{loc zeta val}, we compute $\alpha_1(W)$ for a $(-\epsilon)$-Hermitian space and for a basis $\ue$ for $W$ under some assumptions. As explained in the introduction, this computation of the constant $\alpha_1(W)$ will play an important role in the computation of the constant in the local Siegel-Weil formula (\S\ref{SS local SW formula}).

\subsection{Doubling method} \label{doubling}

Let $(W^\Box, \langle \ , \ \rangle^\Box)$ be the pair where $W^\Box = W \oplus W$ and $\langle \ , \ \rangle^\Box$ is the map $W^\Box \times W^\Box \rightarrow D$ defined by
\[
\langle (x_1, x_2), (y_1, y_2) \rangle^\Box = \langle x_1, y_1\rangle - \langle x_2, y_2\rangle
\]
for $x_1, x_2, y_1, y_2 \in W$. Let $G(W^\Box)$ be the isometric group of $W^\Box$. Then, the natural action
\[
G(W)\times G(W) \curvearrowright W\oplus W\colon
 (x_1, x_2)\cdot (g_1, g_2)  = (x_1\cdot g_1, x_2\cdot g_2)
\]
induces an embedding $\iota\colon G(W)\times G(W) \rightarrow G(W^\Box)$. Consider maximal totally isotropic subspaces
\begin{align*}
W^\tru &= \{ (x,x) \in W^\Box \mid x \in W\}, \mbox{ and }\\
W^\trd &= \{ (x,-x) \in W^\Box \mid x \in W \}. 
\end{align*}
Then we have a polar decomposition $W^\Box= W^\tru \oplus W^\trd$. 
We denote by $P(W^\tru)$ the maximal parabolic subgroup of $G(W^\Box)$ which preserves $W^\tru$. Then, a Levi subgroup of $P(W^\tru)$ is isomorphic to $\GL(W^\tru)$. We denote by $\Delta$ the character of $P(W^\tru)$ given by
\[
\Delta(x) = N_{W^\tru}(x)^{-1}.
\]
Here $N_{W^\tru}(x)$ is the reduced norm of the image of $x$ in ${\rm End}_D(W^\tru)$.
Let $\omega\colon F^\times \rightarrow \C^\times$ be a character. For $s \in \C$, put $\omega_s= \omega\cdot |-|^s$. 
Let $\underline{e}$ be a basis for $W$. Then we define a basis ${\ue'}^\Box = (e_1', \ldots, e_{2n}')$ for $W^\Box$ by
\[
e_i' = (e_i, e_i), \ 
e_{n+i}' = \sum_{k=1}^n a_{jk} (e_i, -e_i)
\]
for $i=1, \dots, n$, where $(a_{jk})_{j,k} = R(\ue)^{-1}$. Then we have
\[
(\langle e_i', e_j'\rangle)_{i,j} = \begin{pmatrix} 0 & 2\cdot I_n \\ -2\epsilon\cdot  I_n & 0 \end{pmatrix}.
\]
We choose a maximal compact subgroup $K({\ue'}^\Box)$ of $G(W^\Box)$ which preserves the lattice
\[
\cO_{W^\Box} = \sum_{i=1}^{2n}\cO_D e_i'
\]
of $W^\Box$. Then, we have $P(W^\tru)K({\ue'}^\Box) = G(W^\Box)$. Denote by $I(s,\omega)$ the degenerate principal series representation
\[
\Ind_{P(W^\tru)}^{G(W^\Box)} (\omega_s\circ\Delta)
\]
consisting of the smooth right $K({\ue'}^\Box)$-finite functions $f\colon G(W^\Box) \rightarrow \C$ satisfying
\[
f(pg) = \delta_{P(W^\tru)}^{\frac{1}{2}}(p)\cdot \omega_s(\Delta(p))\cdot f(g)
\]
for $p \in P(W^\tru)$ and $g \in G(W^\Box)$, where $\delta_{P(W^\tru)}$ is the modular function of $P(W^\tru)$. We may extend $|\Delta|$ to a right $K({\ue'}^\Box)$-invariant function on $G(W^\Box)$ uniquely. 
We denote by $U(W^\tru)$ the unipotent radicals of $P(W^\tru)$. For $f \in I(0,\omega)$, put $f_s = f\cdot |\Delta|^s\in I(s, \omega)$. Then, we define an intertwining operator $M(s,\omega)\colon I(s,\omega)\rightarrow I(-s,\omega^{-1})$ by 
\[
[M(s,\omega)f_s](g) = \int_{U(W^\tru)}f_s(\tau ug) \: du
\]
where $\tau$ is the Weyl element of $G(W^\Box)$ given by
\[
\begin{cases}
\tau(e_i') = e_{n+i}' & (i=1, \ldots, n), \\
\tau(e_i') = -\epsilon e_{i-n}' & (i=n+1, \ldots, 2n).
\end{cases}
\]
This integral converges absolutely for $\Re s > 0$ and admits a meromorphic continuation to $\C$. Let $\pi$ be a representation of $G(W)$ of finite length. For a matrix coefficient $\xi$ of $\pi$, and for $f \in I(0,\omega)$, we define the doubling zeta integral by
\[
Z^W(f_s, \xi) = \int_{G(W)} f_s(\iota(g,1))\xi(g) \: dg.
\]
Then the zeta integral satisfies the following properties, which are stated in \cite[Theorem 4.1]{Yam14}. This gives a generalization of \cite[Theorem 3]{LR05}.

\begin{prop}\label{zetabasic}
\begin{enumerate}
\item The integral $Z^W(f_s, \xi)$ converges absolutely for $\Re s \geq n - \epsilon$ and has an analytic continuation to a rational function of $q^{-s}$.
\item There is a meromorphic function $\Gamma^W(s,\pi, \omega)$ such that
\[
Z^W(M(s,\omega)f_s,\xi) = \Gamma^W(s,\pi,\omega)Z^W(f_s, \xi)
\]
for all matrix coefficient $\xi$ of $\pi$ and $f_s \in I(s, \omega)$.
\end{enumerate}
\end{prop}

%%%%%%%%%%

\subsection{Local $\gamma$-factors}\label{local gamma factor}

Fix a non-trivial additive character $\psi\colon F \rightarrow \C^\times$ and $A \in {\rm End}_D(W^\Box)$ so that $\rank A = n$ and $1 + A \in U(W^\tru)$. We use the Haar measures on $U(W^\trd)$ and $U(W^\tru)$ by identifying them with $\fu_n$ by the basis ${\underline{e'}}^\Box$ (see \S\ref{Haar2}). We define 
\[
\psi_A\colon U(W^\trd) \rightarrow \C^\times\colon u \mapsto \psi({\rm T}_{W^\Box}(uA))
\]
where ${\rm T}_{W^\Box}$ denotes the reduced trace of ${\rm End}_D(V^\Box)$. Moreover, we define the character $\chi_A$ of $F^\times$ by $\chi_A(x) = (x, \mathfrak{d}(A))$ for $x \in F^\times$ where $\mathfrak{d}(A)$ denotes the element of $F^\times/{F^\times}^2$ defined as in \cite[\S5.1]{Kak20}. For $f \in I(0,\omega)$ we define
\[
l_{\psi_A}(f_s) = \int_{U(W^\trd)}f_s(u) \psi_A(u) \: du.
\]
Then, this integral defining $l_{\psi_A}$ converges for $\Re s \gg 0$ and admits a holomorphic continuation to $\C$ (\cite[\S3.2]{Kar79}). Let $A_0 \in \GL_n(D)$ the matrix representation of the linear map $A\colon W^\trd \rightarrow W^\tru$ with respect to the bases $e_{n+1}', \ldots, e_{2n}'$ for $W^\trd$ and $e_1', \ldots, e_n'$ for $W^\tru$.
We denote by $e(G(W))$ the Kottwitz sign of $G(W)$, which is given by
\[
e(G(W)) = 
\begin{cases}
(-1)^{\frac{1}{2}n(n+1)} & (-\epsilon = 1), \\
(-1)^{\frac{1}{2}n(n-1)} & (-\epsilon = -1).
\end{cases}
\]
Then, as in \cite[Proposition 4.2]{Kak20}, we have the following:

\begin{prop}\label{c()}
We have
\[
l_{\psi_A}\circ M(s,\omega) = c(s,\omega,A,\psi)\cdot l_{\psi_A},
\]
where $c(s,\omega,A,\psi)$ is the meromorphic function of $s$ given by
\begin{align*}
c(s,\omega,A,\psi) &=e(G(W))\cdot \omega_s(N(A_0))^{-1}\cdot |2|^{-2ns+n(n-\frac{1}{2})}\cdot\omega^{-1}(4)\cdot\gamma(s-n+\frac{1}{2}, \omega,\psi)^{-1}\\
&\times \prod_{i=0}^{n-1}\gamma(2s-2i, \omega^2, \psi)^{-1}\cdot\gamma(s+\frac{1}{2}, \omega\chi_{A_0}, \psi)\cdot\epsilon(\frac{1}{2}, \chi_{A_0}, \psi)^{-1}
\end{align*}
in the case $-\epsilon = 1$, and 
\begin{align*}
c(s,\omega,A,\psi) &=e(G(W))\cdot \omega_s(N(A_0))^{-1}\cdot |2|^{-2ns+n(n-\frac{1}{2})}\cdot\omega^{-1}(4)\cdot\prod_{i=0}^{n-1}\gamma(2s-2i, \omega^2, \psi)^{-1}
\end{align*}
in the case $-\epsilon = -1$.
\end{prop}

\begin{rem}
These formulas differ from those in \cite[Proposition 4.2]{Kak20}. 
This is caused by a typo where $\omega_{n\pm\frac{1}{2}}(N(R))$ should be replaced by $|N(R)|^{-(n\pm\frac{1}{2})}$ in \cite[Proposition 4.2]{Kak20}. 
\end{rem}

Now we define the doubling $\gamma$-factor as in \cite{Kak20}. Note that the above error has no effect on the definition in \cite{Kak20}.  

\begin{df}
Let $\pi$ be an irreducible representation of $G(W)$, let $\omega$ be a character of $F^\times$, and let $\psi$ be a non-trivial character of $F$. Then we define the $\gamma$-factor by
\[
\gamma^W(s+\frac{1}{2}, \pi\times\omega,\psi) = c(s,\omega,A,\psi)^{-1}\cdot\Gamma^W(s,\pi,\omega)\cdot c_\pi(-1)\cdot R(s,\omega,A,\psi).
\]
where $c_\pi$ is the central character of $\pi$, and 
\[
R(s,\omega,A,\psi) = 
\begin{cases}
\omega_s(N(R(\underline{e})A_0)^{-1}\gamma(s+\frac{1}{2},\omega\chi_A, \psi)
\epsilon(\frac{1}{2}, \chi_A,\psi)^{-1} & (-\epsilon = 1), \\
\omega_s(N(R(\underline{e})A_0)^{-1}\epsilon(\frac{1}{2}, \chi_W, \psi) 
& (-\epsilon=-1).
\end{cases}
\]
\end{df}

The doubling $\gamma$-factor $\gamma^W(s+\frac{1}{2}, \pi\boxtimes\omega,\psi)$ is expected to coincide with the standard $\gamma$-factor $\gamma(s+\frac{1}{2}, \pi\boxtimes\omega, \std, \psi)$ where $\std$ is the standard embedding of ${}^L\!(G(W)\times \GL_1)$. 
Another notable property is the commutativity with parabolic inductions, which is useful in the computation. For example, the doubling $\gamma$-factor of the trivial representation is given by the following lemma, which we use in the computation of the doubling zeta integral (\S\ref{loc zeta val} and \S\ref{App} below). 

\begin{lem}\label{triv gamma}
Denote by $1_{G(W)}$ the trivial representation of $G(W)$. Then we have
\[
\gamma^W(s+\frac{1}{2}, 1_{G(W)}\times 1_{F^\times}, \psi)
=\begin{cases}
\prod_{i=-n}^n\gamma_F(s+\frac{1}{2} + i, 1, \psi) & (-\epsilon = 1), \\
\gamma_F(s+\frac{1}{2}, \chi_W, \psi) \prod_{i=-n+1}^{n-1}\gamma_F(s+\frac{1}{2} + i, 1, \psi)
& (-\epsilon = -1).
\end{cases}
\]\label{triv gamma 1}
\end{lem}

\begin{proof}
\cite[Proposition 7.1]{Kak20}.
\end{proof}

%%%%%%
\subsection{Local zeta values}\label{loc zeta val}

We use the same setting and notation of \S\ref{doubling}.
Let $f_s^\circ \in I(s,1_{F^\times})$ be the unique $K({\ue'}^\Box)$-fixed section with $f_s^\circ(1) = 1$, and let $\xi^\circ$ be the matrix coefficient of the trivial representation of $G(W)$ with $\xi^\circ(1) = 1$. Put $\rho = n -\frac{\epsilon}{2}$. Then, we define 
\[
\alpha_1(W)\coloneqq   Z^W(f_\rho^\circ, \xi^\circ),
\]
which is the first constant we are interested in. The integral defining $\alpha_1(W)$ converges absolutely by Proposition \ref{zetabasic}. The purpose of this subsection is to obtain a formula of $\alpha_1(W)$ in the case where either $R(\ue) \in \GL_n(\cO_D)$ or $W$ is anisotropic. The general formula of $\alpha_1(W)$ will be obtained in \S\ref{det alpha}.

\begin{prop}\label{alpha1 min}
\begin{enumerate}
\item In the case $-\epsilon=1$ and $R(\ue) \in \GL_n(\cO_D)$, we have
\[
\alpha_1(W) = |2|^{n(2n+1)}\cdot q^{-n_0^2-(2n_0+1)r-2r^2}\cdot\prod_{i=1}^n(1+q^{-(2i-1)}).
\]
\label{alpha1 1}
\item In the case $-\epsilon = -1$ and $R(\ue)\in \GL_n(\cO_D)$, we have
\[
\alpha_1(W) = |2|^{n(2n-1)}\cdot q^{-2rn_0-2r^2+r}\cdot\prod_{i=1}^n(1+q^{-(2i-1)}).
\]
\label{alpha1 2}
\item In the case $-\epsilon=-1$ and $W$ is anisotropic, we have
\[
\alpha_1(W) = |N(R(\ue))|^{-n + \frac{1}{2}}\times
\begin{cases} |2|_F\cdot (1+q^{-1}) & (n=1), \\
|2|_F^6\cdot q^{-1}\cdot (1+q^{-1})(1+q^{-3}) & (n=2),  \\
|2|_F^{15}\cdot q^{-3}\cdot (1+q^{-1})(1+q^{-3})(1+q^{-5}) & (n=3).
\end{cases}
\]\label{alpha1 min2}
\end{enumerate}
\end{prop}

Unless $q$ is a power of $2$ , the assertions \eqref{alpha1 1} and \eqref{alpha1 2} are conclusions of \cite[Proposition 8.3]{Kak20} and the volume formula of a maximal compact subgroup containing $\cB$ which can be obtained by a generalization of the Bruhat decomposition (\cite[Appendix, Proposition 8]{PR08}). However, to contain the case $2|q$, we prove them in another way.
Before proving this lemma, we observe the following two important lemmas:

\begin{lem}\label{dim = 1}
\[
\dim_\C \Hom_{G(W)\times G(W)}(I(\rho, 1_{F^\times}), \C) = 1.
\]
\end{lem}

\begin{proof}
By Proposition \ref{zetabasic}, the integral
\[
\int_{G(W)} f((g,1)) \: dg
\]
converges absolutely for $f \in I(\rho, 1_{F^\times})$. We denote it by $Z(f)$, and we obtain a non-zero map $Z \in \Hom_{G(W)\times G(W)}(I(\rho, 1_{F^\times}), \C)$. To prove the lemma, 
it suffices to show that $\ker Z$ is spanned by the set
\[
\{h - R(g)h \mid h \in I(\rho, 1_{F^\times}), g \in G(W) \times G(W)\}.
\]
Here, we denote by $R(g)$ the right translation by $g$. 
Let $f \in \ker Z$. Take a compact open subgroup $K'$ of $K({\ue'}^\Box)$, complex numbers $a_i \in \C$ and elements $g_i \in G(W)\times G(W)$ for $i=1, \ldots, t$ so that 
\[
f = \sum_{i=1}^t a_i R(g_i)\mathfrak{c}
\]
where $\mathfrak{c} \in I(\rho, 1_{F^\times})$ is the section defined by
\[
\mathfrak{c}(g) \coloneqq   
\begin{cases}
\delta_{P(W^\tru)}(p) & g=pk' \ (p \in P(W^\tru), k' \in K'), \\ 
0 & g \not\in P(W^\tru)K'.
\end{cases}
\] 
Then, we have 
\[
a_1 + \cdots + a_t = \frac{Z(f)}{Z(\mathfrak{c})} = 0
\]
and we have
\[
\sum_{i=1}^{t-1} b_i(R(g_i)\fc - R(g_{i+1})\fc) = f
\]
where $b_i \coloneqq   a_1 + \cdots + a_i$ for $i=1, \ldots, t-1$. Hence we have the lemma. 
\end{proof}

\begin{lem}\label{ratio}
For $f \in I(\rho, 1_{F^\times})$, we have
\[
\int_{G(W)} f((g,1)) \: dg = m^\circ(\rho)^{-1}\cdot\alpha_1(W)\cdot \int_{U(W^\tru)} f(\tau u) \: du
\]
where 
\[
m^\circ(s) = 
\begin{cases}
|2|_F^{n(n-\frac{1}{2})}q^{-\frac{1}{2}n(n+1)}\frac{\zeta_F(s-n+\frac{1}{2})}{\zeta_F(s+n+\frac{1}{2})}
\prod_{i=0}^{n-1}\frac{\zeta_F(2s-2i)}{\zeta_F(2s+2n-4i-3)} & (-\epsilon=1), \\
|2|_F^{n(n-\frac{1}{2})}q^{-\frac{1}{2}n(n-1)}\prod_{i=0}^{n-1}\frac{\zeta_F(2s-2i)}{\zeta_F(2s+2n-4i-1)}
& (-\epsilon = -1).
\end{cases}
\]
\end{lem}

\begin{proof}
Define a map $\mathfrak{A}\colon \cS(G(W^\Box)) \rightarrow I(\rho, 1_{F^\times})$ by 
\[
[\mathfrak{A}\varphi](g) = \int_{P(W^\tru)} \delta_{P(W^\tru)}(p)^{-1} \varphi(pg) \: dp.
\]
Then $\mathfrak{A}$ is surjective. Moreover, we have
\begin{align*}
\int_{U(W^\tru)} [\mathfrak{A}\varphi](\tau u) \: du
& = \int_{U(W^\tru)}\int_{M(W^\tru)}\int_{U(W^\tru)} \delta_{P(W^\tru)}^{-1}(m) \varphi(xm\tau y) \: dydmdx \\
& = \gamma(G(W^\Box)/P(W^\tru))\int_{G(W^\Box)} \varphi(g) \: dg.
\end{align*}
Here, $\gamma(G(W^\Box)/P(W^\tru))$ is the constant defined by
\[
\gamma(G(W^\Box)/P(W^\tru)) = \int_{U(W^\tru)} f^\circ(\tau u) \: du
\]
where $f^\circ \in I(\rho, 1_{F^\times})$ is an unique $K({\ue'}^\Box)$-invariant section with $f^\circ(1) = 1$.
Hence we conclude that the map
\[
I(\rho, 1_{F^\times}) \rightarrow \C\colon f \mapsto \int_{U(W^\tru)} f(\tau u) \: du
\]
is $G(W^\Box)$-invariant, in particular, it is $G(W) \times G(W)$-invariant.
Hence, by Lemma \ref{dim = 1}, we conclude that there is a constant $\alpha' \in \C$ such that 
\[
\int_{G(W)} f((g,1)) \: dg = \alpha' \int_{U(W^\tru)} f(\tau u) \: du
\]
for all $f \in I(\rho, 1_{F^\times})$. To determine the constant $\alpha'$, we use $f^\circ$ as a test function. By Gindikin-Karperevich formula (\cite[Theorem 3.1]{Cas80}) or Shimura's computation (\cite[Proposition 3.5]{Shi99}), we have 
\[
\int_{U(W^\tru)} f^\circ(\tau u) \: du = m^\circ(\rho).
\]
Moreover, comparing this to Lemma \ref{zeta formula2}, we have the claim. 
\end{proof}

Now we prove Proposition \ref{alpha1 min}. 
As a consequence of Lemma \ref{ratio}, we use another section $f(s, 1_{\varpi_F\cO_{\fu}}, -) \in I(s, 1)$ to compute the ratio $\alpha_1(W)m^\circ(\rho)^{-1}$. Here, we denote the set $\fu \cap {\rm M}_n(\cO)$ by $\cO_{\fu}$, and we define a section $f(s, \Phi, -) \in I(s, \omega)$ by
\[
f(s, \Phi, g)\coloneqq   \begin{cases} 0 & g \notin P(W^\tru)\tau U(W^\tru), \\
  \omega_{s+\rho}(\Delta(p))\Phi(X) & g = p\tau \begin{pmatrix}1 & 0 \\ X & 1\end{pmatrix} \ (p \in P(W^\tru), \ X \in \fu) \end{cases}
\]
for a character $\omega$ of $F^\times$ and $\Phi \in \cS(\fu)$.
Let $g \in G(W)$ with $\iota(g) \in P(W^\tru)\tau U(W^\tru)$. Then, 
\[
\begin{pmatrix}\frac{g+1}{2} & \frac{g-1}{2}R(\ue)^{-1} \\ R(\ue)\frac{g-1}{2} & R(\ue)\frac{g+1}{2}R(\ue)^{-1}\end{pmatrix} = \begin{pmatrix} a & 0 \\ b & {}^t\!{a^*}^{-1} \end{pmatrix} \tau \begin{pmatrix} 1 & 0 \\ X & 1 \end{pmatrix}
\]
for some $a \in \GL_n(D), b \in {\rm M}_n(D)$ and $X \in \fu$.
If $X \in \varpi_F\cO_\fu$, then $a, g$ are given by
\[
a = (X-R(\ue))^{-1}, \ g=a(X+R(\ue)) = 2aX - 1,
\]
and thus $a \in \GL_n(\cO_D)$ and $g \in -K_{2\varpi_F}^+$. Here we denote the set 
\[
\{ g \in G(W)\cap \GL_n(\cO_D) \mid g-1 \in 2\varpi_F{\rm M}_n(\cO_D) \}
\] 
by $K_{2\varpi_F}^+$. Conversely, if $g \in -K_{2\varpi_F}^+$, then $a, X$ are given by
\[
a I_n = \frac{g-1}{2}R(\ue)^{-1}, \ aX = \frac{g+1}{2},
\]
and thus $a \in \GL_n(\cO_D)$ and $X \in \varpi_F\cO_\fu$.  
Summarizing the above discussions, we have  
\[
f(s, 1_{\varpi_F\fu}, \iota(g,1)) = 1_{-K_{2\varpi_F}^+}(g)
\]
for $g \in G(W)$. Put
\[
m'(s) \coloneqq   \int_{U(W^\tru)}f(s, 1_{2\varpi_F{\cO_\fu}}, \tau u) \: du.
\]
Then, we have
\begin{align*}
\frac{\alpha_1(W)}{m^\circ(\rho)} 
&= \frac{Z(f(\rho,1_{2\varpi_F\cO_{\fu}},-))}{m'(\rho)} \\
&= \frac{|K_{2\varpi_F}^+|}{|\varpi_F \cO_{\fu}|} \\
& =  |2|_F^{2n\rho - n(n-\frac{1}{2})}q^{\frac{1}{2}n(n-\epsilon)}q^{n(2n-\epsilon)}|K_{\varpi_F}^+|.
\end{align*}
Since
\[
\log_q[\cB^+:K_{\varpi_F}^+] 
= 6(n_0r+r(r-1))+5r+n_0-(2r+n_0)\epsilon
\]
and 
\[
\log_q|\cB^+| = \begin{cases}-n^2-n & (-\epsilon = 1), \\
                               -n^2 & (-\epsilon = -1), \end{cases}
\]
we have
\begin{align*}
&{\rm log}_q (q^{\frac{1}{2}n(n-\epsilon)}q^{n(2n-\epsilon)}|K_{\varpi_F}^+|)\\
&= \frac{1}{2}n(n-\epsilon) + n(2n-\epsilon) - 6(n_0r+r(r-1)) - 5r - n_0 + (2r + n_0)\epsilon \\
& \ \ \ -\begin{cases}-n^2 - n & (-\epsilon = 1), \\ -n^2 &(-\epsilon = -1)\end{cases} \\
&=\frac{1}{2}n(n-\epsilon) - \begin{cases}
2r^2+(2n_0+1)r + n_0^2 & (-\epsilon = 1), \\
2r^2+2n_0r-r & (-\epsilon = -1).
\end{cases}\end{align*}
Hence we have
\begin{align*}
\alpha_1(W) &= m^\circ(\rho) \cdot \frac{\alpha_1(W)}{m^\circ(\rho)} \\
& = \begin{cases}
|2|^{n(2n+1)}\cdot q^{-n_0^2-(2n_0+1)r-2r^2}\cdot\prod_{i=1}^n(1+q^{-(2i-1)})
 & (-\epsilon = 1), \\
|2|^{n(2n-1)}\cdot q^{-2rn_0-2r^2+r}\cdot\prod_{i=1}^n(1+q^{-(2i-1)})
 & (-\epsilon = -1).
\end{cases}
\end{align*}
This proves \eqref{alpha1 1} and \eqref{alpha1 2} of Proposition \ref{alpha1 min}.

Finally, we prove \eqref{alpha1 min2}. By the definition of the $\gamma$-factor, we have the following  (local) functional equation of the zeta integral:
\begin{align*}
\frac{Z^W(f_s^\circ, \xi^\circ)}{m^\circ(s)} 
 = &e(G(W))\frac{Z^W(f_{-s}^\circ, \xi^\circ)}{\gamma^W(s+\frac{1}{2},1_{G(W)}\times 1_{F^\times}, \psi)}
  \prod_{i=0}^{n-1}\gamma(2s-2i, 1_{F^\times}, \psi) \\
    &\times |2|_F^{2ns-n(n-\frac{1}{2})} |N(R(\ue))|_F^{-s}\cdot \epsilon(\frac{1}{2}, \chi_W, \psi). 
\end{align*}
Since $f_{-\rho}^\circ$ is the constant function with value $1$ on $G(W^\Box)$, we have $Z^W(f_{\rho}^\circ, \xi^\circ) = |G(W)|$. Hence, by Lemma \ref{triv gamma}, we have
\[
\frac{Z^W(f_{-\rho}^\circ, \xi^\circ)}{m^\circ(\rho)}= |N(R(\ue))|^{-n+\frac{1}{2}}\times
\begin{cases}
|2|_F^{\frac{1}{2}}\cdot e(G(W)) & (n=1), \\
-|2|_F^3\cdot e(G(W)) & (n= 2), \\
-|2|_F^{\frac{15}{2}}\cdot e(G(W)) & (n=3). \end{cases}
\]
Therefore, we have
\[
\alpha_1(W) = |N(R(\ue))|^{-n + \frac{1}{2}}\times
\begin{cases} |2|_F\cdot (1+q^{-1}) & (n=1), \\
|2|_F^6\cdot q^{-1}\cdot (1+q^{-1})(1+q^{-3}) & (n=2),  \\
|2|_F^{15}\cdot q^{-3}\cdot (1+q^{-1})(1+q^{-3})(1+q^{-5}) &(n=3).
\end{cases}
\]
Thus, we complete the proof of Proposition \ref{alpha1 min}.

%%%%%%%%%%%%%%%%%%%%%%%%%%%%%%%%%   S8      %%%%%%%%%%%%%%%%%%%%%%%%%%%%%%%%%%

\section{
            Local Weil representations
            }\label{Local Weil reps}

In this paper, we consider the two reductive dual pairs: $(G(V), G(W^\Box))$ and $(G(V), G(W))$. Here, we use the word ``reductive dual pair'' in the sense of \cite{GS12}. (See \cite[Remarks (a)]{GS12} for the discussion for this.)
The purpose of this section is to describe the Schr\"{o}dinger models of the local Weil representations on $(G(V), G(W^\Box))$ and $(G(V), G(W))$.

\subsection{The metaplectic group and the Weil representation}

First, we recall the definition of the Metaplectic group and the Schr\"{o}dinger model of the Weil representation. 
Let $U$ be a symplectic space over $F$, let $\langle \ ,\ \rangle_U$ be the symplectic form on $U$, and let $K, L$ be maximal totally isotropic subspaces so that $U = K + L$.
We fix a non-trivial additive character $\psi\colon F \rightarrow \C^1$.
We denote by $r_{\psi, L}$ the Segal-Shale-Weil projective representation which is given by
\[
[r_{\psi, L}(g)\phi](x) = \int_{\mathcal{Y}_c} \psi(\frac{1}{2}\langle xa, xb \rangle_U + \langle yc,xb \rangle_U +\frac{1}{2}\langle yc, yd \rangle_U) \phi(xa + yc) \: d\mu_g(y)
\]
for $\phi \in \cS(K)$ and $g = \begin{pmatrix} a&b\\ c&d\end{pmatrix} \in \Sp(U)$.
Here, we consider a basis $(u_1, \ldots, u_{2t})$ with $u_1, \ldots, u_t \in L$ and $u_{t+1}, \ldots, u_{2t}\in K$ to give the matrix representation of $g$, we denote by $\mathcal{Y}_c$ the quotient $\ker(c)\cap L \backslash L$, and we denote by $\mu_g(y)$ the Haar measure on $\mathcal{Y}_c$ so that $r_L(g)$ keeps the $L^2$-norm of $\cS(K)$. Then, there is a 2-cocycle $c_{\psi, L}\colon \Sp(U)\times\Sp(U)\rightarrow \C^1$ so that 
\[
r_{\psi, L}(g_1)r_{\psi, L}(g_2) = c_{\psi, L}(g_1, g_2)r_{\psi, L}(g_1g_2) 
\]
for $g_1, g_2 \in \Sp(U)$. For the discussion of the definition, see \cite[Theorem 3.5]{Rao93}. The explicit formula of $c_{\psi, L}$ has been already established (\cite{Pe80}, \cite{Rao93}), but we do not discuss it. By ${\rm Mp}(U, c_{\psi, L})$ we mean the group $\Sp(U)\times\C^1$ together with the binary operation
\[
(g_1, z_1)\cdot (g_2, z_2) = (g_1g_2, z_1z_2c_{\psi, L}(g_1, g_2))
\]
for $g_1, g_2 \in \Sp(U)$ and $z_1, z_2 \in \C^1$, and we call it the metaplectic group associated to $c_L$. 
Then, the Weil representation $\omega_{\psi, L}$ of ${\rm Mp}(U, c_{\psi, L})$ is realized on the space $\cS(K)$ of Schwartz-Bruhat functions  on $K$ by
\[
[\omega_{\psi, L}(g, z) \phi](x) = z\cdot [r_{\psi, L}(g)\phi](x)
\]
for $g \in \Sp(U)$ and $z \in \C^1$. 

Note that the Segal-Shale-Weil projective representation $r_{\psi, L}$ and the 2-cocycle $c_{\psi, L}$ depends on the symplectic form $\langle \ , \ \rangle$: 
if we consider the symplectic form $-\langle \ , \ \rangle_U$ instead of $\langle \ , \ \rangle_U$, then the associated Segal-Shele-Weil representation and 2-cocycle are the unitary dual $\overline{r_{\psi, L}}$ of $r_{\psi, L}$ and complex conjugation $\overline{c_{\psi, L}}$ of $c_{\psi, L}$ respectively. We denote by $\overline{\omega_{\psi, L}}$ the Weil representation of ${\rm Mp}(U, \overline{c_{\psi, L}})$ induced by $\overline{r_{\psi, L}}$.

\subsection{For the pair $(G(V), G(W^\Box))$}\label{exbox}

In this subsection, we recall the explicit definition of Weil representation for the reductive dual pair $(G(V), G(W^\Box))$, which is given by Kudla \cite{Kud94}.

We fix a basis $\ue$ for $W$, and we take a basis ${\ue'}^\Box$ of $W^\Box$ as in \S\ref{doubling}. In this subsection, we identify $G(W)$ (resp.~ $G(W^\Box)$) with a subgroup of $\GL_n(D)$ (resp.~ $\GL_{2n}(D)$) by the basis $\ue$ (resp.~ ${\ue'}^\Box$). Moreover, we identify $G(V)$ with a subgroup of $\GL_m(D)$ by some fixed basis of $V$. Let $\bW^\Box = V\otimes_D W^\Box$, and let $\langle\langle \ , \ \rangle\rangle^\Box$ be the pairing on $\bW^\Box$ defined by
\[
\langle\langle x\otimes (y_1,y_2), x'\otimes(y_1', y_2')\rangle\rangle^\Box 
= T_D((x,x')\cdot(\langle y_1, y_1'\rangle^* - \langle y_2, y_2'\rangle^*))
\]
for $x_1, x_2 \in V$ and $y_1, y_2, y_1', y_2' \in W$.  
Then, $(G(V), G(W^\Box))$ is a reductive dual pair in $\Sp(\bW^\Box)$.  We consider a polar decomposition $\bW^\Box = (V\otimes W^\trd) \oplus (V\otimes W^\tru)$.
Then we denote by $r_{\psi, V\otimes W^\tru}$ the Segal-Shele-Weil representation of $\Sp(\bW^\Box)$ with respect to the symplectic form $\frac{1}{2}\langle\langle \ , \ \rangle\rangle^\Box$. 
Kudla defined a function $\beta_V\colon G(W^\Box)\rightarrow \C^1$ (\cite[p.~378]{Kud94}), and gave an explicit embedding
\[
\widetilde{j^\Box}\colon G(V) \times G(W^\Box) \rightarrow {\rm Mp}(\bW^\Box, c_{\psi, V\otimes W^\tru})
\]
by $\widetilde{j^\Box}(h, g) = (h^{-1}\otimes g, \beta_V(g))$. 
From now on, we denote by $\omega_\psi^\Box$ the pull-back $\widetilde{j^\Box}^*\omega_\psi^\Box$ of the Weil representation $\omega_{\psi, V\otimes W^\tru}$ of $\Mp(\bW^\Box, c_{\psi, V\otimes W^\tru})$. We will describe it explicitly following Kudla \cite[p.~400]{Kud94}.
Recall that $\tau$ denotes the certain Weyl element (for the definition, see \S\ref{doubling}).
Moreover, for  $a \in \GL(W^\tru)$, we denote by $m(a)$ the unique element of $G(W^\Box)$ such that $m(a)|_{W^\tru} = a$. Then, we have $\beta_V(b) = 1$ for $b \in U(W^\tru)$, $\beta_V(m(a)) = \chi_V(N(a))$ for $a \in \GL(W^\trd)$, and $\beta_V(\tau) = (-1)^{mn}\chi_V(-1)^n$.
Here, we denote by $N$ the reduced norm of ${\rm End}_D(W^\trd)$ over $F$. Thus, we have the following:

\begin{prop}\label{weilbox}
Let $\phi \in \cS(V \otimes W^\trd)$. Then, $\omega_\psi^\Box(h, g)\phi = \beta_V(g)r(g)(\phi\circ h^{-1})$. More precisely, 
\begin{itemize}
\item $[\omega_\psi^\Box(h,1)\phi](x) 
         = \phi(h^{-1}x)$ for $h \in G(V)$, 
\item $[\omega_\psi^\Box(1, m(a))\phi](x) 
          = \chi_V(N(a))|N(a)|^{-m}\phi(x\cdot {{}^t\!a^*}^{-1})$ 
         for $a \in \GL(W^\tru)$,
\item $[\omega_\psi^\Box(1, b)\phi](x) 
         = \psi(\frac{1}{4}\langle\langle x, x\cdot b \rangle\rangle^\Box)
           \phi(x)$
        for $b \in U(W^\tru)$,
\item the action of $\tau$ is given by
        \[
         [\omega_\psi^\Box(1, \tau)\phi](x) = \beta_V(\tau)\cdot \int_{V\otimes W^\trd}
        \psi(\frac{1}{2}\langle\langle y, x\tau \rangle\rangle^\Box)\phi(y) \: dy
        \]
        where $dy$ is the self-dual measure of $V\otimes W^\trd$ with respect to the pairing
        \[
       V\otimes W^\trd\times V\otimes W^\trd \rightarrow \C\colon
       x,y \mapsto \psi(\frac{1}{2}\langle\langle y, x\tau \rangle\rangle^\Box).
       \]
\end{itemize}
\end{prop}

%%%%%%%%%%

\subsection{For the pair $(G(V), G(W))$}\label{compatbox}

Now we consider the dual pair $(G(V), G(W))$. In this case, the splitting of metaplectic cover is defined via the ``doubled'' pair $(G(V), G(W^\Box))$.
Let $\bW = V\otimes_DW$, and let $\langle\langle \ , \ \rangle\rangle$ be the pairing on $\bW$ defined by
\[
\langle\langle x\otimes y, x'\otimes y'\rangle\rangle = T_D((x,x')\cdot \langle y, y'\rangle^*)
\]
for $x,x' \in V$ and $y,y' \in W$. 
Fix a polar decomposition $\bW = \bX \oplus \bY$ where $\bX$ and $\bY$ are certain totally isotropic subspaces. We denote by $r_{\bY}$ the Segal-Shele-Weil representation of $\Sp(\bW)$ with respect to the symplectic form $\frac{1}{2}\langle\langle \ , \ \rangle\rangle$. Since $\bY^\Box$ is a totally isotropic subspace of $\bW^\Box$, there is $\alpha \in \Sp(\bW^\Box)$ so that $\bY^\Box\cdot \alpha = V\otimes W^\tru$.
Put 
\[
\lambda(g) = c_{\bY^\Box}(\alpha, g\alpha^{-1})\cdot c_{\bY^\Box}(g, \alpha^{-1})
\]
for $g \in \Sp(\bW^\Box)$, whose coboundary realizes the ratio of the 2-cocycles  $c_{\bY^\Box}(-, -)$ and $c_{V\otimes W^\tru}(- , - )$  (\cite[(4.4)]{Kud94}). Then we define the function $\beta_\bY^{V}\colon G(W) \rightarrow \C^1$ by $\beta_{\bY}^V(g) = \lambda(1\otimes g)^{-1}\beta_V(g)$ for $g \in G(W)$. Then, the map
\[
g \mapsto (1\otimes g, \lambda(1\otimes g)^{-1}\beta_V(g))
\]
defines the embedding $G(W^\Box) \rightarrow {\rm Mp}(\bW^\Box, c_{\bY^\Box})$. We also define $\beta_\bY^{W}\colon G(V) \rightarrow \C^1$ by the same way using the doubled space $V^\Box$ of $V$. Then, we define the embedding
\[
\widetilde{j}\colon G(V)\times G(W) \rightarrow \Sp(\bW)
\]
by
\[
\widetilde{j}(h,g) = (h^{-1}\otimes g, \beta_\bY^W(h^{-1})\beta_\bY^V(g)c_\bY(h^{-1}\otimes 1, 1\otimes g))
\]
for $h \in G(V)$ and $g \in G(W)$. From now on, we denote by $\omega_\psi$ the pull-back $\widetilde{j}^*\omega_{\psi, \bY}$. An important property of $\omega_\psi$ is the relation with $\omega^\Box_\psi$. We fix Haar measures $dx$ and $dy$ of $\bX$ and $\bY$ so that they are dual each other with respect to the pairing
\[
\bX \times \bY \rightarrow \C^\times\colon (x,y) \mapsto \psi(\langle\langle x, y\rangle\rangle).
\]
Moreover, we define
\[
\bX^\tru = (\bX\oplus\bX)\cap W^\tru, \ \bX^\trd = (\bX\oplus\bX) \cap W^\trd
\]
and
\[
\bY^\tru = (\bY\oplus\bY)\cap W^\tru, \ \bY^\trd = (\bY\oplus\bY) \cap W^\trd.
\]
Then the vector space $V\otimes W^\trd$ decomposes into the direct sum 
\[
 \bX^\trd\oplus\bY^\trd.
\]
For $z \in V\otimes W^\trd$, we denote by $z_x$ (resp.~ $z_y$) the $\bX^\trd$-component (resp.~ the $\bY^\trd$-component) of $z$. 
We define the Haar measure $dx^\tru$ on $\bX^\tru$ by the push out measure $p_*(dx)$ where $p\colon \bX^\tru\rightarrow \bX$ is the first projection. We define the Haar measures $dx^\trd$, $dy^\tru$, $dy^\trd$ in the same way.
Then, the map 
\[
\delta\colon \cS(\bX)\otimes \overline{\cS(\bX)}=\cS(\bX\oplus\bX) \rightarrow \cS(V\otimes W^\trd)
\]
given by the partial Fourier transform
\[
[\delta(\phi_1\otimes\overline{\phi_2})](z) = \int_{\bX^\tru}(\phi_1\otimes\overline{\phi_2})(x^\tru + z_x) \cdot\psi(\frac{1}{2}\langle\langle x^\tru, z \rangle\rangle) \: dx^\tru
\]
is known to be compatible with the embedding $\iota\colon G(W) \times G(W) \rightarrow G(W^\Box)$.
Hence, we have
\[
F_{\delta(\phi_1\otimes\overline{\phi_2})}(\iota(g,1)) = (\omega_\psi(g)\phi_1, \phi_2)_\bX
\]
for $\phi_1, \phi_2 \in \cS(\bX)$ where $( \ , \ )_\bX$ is the $L^2$-inner product on $\bX$ defined by the measure $dx$. 

Finally, we prove the Plancherel formula for $\delta$:

\begin{prop}\label{inner prods}
Let $dz$ be the self-dual Haar measure on $V\otimes W^\trd$ with respect to the pairing
\begin{align}\label{pairing_VWtrd}
V\otimes W^\trd \times V\otimes W^\trd \rightarrow \C^\times
\colon (x,y) \mapsto \psi(\frac{1}{2}\langle\langle y, x\tau\rangle\rangle)
\end{align}
and let $( \ , \ )$ be the $L^2$-inner product on $V\otimes W^\trd$ defined by $dz$. Then, we have  
\[
(\delta(\phi_1\otimes\overline{\phi_2}), \delta(\phi_3\otimes\overline{\phi_4})) 
=|2|_F^{-2mn}\cdot |N(R(\ue))|^m\cdot (\phi_1, \phi_3)_\bX\cdot\overline{(\phi_2, \phi_4)}_\bX
\]
for $\phi_1, \phi_2, \phi_3, \phi_4 \in \cS(\bX)$.
\end{prop}

\begin{proof}
First, one can prove that $dz = |N(R(\ue))|^m\cdot dz_x^\trd\otimes dz_y^\trd$. Hence, we have
\begin{align*}
&(\delta(\phi_1\otimes\overline{\phi_2}), \delta(\phi_3\otimes\overline{\phi_4})) \\
&=\int_{V\otimes W^\trd}\delta(\phi_1\otimes\overline{\phi_2})(z)\cdot\overline{\delta(\phi_3\otimes\overline{\phi_4})(z)} \: dz \\
&=|N(R(\ue))|^m\int_{\bX^\trd}\int_{\bY^\trd}\delta(\phi_1\otimes\overline{\phi_2})(z_x+z_y)\cdot\overline{\delta(\phi_3\otimes\overline{\phi_4})(z_x+z_y)} \: dz_x^\trd dz_y^\trd \\
&=|N(R(\ue))|^m\int_{\bX^\trd}\int_{\bY^\trd}\int_{\bX^\tru}(\phi_1\otimes\overline{\phi_2})(z_x+x^\tru)\psi(\frac{1}{2}\langle\langle x^\tru, z_y\rangle\rangle^\Box) \\
& \ \ \ \ \cdot \overline{\delta(\phi_3\otimes\overline{\phi_4})(z_x+z_y)}\: dx^\trd dz_x^\trd dz_y^\trd \\
&=|N(R(\ue))|^m\int_{\bX^\trd}\int_{\bX^\tru}(\phi_1\otimes\overline{\phi_2})(z_x+x^\tru)\cdot\overline{(\phi_3\otimes\overline{\phi_4})(z_x+x^\tru)} \: dx^\tru dz_x^\trd \\
& = |2|^{-2mn}\cdot|N(R(\ue))|^m \int_{\bX}\int_{\bX}(\phi_1\otimes\overline{\phi_2})(x,x')\cdot\overline{(\phi_3\otimes\overline{\phi_4})(x,x')} \: dx dx'\\
& = |2|^{-2mn}\cdot|N(R(\ue))|^m\cdot (\phi_1, \phi_3)_\bX\cdot \overline{(\phi_2, \phi_4)_\bX}.
\end{align*}
Thus, we have the proposition.
\end{proof}

%%%%%%%%%%%%%%%%%%%%%%%%%%%%%%%%%     S9       %%%%%%%%%%%%%%%%%%%%%%%%%%%%%%%%%%

\section{
           Local theta correspondence
           }\label{Loc theta corr}

In this section, we recall notations and properties of local theta correspondence for quaternionic dual pairs.
   
%%%%%%%%%%
\subsection{
                Definition
                }

Fix a non-trivial additive character $\psi$ of $F$.  Let $\omega_\psi$ be the Weil representation of $G(V)\times G(W)$ (see \S\ref{compatbox}).
For an irreducible representation $\pi$ of $G(W)$, we define $\Theta_\psi(\pi,V)$ as the largest quotient module
\[
  (\omega_\psi \otimes \pi^\vee)_{G(W)} 
\]
 of $\omega_\psi \otimes \pi^\vee$ on which $G(W)$ acts trivially. This is a representation of $G(V)$. We define the theta correspondence $\theta_\psi(\pi, V)$ of $\pi$ by
\[
\theta_\psi(\pi,V) = \begin{cases}
			   0 & (\Theta_\psi(\pi,V) = 0), \\
			\mbox{the maximal semisimple quotient of $\Theta_\psi(\pi, V)$} 
                         & (\Theta_\psi(\pi, V) \not= 0). \end{cases}
\]
If we consider the pair $(G(V), G(W^\Box))$, we use $\omega_\psi^\Box$ instead of $\omega_\psi$ to define the theta correspondence.

The following theorem is a fundamental result in the study of theta correspondence. In particular,  the properties \eqref{hd1} and \eqref{hd2} are called the Howe duality, which was proved by Waldspurger \cite{Wal90} when the residual characteristic of $F$ is not $2$, and was completely proved by Gan and Takeda \cite{GT16} (for the non-quaternionic dual pairs) and Gan and Sun \cite{GS17} (for the quaternionic dual pairs). 

\begin{thm}\label{Howe duality}
For irreducible representations $\pi_1, \pi_2$ of $G(W)$, we have
\begin{enumerate}
\item $\theta_\psi(\pi_1, V)$ is irreducible if it is non-zero, \label{hd1}
\item $\pi_1 \cong \pi_2$ if $\theta_\psi(\pi_1, V) \cong \theta_\psi(\pi_2, V) \not=0$, \label{hd2}
\item $\theta_\psi(\pi_1, V)^\vee \cong \theta_{\overline{\psi}}(\pi_1^\vee, V)$. \label{hd3}
\end{enumerate}
\end{thm}

\begin{proof}
\cite[Theorem 1.3]{GS17}.
\end{proof}

For an irreducible representation $\rho$ of a group $H$, we denote by $\omega_\rho$ the central character of $\rho$.

\begin{prop}\label{cent_char}
Let $\pi$ be an irreducible representation of $G(W)$, and suppose that $\theta_\psi(\pi, V)$ is non-zero. We denote by $\sigma$ the representation $\theta_\psi(\pi, V)$. Then, we have
\[
 c_\pi(-1) c_\sigma(-1) = \chi_V(-1)^n \chi_W(-1)^m.
\] 
\end{prop}

\begin{proof}
Let $\bW = \bX + \bY$ be a polar decomposition as in \S\ref{compatbox}.
It suffices to show that $\omega_\psi(-1, -1)$ acts on $\cS(\bX)$ by the scalar multiplication by $\chi_V(-1)^n \chi_W(-1)^m$. Since $r_{\psi, \bY}(-1, -1)$ is the identity operator on $\cS(\bX)$, we have the action of $\omega_\psi(-1, -1)$ is the scalar multiplication by $\beta_\bY^V(-1)\beta_\bY^W(-1) c_{\psi, \bY}(-1, -1)$. One can show that $c_{\psi, \bY}(-1, -1) = 1$. Besides, by the definition of $\beta_\bY^V$, we have
\[
\beta_\bY^V(-1) = \beta_V(\iota(-1, 1)) = \beta_V(\tau) = (-1)^{mn}\chi_V(-1)^n.
\]
By the same way, we have $\beta_\bY^W(-1) = (-1)^{mn}\chi_W(-1)^m$. These imply the proposition.
\end{proof}

%%%%%%%%%%
\subsection{Square integrability}

In this subsection, we explain the preservation of the square integrability under the theta correspondence, which is necessary for the setup of the main result. Let $\pi$ be an irreducible square-integrable representation of $G(W)$, and let $\sigma\coloneqq  \theta_\psi(\pi, V)$. In this subsection, we assume that $l=1$ and $\sigma \not=0$. We denote by $\theta$ the $G(V)$-equivalent and $G(W)$-invariant  natural quotient map
\[
\omega_\psi \otimes \pi \rightarrow \sigma.  
\]
Let $( \ , \ )_\pi\colon \pi\times\pi \rightarrow \C$ be a non-zero $G(W)$-invariant Hermitian pairing on $\pi$. We define a non-zero $G(V)$-invariant Hermitian pairing $( \ , \ )_\sigma\colon \sigma\times\sigma\rightarrow \C$ by
\begin{align}\label{def of pairing}
(\theta(\phi_1, v_1), \theta(\phi_2, v_2))_\sigma \coloneqq    \int_{G(W)} (\omega_\psi(g)\phi_1, \phi_2)\cdot \overline{(\pi(g)v_1, v_2)_\pi} \: dg.
\end{align}
\begin{lem}\label{sqrrrrr}
The integral of the right-hand side of \eqref{def of pairing} converges absolutely, and the map yielded by the integral factors through the natural quotient map
\[
\omega_\psi \otimes \pi ^\vee\times \omega_\psi \otimes \pi^\vee \rightarrow \sigma \times \sigma.
\]
Moreover, we have that $\sigma$ is a square-integrable representation.
\end{lem}

\begin{proof}
One can construct an integrable dominating function of the function 
\[
g \mapsto (\omega_\psi(g)\phi_1, \phi_2)\cdot \overline{(\pi(g)v_1, v_2)_\pi}
\]
on $G(W)$, which implies that the integral converges absolutely (see \cite[Lemma 9.5]{GI14}).
Similar to \cite[Appendix D, Lemma D1]{GI14}.
\end{proof}

%%%%%%%%%%
\subsection{
		   Tower properties
		   }\label{tower properties}

In this subsection, we discuss some properties related to Witt towers.
Let $V_0$ be a right anisotropic $\epsilon$-Hermitian space. Put $m_0 \coloneqq   \dim_D V_0$.
For a non-negative integer $t$, we define 
\[
V_t = X_t \oplus V_0 \oplus X_t^*
\]
where $X_t$ and $X_t^*$ are $t$-dimensional right $D$-vector spaces. Fix a basis $\lambda_1, \ldots, \lambda_t$ for $X_t$ and fix a basis $\lambda_{-1}, \ldots, \lambda_{-t}$ for $X_t^*$. Then we define an $\epsilon$-Hermitian pairing $( \ , \ )_t$ on $V_t$ by 
\[
(\lambda_i, \lambda_{-j})_t = \delta_{ij}, \ (\lambda_i, x_0)_t = (x_0, \lambda_{-j})_t=0, \ (x_0, x_0')_t = (x_0, x_0')_0 
\]
for $i,j=1, \ldots, t$ and $x_0, x_0' \in V_0$. Here $( \ , \ )_0$ is the pairing associated with $V_0$.

First, we state the conservation relation of Sun and Zhu [SZ15]. Let $V_0^\dagger$ be a right anisotropic $\epsilon$-Hermitian space such that $\chi_{V_0^\dagger} = \chi_{V_0}$ and $V_0^\dagger \not\cong V_0$. Such $V_0^\dagger$ is determined uniquely. Take $\{V_t^\dagger\}_{t\geq0}$ as the Witt tower containing $V_0^\dagger$. Let $\pi$ be an irreducible representation of $G(W)$. There is a non-negative integer $r(\pi)$ such that $\Theta_\psi(\pi, V_{r(\pi)}) \not=0$ and $\Theta(\pi, V_t) =0$ for $t < r(\pi)$. It is known that $\Theta_\psi(\pi,V_{r(\pi)})$ is irreducible and supercuspidal if $\pi$ is supercuspidal \cite[p.~69]{MVW87}. We call $r(\pi)$ the first occurrence index for the theta correspondence from $\pi$ to the Witt tower $\{V_t\}_{t\geq0}$. Denote by $r^\dagger(\pi)$ the first occurrence index for the theta correspondence from $\pi$ to $\{V_t^\dagger\}_{t\geq0}$. 

\begin{prop}\label{cons.rel}
Let $\pi$ be an irreducible representation of $G(W)$. Then we have
\[
m(\pi) + m^\dagger(\pi) = 2n + 2 - \epsilon
\]
where $m(\pi) = 2r(\pi) + \dim_DV_0$, and $m^\dagger(\pi) = 2r^\dagger(\pi) + \dim_DV_0^\dagger$.
\end{prop}

\begin{proof}
\cite{SZ15}.
\end{proof}

Then, we explain the behavior of theta correspondence when we change indexes of Witt towers. However, before doing that, we state here the analogue of the Gelfand-Kazhdan Theorem (\cite[Theorem 7.3]{BZ76}) for $\GL_r(D)$, which we use in the proof of Proposition \ref{ttoowweerr}:

\begin{lem}\label{Gelfand-Kazhdan}
Let $\tau$ be an irreducible representation of $\GL_r(D)$, and let $\tau^\theta$ be the irreducible representation of $\GL_r(D)$ defined by $\tau^\theta(g) = \tau({{}^t\!g^*}^{-1})$ for $g \in \GL_r(D)$. Then, $\tau^\theta$ is equivalent to the contragredient representation $\tau^\vee$ of $\tau$.
\end{lem}

\begin{proof}
\cite[Theorem 3.1]{Rag02}.
\end{proof}

\begin{prop}\label{ttoowweerr}
Let $\{ W_t\}_{t\geq 0}$ be a Witt tower of right $(-\epsilon)$-Hermitian spaces.
\begin{enumerate}
\item \label{11111}
Let $\pi$ be an irreducible representation of $G(W_i)$, and let $\sigma = \theta_\psi(\pi, V_j)$.
Suppose that $j \geq r(\pi)$, and we denote by $\sigma_{j'}$ the representation $\theta_\psi(\pi, V_{j'})$ for $r(\pi) \leq j' \leq j$. Then, $\sigma$ is a subquotient of an induced representation 
\[
\Ind_{Q_{j',j}}^{G(V_i)}\sigma_{j'}\boxtimes {\chi_W}|N_{X_{j',j}}|^{l_{i,j}+j - r(\pi)}.
\]  
Here, $l_{i,j} = 2\dim W_i - 2\dim V_j - \epsilon$, $X_{j',j}$ is a subspace of $X_{j'}$ spanned by $\lambda_{j'+1}, \ldots, \lambda_j$, $N_{X_{j',j}}$ is the reduced norm of ${\rm End}(X_{j',j})$, and $Q_{j',j}$ is the parabolic subgroup preserving $X_{j',j}$.  
\item \label{22222}
Let $\pi$ be an irreducible representation of $G(W_{i'})$, 
let $\sigma = \theta_\psi(\pi, V_{j'})$, 
let $\tau$ be a non-trivial supercuspidal irreducible representation of $\GL_r(D)$, 
let $s$ be a complex number, and let $\pi'$ be an irreducible subquotient of $\Ind_{P_{i',i}}^{G(W_i)}(\pi\boxtimes\tau_s\chi_V)$ where $i=i'+r$ and $P_{i',i}$ is the parabolic subgroup preserving an $r$-dimensional totally isotropic subspace of $W_{i'}$. 
Suppose that $\sigma \not=0$. Then, we have that $\theta_\psi(\pi', V_{j})$ is a subquotient of $\Ind_{Q_{j',j}}^{G(V_{j})}\sigma\boxtimes \tau_s\chi_W$. Here, $j=j'+r$, and $\tau_s\chi_W$ is the representation of $\GL_r(D)$ defined by $\tau_s\chi_W(g) = \tau(g)\chi_W(N(g)) |N(g)|^s$ for $g \in \GL_r(D)$, where $N$ denotes the reduced norm.
\end{enumerate}
\end{prop}

\begin{proof}
These properties are proved by analyzing the Jacquet module of Weil representations: it goes a similar line with \cite{Mui06}, however, we explain for the readers (see also \cite{Han11}).
In the proof, we denote by $\omega_\psi[j,i]$ the Weil representation associated with the reductive dual pair $(G(V_j), G(W_i))$. Moreover, for a representation $\rho$ of $G(V_j)\times G(W_i)$, for $0 \leq i' \leq i$, and for $0 \leq j' \leq j$, we denote by $J_{j',i'}\rho$ the Jacquet module of $\rho$ with respect to the parabolic subgroup $Q_{j',j} \times P_{i',i}$. Then, by \cite{MVW87}, we have a $G(V_{j'})\times \GL_{j-j'}(D) \times G(W_i)$ equivalent filtration: 
\[
J_{j',i}(\omega_\psi[j,i]) = R_0 \supset R_1 \supset \cdots \supset R_t \supset R_{t+1} = 0.
\]
Here, 
\begin{align*}
&t = \min\{ j-j', i \}, \\
&R_0/R_1 = \chi_W |N_{X_{j',j}}|^{l_{i,j}+j-j'}\boxtimes \omega_\psi[j',i], \\
&R_k/R_{k+1} = \Ind_{P_{i-k,i}}^{G(W_i)}\rho_k \mbox{ for some representation } \rho_k \ (k=1, \ldots, t-1),
\end{align*}
and moreover if $j-j' \leq i$, we have
\[
R_t = \Ind_{P_{i',i}}^{G(W_i)} \cS(\GL_{j-j'}(D))\boxtimes\omega_\psi[j', i']
\]
where $i' = i - (j-j')$, and the action of $\GL_{j-j'}(D)\times \GL_{i-i'}(D)$ on $\cS(\GL_{j-j'}(D))$ is given by
\[
[(g_1, g_2)\cdot \varphi] (g) = \chi_W(N(g_1))\chi_V(N(g_2))\varphi(g_1^{-1}gg_2)
\]
for $g_1 \in \GL_{j-j'}(D), g \in \GL_{j-j'}(D)$, and $g_2 \in \GL_{i-i'}(D)$, where $N$ denotes the reduced norm.
Now we prove \eqref{11111}. Composing $J_{j',i}(\omega_\psi[j,i]) \rightarrow R_0/R_1$ with the $G(V_{j'})\times G(W_i)$-equivalent surjection
\[
\omega_\psi[j',i] \rightarrow \sigma\boxtimes\pi,
\]
we have a non-zero morphism
\[
J_{j',i}(\omega_\psi[j,i]) \rightarrow \chi_W|N_{X_{j',j}}|^{l+j-j'} \boxtimes \sigma\boxtimes \pi.
\]
Hence we have \eqref{11111}. Then we prove \eqref{22222}. Let $\pi'$ be an irreducible component of $\Ind_{P_{i',i}}^{G(W_i)}\pi\boxtimes\tau_s\chi_V$.
First, we have
\[
\Hom(R_k/R_{k+1}, \pi') = \Hom(\rho_k, J_{i-k}\pi').
\]
Here, we denote by $J_{i-k}\pi'$ the Jacquet module with respect to the parabolic subgroup $P_{i',i}$. However, since $\tau$ is supercuspidal, one can prove $J_{i-k}\Ind_{P_{i',i}}^{G(W_i)}\pi\boxtimes\tau_s\chi_V = 0$ for $k=1,2,\ldots, t-1$ by considering the filtration of  Bernstein and Zelevinsky (\cite[Theorem 5.2]{BZ77}), and thus the right-hand side is $0$. Hence, we have 
\[
R_1\otimes {\pi'}^\vee \cong R_t \otimes {\pi'}^\vee.
\]
Moreover, since $\tau_s\chi_W \not\cong \chi_W|N_{j',j}|^{l_{i,j}+j-j'}$, we have
\[
R_0\otimes(\tau_s\chi_W)^\vee \cong R_1\otimes (\tau_s\chi_W)^\vee.
\]
On the other hand, the nonzero $\GL_r(D)\times\GL_r(D)$-equivalent map
\[
\cS(\GL_{j-j'}(D)) \otimes ((\tau_s\chi_V)^\vee \boxtimes \tau_s\chi_W) \rightarrow \C \colon (\varphi, x,x') \mapsto \int_{\GL_r(D)}\varphi(g)\langle \tau_s(g)x, x'\rangle \chi_W\chi_V^{-1}(N_{j-j'}(g)) \: dg 
\]
yields a nonzero $\GL_r(D)\times\GL_r(D)$-equivalent map
\[
\cS(\GL_{j-j'}(D)) \otimes (\tau_s\chi_V)^\vee \rightarrow (\tau_s\chi_W)^\vee. 
\]
By combining the above arguments, and by Lemma \ref{Gelfand-Kazhdan}, we have a nonzero $G(V_{j'})\times \GL_{j-j'}(D) \times G(W_i)$-equivalent map
\begin{align*}
&J_{j',i}(\omega_\psi[j,i])\otimes (\sigma\boxtimes\tau_s\chi_W)^\vee\otimes(\pi')^\vee \\
& = R_t \otimes (\sigma\boxtimes\tau_s\chi_W)^\vee\boxtimes (\pi')^\vee \\
& = (\Ind_{P_{i',i}}^{G(W_i)} \cS(\GL_{j-j'}(D))\boxtimes\omega_\psi[i',j']) \otimes  (\sigma\boxtimes\tau_s\chi_W)^\vee\boxtimes (\pi')^\vee \\
& \rightarrow (\Ind_{P_{i',i}}^{G(W_i)} (\tau_s\chi_V)^\vee\boxtimes\pi)\otimes (\pi')^\vee \\
&\cong (\Ind_{P_{i',i}}^{G(W_i)} ({\tau^\theta}\chi_V)_{-s}\boxtimes\pi)\otimes (\pi')^\vee \\
&\cong (\Ind_{P_{i',i}}^{G(W_i)} (\tau_s\chi_V)\boxtimes\pi)\otimes (\pi')^\vee \\
&\rightarrow \C.
\end{align*}
Hence we have \eqref{22222}.
\end{proof}

By the proof of Proposition \ref{ttoowweerr}, we also have a slightly different property:
\begin{cor}\label{reducing}
Let $\{W_t\}_{t\geq0}$ be a Witt tower of right $(-\epsilon)$-Hermitian spaces, let $i,j,j'$ be non-negative integers so that $j-j' >0$, let $\pi$ be an irreducible representation of $G(W_{i})$, let $\sigma = \theta_\psi(\pi, V_j)$.
Suppose that $\sigma\not=0$, and $\sigma$ is a subrepresention of an induced representation $\Ind_{Q_{j',j}}^{G(V_j)}\sigma'\boxtimes\tau_s\chi_W$ where $\sigma'$ is an irreducible representation of $G(V_{j'})$, $\tau$ is an irreducible supercuspidal representation of $\GL_{j-j'}(D)$, and $s \in \C$. Moreover, suppose that $\theta_\psi(\pi, V_{j'}) =0$. Then, we have $i \geq j-j'$, and there exists an irreducible representation $\pi'$ of $G(W_{i'})$ such that $\theta_\psi(\pi', V_{j'}) \cong \sigma'$. Here we put $i' = i- (j-j')$. Moreover, $\pi$ is an irreducible subquotient of $\Ind_{P_{i',i}}^{G(W_i)}\pi'\boxtimes\tau_s\chi_W$.
\end{cor}

\begin{proof}
We use the notation of the proof of Proposition \ref{ttoowweerr}.
Since there is a non-zero $G(V_j)\times G(W_i)$-equivalent map
\[
\omega_\psi[j,i] \rightarrow \sigma\boxtimes\pi,
\]
by the Frobenius reciprocity, we have a non-zero $G(V_{j'})\times \GL_{j-j'}(D) \times G(W_i)$-equivalent map
\begin{align}\label{theta x1}
(\tau_s\chi_W)^\vee\boxtimes\pi^\vee\otimes J_{j',i}\omega_\psi[j,i] \rightarrow \sigma'.
\end{align}
Then, the assumption $\theta_\psi(\pi, V_{j'}) = 0$ implies that
\[
\pi^\vee\otimes R_{0}/R_{1} = 0.
\]
Moreover, as in the proof of Proposition \ref{ttoowweerr} \eqref{22222}, we have
\[
(\tau_s\chi_W)^\vee\boxtimes\pi^\vee\otimes J_{j',i}\omega_\psi[j,i] = (\tau_s\chi_W)^\vee\boxtimes\pi^\vee\otimes R_{j-j'}.
\]
(Here, we put $R_k = 0$ for $k > t$.) Thus, $R_{i-i'}$ is forced not to be zero, and we have $i \geq j-j'$. By using the Frobenius reciprocity again, we have a nonzero $G(V_{j'})\times \GL_{j-j'}(D)\times G(W_{i'})\times \GL_{i-i'}(D)$-equivalent map
\[
((\tau_s\chi_W)^\vee\boxtimes(J_{i',j}\pi)^\vee) \otimes (\cS(\GL_{i-i'}(D))\boxtimes\omega_\psi[j',i']) \rightarrow \sigma'. 
\]
Thus, ${\sigma'}^\vee\otimes\omega_\psi[j',i'] \not=0$.
Put $\pi' \coloneqq   \theta_\psi(\sigma', W_{i'})$. Then, $\theta_\psi(\sigma, W_i)$ is nonzero, and it is an irreducible subquotient $\pi''$ of $\Ind_{P_{i',i}}^{G(W_i)}\pi' \boxtimes \tau_s\chi_V$. However, by the Howe duality (Theorem \ref{Howe duality}), $\pi''$ coincides with $\pi$. Thus we have the corollary.
\end{proof}

%%%%%%%%%%%%%%%%%%%%%%%%%%%%%%%%      S10     %%%%%%%%%%%%%%%%%%%%%%%%%%%%%%%%%

\section{
            The local Siegel-Weil formula
            }\label{SS local SW formula}
In this section, we state the local Siegel-Weil formula, which is a local analogue of the (bounded and first term) Siegel-Weil formula. {\bf We assume $l=1$ and $n \geq 0$ in this section}.

\subsection{The map $\cI$}\label{def_of_I}

We define the $\Delta G(W^\Box)\times G(V) \times G(V)$-invariant map
\[
\cI\colon \omega_\psi^\Box\otimes\overline{\omega_\psi^\Box} \rightarrow \C
\]
by 
\[
\cI(\phi, \phi') = \int_{G(V)} (\omega_\psi^\Box(h)\phi, \phi') \: dh
\]
for $\phi, \phi' \in \omega_\psi^\Box$ where $( \ , \ )$ is the $L^2$-norm of $\cS(V\otimes W^\trd)$ as in Proposition \ref{inner prods}. The integral defining $\cI( \ , \ )$ converges absolutely by \cite[Theorem 3.2]{Li89}.

\subsection{The map $\cE$}\label{def_of_E}

Let $V^\flat$ be the unique $\epsilon$-Hermitian space over $D$ so that $\dim_DV^\flat = m + 1$ and $\chi_{V^\flat} = \chi_V$.
Such space exists since we have assumed that $l = 1$ and $n \geq 1$.
Consider the $G(W^\Box)$-invariant map
\[
\cS(V\otimes W^\trd) \rightarrow I(-\frac{1}{2}, \chi_V)\colon \phi \mapsto F_\phi
\]
defined by $F_\phi(g) = [\omega_\psi^\Box(1,g)\phi](0)$ for $\phi \in $ and $g \in G(W^\Box)$. Similarly, there is a $G(W^\Box)$-invariant map $\cS(V^\flat\otimes W^\trd)\rightarrow I(\frac{1}{2}, \chi_V)$. We denote by $R^W(V)$ and $R^W(V^\flat)$ the images of the above maps respectively.
Then we have the following exact sequence:
\[
\xymatrix{
0 \ar[r] & R^W(V^\flat) \ar[r] & I(\frac{1}{2}, \chi_V) \ar[rr]^-{M(\frac{1}{2},\chi_V)} &
& R^W(V) \ar[r] & 0 
}
\]
(\cite[Proposition 7.6]{Yam11}). For $\phi \in \cS(V\otimes W^\trd)$, we denote by $F_\phi^\dagger \in I(\frac{1}{2},\chi_V)$ a section such that $M(\frac{1}{2}, \chi_V)F_\phi^\dagger = F_\phi$.
Then, we define the map $\cE$ by 
\[
\cE(\phi, \phi') = \int_{G(W)} F_\phi^\dagger(\iota(g,1))\cdot \overline{F_{\phi'}(\iota(g,1))}\: dg.
\]
The integral defining $\cE$ converges absolutely by Proposition \ref{zetabasic}. Moreover, the following lemma implies that  the definition of $\cE(\phi, \phi')$ does not depend on the choice of $F_\phi^\dagger$.

\begin{lem}\label{zerozero}
If $f \in R^W(V^\flat)$ and $h \in R^W(V)$, then we have 
\[
\int_{G(W)}f(\iota(g,1))\cdot \overline{h(\iota(g,1))}\: dg = 0.
\]
\end{lem}

\begin{proof}
By the proof of Lemma \ref{dim = 1}, we have
\[
\Hom_{G(V)\times G(V)} (I(\rho, 1_{F^\times}), \C) = \Hom_{G(V^\Box)}(I(\rho, 1_{F^\times}),\C) = Z\cdot \C
\]
where 
\[
Z(F) = \int_{G(V)}F(\iota(g,1))\: dg
\]
for $F \in I(\rho, 1)$. Thus, if there are $f \in R^W(V^\flat), h \in R^W(V)$ so that $Z(f\cdot \overline{h}) \not=0$, we would have $R^W(V^\flat) \cong \overline{R^W(V)}^\vee$. 
Since $\overline{I(-\frac{1}{2}, \chi_V)} \cong I(-\frac{1}{2}, \chi_V)$, we have $\overline{R^W(V)}\cong R^W(V)$. Put $\sigma \coloneqq   R^W(V^\flat)$. Then, we have
\[
\Theta(\sigma, V^\flat) = 1_{V^\flat}, \ \Theta(\sigma, V) = 1_V.
\]
However, according to the conservation relation (Proposition \ref{cons.rel}), one of them must vanish since $\dim V + \dim V^\flat = 2n-\epsilon$. This is a contradiction, and we have the lemma. 
\end{proof}

\subsection{Local Siegel-Weil formula}\label{lSW}

The following lemma gives the definition of $\alpha_2(V,W)$, which is the second constant we are interested in.

\begin{lem}\label{locSW2}
There is a non-zero constant $\alpha_2(V,W)$ such that $\cI = \alpha_2(V,W)\cdot \cE$.
\end{lem}

\begin{proof}
The two maps $\cI$, $\cE$ are $\Delta G(W^\Box)\times G(V)\times G(V)$-invariant map. On the other hand, we have
\[
\dim\Hom_{\Delta G(W^\Box)\times G(V)\times G(V)}(\omega_\psi^\Box\otimes\overline{\omega_\psi^\Box}, \C) = \dim\Hom_{\Delta G(W^\Box)}(R^W(V)\otimes\overline{R^W(V^\flat)}, \C) = 1.
\]
Hence, it suffices to show that $\cI$ and $\cE$ are non-zero. Let $\phi\in \cS(V\otimes W^\trd)$ be a positive function. Choose  a neighbourhood $U$ of $1$ in $G(V)$ so that $\omega_\psi^\Box(h)\phi = \phi$ for all $h \in U$.  
Then, we have
\begin{align*}
\cI(\phi, \phi) &\geq \int_U (\omega_\psi^\Box(h)\phi, \phi) \:dh \\
&=|U|\cdot (\phi, \phi) > 0.
\end{align*}
Thus we have $\cI\not=0$. The non-vanishing of $\cE$ is obtained by Lemma \ref{anE} below. However, we also give a short proof. 
Consider the non-zero pairing
\begin{align}\label{Epair}
I(\frac{1}{2}, \chi_V)\times I(-\frac{1}{2}, \chi_V) \rightarrow \C\colon (f,h) \mapsto Z(f\cdot\overline{h})
\end{align}
where $Z$ is the map as in the proof of Lemma \ref{zerozero}. 
We assume $\cE = 0$ to derive a contradiction. Then the pairing \eqref{Epair} factors through the quotient
\[
I(\frac{1}{2}, \chi_V)\times I(-\frac{1}{2}, \chi_V)\rightarrow I(\frac{1}{2}, \chi_V)\times R(V^\flat)
\]
by \cite[Theorems 1.3, 1.4]{Yam11}. But this implies $R(V) \cong \overline{R(V^\flat)}^\vee$, which contradicts the conservation relation as in the proof of Lemma \ref{zerozero}. Hence we have $\cE \not=0$, and we finish the proof of Lemma \ref{lSW}.
\end{proof}

We will determine the constant $\alpha_2(V,W)$ completely in \S\ref{det alpha}. However we calculate $\alpha_2(V,W)$ directly when either $V$ or $W$ is anisotropic. The proof will be given in \S\S\ref{Min1}--\ref{Min2}:

\begin{prop}\label{locSW1}
\begin{enumerate}
\item Suppose that $-\epsilon=1$ and $V$ is anisotropic, then we have
\begin{align*}
\alpha_2(V, W) &= |N(R(\ue))|^{n+\frac{1}{2}}\chi_V(-1)^n \\
&\times
\begin{cases}
-|2|_F^{-\frac{5}{2}}\cdot (1+q^{-1}) & (m=1 \mbox{ with $\chi_V$ unramified}), \\
-|2|_F^{-\frac{5}{2}}\cdot q^{-\frac{1}{2}} & (m=1 \mbox{ with $\chi_V$ ramified}), \\
|2|_F^{-7}\cdot q^{-2}(1+q^{-2}) & (m=2 \mbox{ with $\chi_V$ unramified}), \\
|2|_F^{-7}\cdot q^{-\frac{5}{2}}\cdot (1+q^{-1}) & (m=2 \mbox{ with $\chi_V$ ramified}), \\
-|2|_F^{-\frac{27}{2}}\cdot q^{-6}(1+q^{-1})(1+q^{-2}) & (m=3).
\end{cases}
\end{align*}\label{SWmin2}
\item Suppose that $-\epsilon = -1$ and either $V$ or $W$ is anisotropic, then we have
\begin{align*}
\alpha_2(V, W) &= |N(R(\ue))|^{n - \frac{1}{2}} \\ 
&\times\begin{cases}
|2|_F^{-\frac{1}{2}} & (n=1), \\
|2|_F^{-3}\cdot q^{-1}\cdot (1+q^{-1}) & (n=2), \\
-|2|_F^{-\frac{15}{2}}q^{-4}\cdot\frac{(1+q^{-1})(1-q^{-4})}{1-q^{-3}} & (n=3).
\end{cases}
\end{align*}\label{SWmin1}
\end{enumerate}
\end{prop}

%%%%%%%%%%%%%%%%%%%%%%%%%%%%%%%%%%%%%%%%%%%%%%%%%%%%%%%%%%%%%%%%%%%%%%%%%%%%%%%

\section{
	     Formal degrees and local theta correspondence
           }\label{main theorem}

In this section, we state the behavior of the formal degree under the local theta correspondence, which extends the result of Gan and Ichino \cite{GI14}. 
Let $G$ be a connected reductive group over $F$, and let $\pi$ be a square-integrable irreducible representation of $G$. Then, the formal degree is a number $\deg \pi$ satisfying 
\[
\int_{G/A_G}(\pi(g)v_1, v_2)\cdot \overline{(\pi(g)v_3, v_4)} \: dg 
= \frac{1}{\deg \pi} (v_1,v_3) \cdot \overline{(v_2,v_4)}
\]
for $v_1, \ldots, v_4 \in \pi$, where $A_G$ is the maximal $F$-split torus of the center of $G$. 

Again, we consider a right $m$-dimensional $\epsilon$-Hermitian space and a left $n$-dimensional $(-\epsilon)$-Hermitian space. In this section, {\bf we assume that $l=1$}.
The purpose of this section is to describe the behavior of the formal degree under the theta correspondence for the quaternionic dual pair $(G(V), G(W))$. Let $\pi$ be an irreducible square-integrable representation of $G(W)$, and let $\sigma = \theta_\psi(\pi,V)$. Assume that $\sigma\not=0$. Then, we recall that $\sigma$ is also square-integrable. 
\begin{lem}\label{preliminaryFD}
The number
\begin{align}\label{alpha33}
\frac{\deg \pi}{\deg \sigma}\cdot  c_\sigma(-1)\cdot \gamma^V(0, \sigma\times\chi_W, \psi)^{-1}
\end{align}
does not depend on $\pi$.
\end{lem}

We will prove the lemma later (Proposition \ref{alpha23}). We denote the constant \eqref{alpha33} by $\alpha_3(V,W)$.
Now we state our main theorem:

\begin{thm}\label{fd theta1}
We have
\[
\alpha_3(V,W) =\begin{cases}
\epsilon(\frac{1}{2}, \chi_V, \psi)^{-1} & (-\epsilon = 1), \\
\frac{1}{2}\chi_W(-1)^m\epsilon(\frac{1}{2}, \chi_W, \psi)^{-1} & (-\epsilon = -1).
\end{cases}
\]
\end{thm}

We prove Theorem \ref{fd theta1} in later sections. In this section, we see an example:

\begin{egn}
Consider the case where $\epsilon = 1$, $m=1$, $n=2$, and $\chi_W= 1_{F^\times}$. We denote by ${\rm St}$ the Steinberg representation of $G(W)$. Then, it is known that $\theta_\psi({\rm St}, V)$ is the trivial representation $1_{G(V)}$ of $G(V)$. The local Langlands correspondence for $G(W)$ has been established (see \cite[\S5]{Cho17}) and the $L$-parameter of ${\rm St}$ is the principal parameter of $\widehat{G}$ (see e.g. \cite[\S3.3]{GR10}). Then, as representations of $W_F \times \SL_2(\C)$, we have 
\[
{\rm Ad}\circ\phi_0 = (1_{W_F}\otimes r_3) \oplus (1_{W_F}\otimes r_3)
\]
where $1_{W_F}$ is the trivial representation of $W_F$, and $r_3$ is the unique three-dimensional irreducible representation of $\SL_2(\C)$. Thus, we have
\[
\gamma(s+\frac{1}{2}, {\rm St}, {\rm Ad}, \psi) = q^{-4s} \cdot \frac{\zeta_F(-s+\frac{3}{2})^2}{\zeta_F(s+\frac{3}{2})^2}.
\]
Moreover, the centralizer $C_{\phi_0}(\widehat{G})$ of $\Im \phi_0$ in $\widehat{G}$ is $\{\pm 1\}\subset \widehat{G}$, and the component group $\widetilde{S}_{\phi_0}(\widehat{G})$ is abelian. Since the formal degree conjecture for $G(W)$ is available (see \S\ref{FDC} below), we have
\[
\deg {\rm St} = \frac{1}{2}\cdot\frac{q^2}{(1+q^{-1})^2}.
\]
On the other hand, we have
\[
\deg 1_{G(V)} = |G(V)|^{-1} = \frac{q}{1+q^{-1}}.
\]
(Recall that the volume $|G(V)|$ of $G(V)$ is given by Corollary \ref{vol of GW}.) 
Therefore, by Lemma \ref{triv gamma}, we have
\[
\frac{\deg {\rm St}}{\deg 1_{G(V)}} = \frac{1}{2}\cdot \gamma(0, 1_{G(V)}\boxtimes1_{F^\times}, \psi)
\]
which agrees with Theorem \ref{fd theta1}.
\end{egn} 

We explain the strategy of the proof of the theorem.
First, we consider the case where either $W$ or $V$ is anisotropic (i.e. the minimal cases in the sense of the parabolic induction).
In these cases, we can express $\alpha_2(V,W)$ with $\alpha_1(W)$ which is already determined in \S\ref{loc zeta val}. And hence we obtain Proposition \ref{locSW1} (\S\S\ref{Min1}--\ref{Min2}).
Second, we relate $\alpha_3(V,W)$ with $\alpha_2(W)$ (\S\S\ref{gamma under theta}--\ref{loc Rallis inn prod formula}). Then we have Theorem \ref{fd theta1} in the minimal cases.
And finally, we prove that the constant $\alpha_3(V,W)$ is compatible with parabolic inductions (\S\S\ref{planc mes}--\ref{ind_arg}), which completes the proof of Theorem \ref{fd theta1}.
Moreover, once $\alpha_3(V,W)$ is determined, the above processes can be reversed to obtain the general formula for $\alpha_1(W)$ and $\alpha_2(W)$ (\S\ref{det alpha}).

\begin{rem}
As written in \cite[\S5.3]{Kak20}, the definition of the doubling $\gamma$-factor of Lapid and Rallis \cite{LR05} should be modified by a constant multiple. Thus, it is natural to ask whether the statement of the main theorem of \cite{GI14} might change. However, \cite[Theorem 15.1]{GI14} is still true. This is because their proof uses the doubling $\gamma$-factor not to determine the ``constant $\mathcal{C}$'' (see \cite[\S20.2]{GI14}) but to show the existence of the constant $\mathcal{C}$. Hence, the difference of constant multiples is offset at the time of calculation of $\mathcal{C}$.
\end{rem}

%%%%%%%%%%%%%%%%%%%%%%%%%%%%%%%%%%  S12      %%%%%%%%%%%%%%%%%%%%%%%%%%%%%%%%%%%%

\section{
           Minimal cases (I)
            }\label{Min1}

In this section, we prove Proposition \ref{locSW1} \eqref{SWmin1} in the case $\dim V = 2$.

Suppose that $\epsilon = 1,$ $V_0=0$ and $\dim_DV=2$. Then, we can take a basis $\ue^V = (e_1^V, e_2^V)$ of $V$ so that 
\[
(e_1^V, e_1^V) = (e_2^V, e_2^V) = 0, \mbox{ and } (e_1^V, e_2^V) = 1.
\]
We take bases $\ue$ of $W$ and ${\ue'}^\Box$ of $W^\Box$ as in \S\ref{doubling}.
Let $\cL$ be a lattice
\[
\left(\bigoplus_{i=1}^n e_1^V \varpi_D^{-1}\cO_D\otimes e_{n+i}' \right)
\oplus \left(\bigoplus_{i=1}^ne_2^V\cO_D \otimes e_{n+i}'\right)
\]
of $V\otimes W^\trd$, and we denote by $1_\cL$ the characteristic function of $\cL$.
By the fact that $\cL$ is self-dual with respect to the pairing \eqref{pairing_VWtrd}, we have $|\cL| = 1$.

\begin{lem}
We have
\[
\cI(1_\cL, 1_\cL) = q^{-2} \frac{(1-q^{-2})(1+q^{-2})(1+q^{-5})}{1-q^{-3}}.
\]
\end{lem}

\begin{proof}
Let $\cB$ be the subgroup 
\[
\left\{ \begin{pmatrix} a & b \\ c & d \end{pmatrix}\in G(V) \ \middle| \ a,b,d \in \cO_D, c \in \varpi_D\cO_D \right\}
\]
of $G(V)$, which fixes the lattice $\cL$. Then, $\cB$ is an Iwahori subgroup by Lemma \ref{decomp iwahori} and the volume $|\cB|$ is given by $q^{-4}(1-q^{-2})$  by Proposition \ref{Iwa_vol} \eqref{Iwa_vol2}. 
By \cite[Th\'{e}or\`{e}m 5.1.3]{BT72}, we have $G(V) = \cB \cdot\mathcal{N}\cdot\cB$ where $\mathcal{N}$ is the normalizer of the maximal $F$-split torus consisting of the diagonal matrices in $G(V)$. Moreover, we can take a system of representatives 
\[
\{ a(t) \mid t \in \Z \} \cup \{ w(t) \mid t \in \Z \}
\]
for $\cB\backslash G(V)\slash\cB$, where
\[
a(t) = \begin{pmatrix} \varpi_D^t & 0 \\ 0 & (-\varpi_D)^{-t} \end{pmatrix} \mbox{ and } w(t) =  \begin{pmatrix} 0 & \varpi_D^t \\ (-\varpi_D)^{-t} & 0 \end{pmatrix}.
\]
Hence we have
\begin{align*}
\cI(1_\cL, 1_\cL) & = |\cB|\cdot\sum_{t\in \Z} (|\cL \cap a(t)\cL|\cdot[\cB a(t) \cB: \cB] +|\cL\cap w(t)\cL|\cdot[\cB w(t) \cB:\cB]) \\
&=|\cB|\cdot\sum_{t\in \Z}(q^{-3|t|} + q^{-6|t-1|+|1+3t|})\\
&=q^{-4}(1-q^{-2})\cdot (\frac{1+q^{-3}}{1-q^{-3}} + \frac{q^2 + q^{-5}}{1-q^{-3}}) \\
&=q^{-2}\frac{(1-q^{-2})(1+q^{-2})(1+q^{-5})}{1-q^{-3}}.
\end{align*}
Thus we have the lemma.
\end{proof}

\begin{lem}
We have
\[
\cE(1_\cL, 1_\cL) = m^\circ(\frac{1}{2})^{-1}\cdot\alpha_1(W)
\]
where $m^\circ(s)$ is a function as in Lemma \ref{ratio}.
\end{lem}

\begin{proof}
One can show that $1_\cL$ is a $K({\ue'}^\Box)$ fixed function with $1_\cL(0) = 1$. Thus, we have $\cF_{1_\cL} = f_{-\frac{1}{2}}^\circ$ where we mean by $f_s^\circ$ the unique $K({\ue'}^\Box)$ fixed section in $I(-\frac{1}{2}, 1)$ with $f_s^\circ(1)=1$. By the Gindikin-Karperevich formula (see e.g. \cite{Cas89}), we can take $\cF_{1_\cL}^\dagger = m^\circ(\frac{1}{2})^{-1}f_{\frac{1}{2}}^\circ$. Hence, we have
\begin{align*}
\cE(1_\cL, 1_\cL) &= m^\circ(\frac{1}{2})^{-1}\int_{G(W)} f_\rho^\circ(\iota(g,1)) \: dg \\
&= m^\circ(\frac{1}{2})^{-1}\cdot \alpha_1(W).
\end{align*}
\end{proof}

Hence, by the above two lemmas, we have:
\begin{prop}\label{SWprf1}
If $\epsilon=1$, $V_0=0$ and $\dim_DV = 2$, then we have
\[
\alpha_2(V,W) = -|2|_F^{-\frac{15}{2}}\cdot|N(R(\ue))|^{\frac{5}{2}}\cdot q^{-4}\cdot\frac{(1+q^{-1})(1-q^{-4})}{1-q^{-3}}.
\]
\end{prop}

%%%%%%%%%%%%%%%%%%%%%%%%%%%%%%%%%%%%%%%%%%%%%%%%%%%%%%%%%%%%%%%%%%%%%%%%%%

\section{
           Minimal cases (II)
            }\label{Min2}

In this section, we prove Proposition \ref{locSW1} \eqref{SWmin2} and the remaining cases of Proposition \ref{locSW1} \eqref{SWmin1}.

Assume that $V$ is anisotropic. Recall that $\tau \in G(W^\Box)$ denotes the certain Weyl element (see \S\ref{doubling}), and $\fu_n$ is the certain $F$ subspace of ${\rm M}_n(D)$ (see \S\ref{Haar2}). For $\Phi \in \cS(\fu_n)$, we define a section $f(s, \Phi, -) \in I(s,\chi_V)$ by
\[
f(s, \Phi, g) \coloneqq  
\begin{cases}
0 & (g \not\in P(W^\tru)\tau U(W^\tru)), \\
\chi_{V, s+\rho}(\Delta(p))\cdot\Phi(X) & (g=p\tau \begin{pmatrix} 1 & 0 \\ X & 1\end{pmatrix} \in P(W^\tru)\tau U(W^\tru).
\end{cases}
\]
Here $G(W^\Box)$ is embedded in $\GL_{2n}(D)$ by the basis ${\ue'}^\Box$.  
For $t \in \Z, \phi \in \cS(V\otimes W^\trd)$, and $\Phi \in \cS(\fu_n)$, we define $\phi_t \in \cS(V\otimes W^\trd)$, and $\Phi_t \in \cS(\fu_n)$ by 
\[
\phi_t(x) \coloneqq   q^{-4mnt}\phi( x \varpi_F^{t}), \mbox{ and } \Phi_t(X) \coloneqq   q^{-4n\rho t}\Phi(X\varpi^{2t}).
\]
Then we have the following lemma:
\begin{lem}
\begin{enumerate}
\item For $\phi\in \cS(V\otimes W^\trd)$, we have $\widehat{\phi_t} = q^{-4mnt}(\widehat{\phi})_{-t}$. \label{ankey(1)}
\item Let $\phi \in \cS(V\otimes W^\trd)$, and let $\Phi \in \cS(\fu_n)$ such that $M(\frac{1}{2},\chi_V)f(\frac{1}{2}, \Phi, -) = F_\phi$. Then we have
\[
M(\frac{1}{2},\chi_V)f(\frac{1}{2},\Phi_t, -) = q^{-4mnt} F_{\phi_{-t}}.
\] \label{ankey(2)}
\end{enumerate}
\end{lem}

\begin{proof}
We have
\begin{align*}
\widehat{\phi_t}(x) 
&= \int_{V\otimes W^\trd}
     q^{-4mnt} \phi(y\varpi^t)\psi(\frac{\epsilon}{2}\langle\langle x, y\tau\rangle\rangle) \: dy \\
&= \int_{V\otimes W^\trd}
     \phi(y)\psi(\frac{\epsilon}{2}\langle\langle x\varpi^{-t}, y\tau\rangle\rangle)\: dy\\
&= q^{-4mnt}(\widehat{\phi})_{-t}(x).
\end{align*}
Hence we have \eqref{ankey(1)}. 
\begin{align*}
M(\frac{1}{2}, \chi_V)f(\frac{1}{2}, \Phi_t, \tau \begin{pmatrix} 1 & 0 \\ X & 1 \end{pmatrix}) 
& = \int_{\fu} f(\frac{1}{2}, \Phi_t, \tau \begin{pmatrix} 1 & 0 \\ Y & 1 \end{pmatrix}
      \tau \begin{pmatrix} 1 & 0 \\ X & 1 \end{pmatrix}) \: dY \\
& = \int_{\fu} f(\frac{1}{2}, \Phi_t, \begin{pmatrix}Y & 0 \\ -\epsilon & \epsilon \cdot Y^{-1} \end{pmatrix}\tau \begin{pmatrix}1 & 0 \\ -\epsilon Y^{-1} + X & 1\end{pmatrix} ) \: dY\\
& = q^{-4n\rho t} \int_{\fu} \chi_{\frac{1}{2}+\rho}(N(Y))^{-1} f(\frac{1}{2}, \Phi, \tau\begin{pmatrix} 1 & 0 \\ (-\epsilon Y^{-1} + X)\varpi^{2t} & 1\end{pmatrix})) \: dY \\
&= q^{-4mnt} \int_{\fu} \chi_{\frac{1}{2}+\rho}(N(Y))^{-1} f(\frac{1}{2}, \Phi, \tau\begin{pmatrix} 1 & 0 \\ -\epsilon Y^{-1} + X\varpi^{2t} & 1\end{pmatrix})) \: dY \\
&= q^{-4mnt} M(\frac{1}{2}, \chi_V)f(\frac{1}{2}, \Phi, \tau\begin{pmatrix} 1 & 0 \\ X\varpi^{2t} & 1 \end{pmatrix}) \: dY \\
&= q^{-4mnt} F_\phi(\tau \begin{pmatrix} 1 & 0 \\ X\varpi^{2t} & 1 \end{pmatrix}) \\
&= q^{-4mnt} \beta_V(\tau) \int_{V\otimes W^\trd}
  \phi(x)\psi(\frac{1}{4}\langle\langle x, x\begin{pmatrix} 1 & 0 \\ X\varpi^{2t} & 1 \end{pmatrix} \rangle\rangle)
  \: dx \\
&= \beta_V(\tau)\int_{V\otimes W^\trd}
  \phi(x\varpi^{-t})\psi(\frac{1}{4}\langle\langle x, x\begin{pmatrix} 1 & 0 \\ X & 1 \end{pmatrix} \rangle\rangle)
  \: dx \\
&= q^{-4mnt} F_{\phi_{-t}}(\tau\begin{pmatrix} 1 & 0 \\ X & 1 \end{pmatrix}).
\end{align*}
Hence we have \eqref{ankey(2)}.
\end{proof}

\begin{prop}\label{anI}
Let $\phi, \phi' \in \cS(V\otimes W^\trd)$. Then, for sufficiently large $t\in \Z$, we have
\[
\cI(\phi_t, \phi') = (-1)^{mn}\chi_V(-1)^nq^{-4mnt}|G(V)|F_{\phi}(1)\overline{F_{\phi'}(\tau)}.
\]
\end{prop}

\begin{proof}
The Fourier transform on the  space $\cS(V\otimes W^\trd)$ is given by the action of the Weyl element $\tau$ of $G(W^\Box)$. Hence we have
\begin{align*}
\cI(\phi_t, \phi') &= \cI(\widehat{\phi_t}, \widehat{\phi'}) \\
&= q^{-4mnt}\cI((\widehat{\phi})_{-t}, \widehat{\phi'}) \\
&= q^{-4mnt}\int_{G(V)} 
  ((\widehat{\phi})_{-t}, \overline{\omega_\psi^\Box(h)\widehat{\phi'}}) \: dh.
\end{align*}
When $t$ is sufficiently large, the support of $(\widehat{\phi})_{-t}$ is sufficiently small. Hence this integral is 
\begin{align*}
 & q^{-4mnt}|G(V)| \widehat{(\widehat{\phi})_{-t}}(0) \overline{\widehat{\phi'}(0)}\\
 &= q^{-4mnt}|G(V)| q^{4mnt}(\widehat{\widehat{\phi}})_t(0)\overline{\widehat{\phi'}(0)} \\
 &= q^{-4mnt}|G(V)| \phi_t(0)\overline{\widehat{\phi'}(0)}\\
 &= q^{-4mnt}\beta_V(\tau)|G(V)|F_{\phi}(1)\overline{F_{\phi'}(\tau)}.
\end{align*}
Hence we have the proposition.
\end{proof}

\begin{prop}\label{anE}
Let $\phi, \phi' \in \cS(V\otimes W^\trd)$. Then, for sufficiently large $t\in \Z$, we have
\[
\cE(\phi_t, \phi') = m^\circ(\rho)^{-1}\alpha_1(W) q^{-4mnt}F_{\phi}(1)\overline{F_{\phi'}(\tau)}.
\]
\end{prop}

\begin{proof}
When $t$ is sufficiently large, the support of $\Phi_{-t}$ is sufficiently small. Then, by using Lemma \ref{ratio}, we have
\begin{align*}
\cE(\phi_t, \phi') 
&= q^{-4mnt} \int_{G(W)} f(\frac{l}{2}, \Phi_{-t}, (g,1))\overline{F_{\phi'}(g,1)} \: dg \\
&=m^\circ(\rho)^{-1}\alpha_1(W) q^{-4mnt} \int_{U(W^\tru)} f(\frac{l}{2}, \Phi_{-t}, \tau u)\overline{F_{\phi'}(\tau u)} \: du \\
&= m^\circ(\rho)^{-1}\alpha_1(W) q^{-4mnt} \left( \int_{\fu_n}\Phi_{-t}(X)\: dX\right) \overline{F_{\phi'}(\tau)}\\
&= m^\circ(\rho)^{-1}\alpha_1(W) q^{-4mnt} F_{\phi}(1)\overline{F_{\phi'}(\tau)}.
\end{align*}
Hence we have the proposition.
\end{proof}

By Propositions \ref{anI} and \ref{anE}, we have the following:
\begin{prop}\label{aaaaaa} If $V$ is anisotropic, then we have
\[
\alpha_2(V,W) = |G(V)|\cdot m^\circ(\rho)\cdot \alpha_1(W)^{-1}.
\]
\end{prop}

By substituting the values of $|G(V)|$ (Corollary \ref{vol of GW}) and $\alpha_1(W)$ (Proposition \ref{alpha1 min}) for Proposition \ref{aaaaaa}, we obtain Proposition \ref{locSW1} \eqref{SWmin1} and Proposition \ref{locSW1} \eqref{SWmin2} with $V$ anisotropic. Thus, we finish the proof of Proposition \ref{locSW1}.

%%%%%%%%%%%%%%%%%%%%%%%%%%%%%%%%%%    S14     %%%%%%%%%%%%%%%%%%%%%%%%%%%%%%%%%%

\section{
	     The behavior of the $\gamma$-factor under the local theta correspondence
	     }\label{gamma under theta}

The purpose of this section is to explain the behavior of the $\gamma$-factor under the local theta correspondence, which extends \cite[Theorem 11.5]{GI14}. 
Let $V$ be a right $\epsilon$-Hermitian space of dimension $m$, and let $W$ be a left $(-\epsilon)$-Hermitian space of dimension $n$. In this section, we allow $D$ to be split and $l$ not to be $1$.

\begin{thm}\label{gamma_main}
Let $\pi$ be an irreducible representation of $G(W)$ and let $\omega$ be a character of $F^\times$. We denote $\sigma = \theta(\pi, V)$ and we assume $\sigma \not=0$.
\begin{enumerate}
\item If $l > 0$, then we have
\[
\frac{\gamma^V(s,\sigma\times\omega\chi_V,\psi)}{\gamma^W(s,\pi\times\omega\chi_W,\psi)}
= \prod_{i=1}^l \gamma_F(s+\frac{l+1}{2}-i, \omega\chi_V\chi_W,\psi)^{-1}.
\]
\item If $l < 0$, then we have
\[
\frac{\gamma^V(s,\sigma\times\omega\chi_V,\psi)}{\gamma^W(s,\pi\times\omega\chi_W,\psi)}
= \prod_{i=1}^{-l} \gamma_F(s+\frac{-l+1}{2}-i, \omega\chi_V\chi_W,\psi).
\]
\end{enumerate}
\end{thm}

The proof of Theorem \ref{gamma_main} consists of four subsections (\S\S\ref{ma_gamma}--\ref{ur_gamma}). In the first three subsections, we reduce the theorem to the unramified cases by using properties of the doubling $\gamma$-factor.  
In the last subsection, we discuss the unramified cases to finish the proof of the theorem.

\subsection{Multiplicative argument}\label{ma_gamma}

We put 
\[
f_D(s, V,W,\omega,\psi) = \begin{cases}
\prod_{i=1}^l \gamma_F(s+\frac{l+1}{2}-i, \omega\chi_V\chi_W,\psi)^{-1} & (l>0), \\
\prod_{i=1}^{-l} \gamma_F(s+\frac{-l+1}{2}-i, \omega\chi_V\chi_W,\psi) & (l < 0),
\end{cases}
\]
and we put 
\[
e_{D}(s, V,W,\pi,\omega,\psi) 
= \frac{\gamma^W(s,\sigma\times\omega\chi_W,\psi)}{\gamma^V(s,\pi\times\omega\chi_V,\psi)} \cdot f_D(s, V,W,\omega,\psi).
\]
Then, Theorem \ref{gamma_main} is equivalent to $e_D(s,V,W,\pi,\omega,\psi) =1$. When $D$ is split, this equality has been already proved (\cite[Theorem 11.5]{GI14}).

Let $\{V_p\}_{p \geq 0}$ and $\{W_q\}_{q \geq 0}$ be Witt towers containing $V$ and $W$ respectively. Put $V = V_r, W = W_t, m_0= \dim_D V_0, m_0 = \dim_DW_0$, and $l_p = l_{V_p,W}$.
We denote by $\cJ_q(\pi)$ the set of the $G(W_0)$-part of the non-zero irreducible quotients of the Jacquet modules $J_P(\pi)$ for all parabolic subgroup $P$ whose Levi subgroups contain $G(W_q)$ as a direct factor.  
We first state the multiplicativity:

\begin{lem}\label{mult3}
We denote by $r(\pi)$ the first occurrence index of $\pi$ (see \S\ref{tower properties}). Suppose that $\cJ_q(\pi)\not=\varnothing$ and $p \geq r(\pi)$. Then, for an irreducible representation $\pi_q \in \cJ_q(\pi)$, we have
\[
e_D(V_p, W_q, \pi_q, \omega,\psi) = e_D(V,W,\pi,\omega,\psi).
\]
\end{lem}

\begin{proof}
First, we consider the case $q=t$ and $\pi_q = \pi$. We may assume that $r = r(\pi)$.
Put $\sigma = \theta_\psi(\pi, V)$ and $\sigma' = \theta_\psi(\pi, V_p)$. Then, by Proposition \ref{ttoowweerr} \eqref{11111}, we have
\begin{align*}
&\gamma^{V_p}(s, \sigma' \times \omega, \psi) 
\cdot \gamma^V(s, \sigma \times\omega, \psi)^{-1} \\
& = \gamma_{\GL_{p-r}(D)}^{GJ}(s+(\frac{l_p}{2}+p-r), \omega,\psi) \cdot \gamma_{\GL_{p-r}(D)}^{GJ}(s-(\frac{l_p}{2}+p-r), \omega, \psi) \\
&= \prod_{i=1}^{p-r} \gamma_{D^\times}^{GJ}(s+\frac{l_p}{2} + (2i-1), \omega,\psi)
     \cdot \prod_{i=1}^{p-r} \gamma_{D^\times}^{GJ}(s- (\frac{l_p}{2} + (2i-1), \omega,\psi) \\
& = \prod_{i=1}^{2(p-r)} \gamma_F(s+\frac{l_p - 1}{2} + i, \omega,\psi)
     \cdot \prod_{i=1}^{2(p-r)} \gamma_F(s + \frac{-l_p +1}{2}-i, \omega,\psi) \\
& = f_D(V_p,W,\pi,\omega,\psi) f_D(V,W,\pi,\omega,\psi)^{-1}.
\end{align*}
Here, $\gamma_{\GL_u(D)}^{GJ}(s, \omega, \psi)$ is the $\gamma$-factor defined by
\[
\epsilon_{\GL_u(D)}^{GJ}(s, \omega, \psi)\frac{L_{\GL_u(D)}^{GJ}(1-s, \omega^{-1})}{L_{\GL_u(D)}^{GJ}(s, \omega)}
\]
where $\epsilon_{\GL_u(D)}^{GJ}(s, - , \psi)$ and $L_{\GL_u(D)}^{GJ}(s, -)$ are $\epsilon$-and $L$-factors defined in \cite{GJ72}, and $\omega$ denotes the composition $\omega\circ N$ of $\omega$ with the reduced norm $N$ of $\GL_u(D)$.
Thus we have
\[
e_D(V_p,W,\pi,\omega,\psi) = e_D(V,W,\pi,\omega,\psi).
\]

Second, we consider the general case. Put 
\[
t(\pi_q) = \min\{ q' = 0, \ldots, q \mid \cJ_{q'}(\pi_q) \not=\varnothing \}.
\]
Then, any $\pi_{t(\pi)} \in \cJ_{t(\pi_q)}(\pi_q)$ is supercuspidal. Take a positive integer $p'$ so that $p' \geq \max\{r+q-t, r(\pi) +q-t(\pi)\}$. Then, by the first part of this proof, we have
\[
e_D(s, V_p, W_q, \pi_q, \omega,\psi) = e_D(s, V_{p'}, W_q, \pi_q, \omega, \psi).
\]
Moreover, by using Proposition \ref{ttoowweerr} \eqref{22222} repeatedly, we can show that
\[
e_D(s, V_{p'}, W_q, \pi_q, \omega, \psi) = e_D(s, V_{p'-(q-t(\pi))},W_{t(\pi)},\pi_{t(\pi)},\omega,\psi).
\]
By tracing the above discussions conversely, the right-hand side is equal to
\[
e_D(s, V_{p' + (t-q)}, W, \pi, \omega, \psi) = e_D(s, V, W, \pi, \omega, \psi).
\]
Thus we have the lemma.
\end{proof}

\subsection{Global argument}\label{ga1_gamma}

In this subsection, we explain the global argument which we use in the proof of Theorem \ref{gamma_main}. 

\begin{lem}\label{ga1_gamma_lemma}
Let $\F$ be a number field, let $\A$ be the ring of its adeles, let $\D$ be a division quaternion algebra over $\F$, let $\underline{V}$ be a right $\epsilon$-Hermitian space over $\D$, let $\underline{W}$ be a left $(-\epsilon)$-Hermitian space over $\D$, let $\Pi$ be an irreducible cuspidal automorphic representation of $G(\underline{W})(\A)$, let $\underline{\omega}$ be a Hecke character of $\A^\times/\F^\times$, and let $\underline{\psi}$ be a non-trivial additive character of $\A/\F$. Then, we have
\[
\prod_{v} e_{\D_v}(s, \underline{V}_v, \underline{W}_v, \Pi_v, \underline{\omega}_v, \underline{\psi}_v)=1.
\]
\end{lem}

\begin{proof}
Consider the Witt tower $\{\underline{V}_p\}_{p=0}^\infty$ so that $\underline{V}_r = \underline{V}$. Denote by $r(\Pi)$ the first occurrence index of $\Pi$ in $\{\underline{V}_p\}_{p=0}^\infty$, by $\Sigma$ the theta correspondence $\theta(\Pi, \underline{W}_{r(\Pi)})$ of $\Pi$, and by $S$ the set of the places where $\D_v$ is a division algebra. Then, we have $\theta_{\underline{\psi}}(\Pi, \underline{V}_{r(\Pi)})$ is cuspidal, and we have
\begin{align*}
\prod_{v} e_{\D_v}(s, \underline{V}_v, \underline{W}_v, \Pi_v, \underline{\omega}_v, \underline{\psi}_v)
&= \prod_{v\in S} e_{\D_v}(s, {\underline{V}_{r(\Pi)}}_v, \underline{W}_v, \Pi_v, \underline{\omega}_v, \underline{\psi}_v) \\
&= \prod_{v \in S} \frac{\gamma^V(s, \Sigma\boxtimes\underline{\omega}\chi_{\underline{W}}, \underline{\psi})}{\gamma^W(\pi, \Pi\boxtimes\underline{\omega}\chi_{\underline{V}}, \underline{\psi})}\cdot f_{\D_v}(s, \underline{V}, \underline{W}, \underline{\omega}, \underline{\psi})\\
&\ \ \times \frac{L^S(s, \Sigma\boxtimes\underline{\omega}\chi_{\underline{W}})L_f^S(s)}{L^S(s, \Pi\boxtimes\underline{\omega}\chi_{\underline{V}})} \cdot
\frac{L^S(1-s, \Pi\boxtimes\underline{\omega}\chi_{\underline{V}})}{L^S(1-s, \Sigma\boxtimes\underline{\omega}\chi_{\underline{W}})L_f^S(1-s)} \\
&= 1
\end{align*}
where $L_f^S(s) = \prod_{v \not\in S}L_{f, v}(s)$ with 
\[
L_{f,v}(s) =  \begin{cases} \prod_{i=1}^lL_{\F_v}(s+\frac{l+1}{2}-i, \underline{\omega}_v\chi_{\underline{V}_v}\chi_{\underline{W}_v}) & (l > 0), \\
\prod_{i=1}^{-l}L_{\F_v}(s+\frac{-l+1}{2}-i, \underline{\omega}_v\chi_{\underline{V}_v}\chi_{\underline{W}_v})^{-1} & (l < 0).
\end{cases}
\] 
Hence we have the lemma.
\end{proof}

\subsection{Globalization}\label{ga2_gamma}

\begin{lem}\label{Heckegl}
Let $\F$ be a global field of characteristic zero, let $v_1, \ldots, v_d$ be places of $\F$, and let $\omega_1, \ldots, \omega_d$ be unitary characters of $\F_{v_1}^\times, \ldots, \F_{v_d}^\times$ respectively.
\begin{enumerate}
\item Suppose that $\omega_i$ is trivial for $i=2, \ldots, d$. Then, there exists a Hecke character $\underline{\omega}$ of $\A^\times$ so that $\underline{\omega}_j = \omega_{v_j}$ for $j=1, \ldots, d$. \label{Hecke1}
\item Suppose that $\F_{v_1} = \cdots = \F_{v_d}$ and $\omega_1 = \cdots = \omega_d$. Then, there exists a Hecke character $\underline{\omega}$ of $\A^\times/\F^\times$ so that $\underline{\omega}_j = \omega_{v_j}$ for $j=1, \ldots, d$. \label{Hecke2}
\end{enumerate}
\end{lem}

\begin{proof}
Let $\eta_j$ be the character of $\Gal(\F_{v_j}^s/\F_{v_j})$ associated with $\omega_j$ via the local class field theory for $j=1, \ldots, d$.
First, suppose that $\omega_i$ is trivial for $i=2, \ldots, d$. Take a finite Galois extension $\mathbb{L}$ of $\F$ so that $\ker\eta_1 = \Gal(\F_{v_1}^s/\mathbb{L}_{w_1})$ and $\mathbb{L}_{w_i} = \F_{v_i}$ for $i=2, \ldots, d$. Here, $w_1, \ldots, w_d$ are some places of $\mathbb{L}$ lying above $v_1, \ldots, v_d$ respectively.
Take a character $\widetilde{\eta}$ of $\Gal(\mathbb{L}/\F)$ so that $\widetilde{\eta}|_{\Gal(\mathbb{L}_{w_1}/\F_{v_1})} = \eta$, and define $\underline{\omega}$ as the Hecke character of $\A^\times$ associated with $\widetilde{\eta}$ via the global class field theory.
Then we have $\underline{\omega}_{v_1} = \omega_1$ and $\underline{\omega}_{v_i} = 1_{F_{v_i}^\times}$ for $i=2, \ldots d$, and thus we have \eqref{Hecke1}.

Then, suppose that $\F_{v_1} = \cdots = \F_{v_d}$ and $\omega_1 = \cdots = \omega_d$. By \eqref{Hecke1}, there exists a Hecke character $\underline{\chi}$ of $\A^\times$ so that $\underline{\chi}_{v_1} = \omega_1$ and $\underline{\chi}_{v_i} = 1_{\F_{v_i}^\times}$ for $i=2, \ldots, d$. Let $\underline{\omega}$ be the continuous unitary character of $\A^\times$ given by
\[
\underline{\omega}_v = \begin{cases} \omega_1 & (v = v_1, \ldots, v_d), \\ \underline{\chi}_v^d & (\mbox{otherwise}). \end{cases}
\]
Then we have $\underline{\omega}(\F^\times) = 1$ and it is a Hecke character of $\A^\times$ satisfying $\underline{\omega}_{v_j} = \omega_j$ for $j = 1, \ldots, d$. Thus we have \eqref{Hecke2}, and we finish the proof of Lemma \ref{Heckegl}.
\end{proof}

Let $\psi$ be a unitary non-trivial additive character of $F$. For $a \in F^\times$ we denote by $\psi_a$ the character of $F$ defined by $\psi_a(x) = \psi(ax)$ for $x \in F$.
\begin{lem}\label{globalization}
Assume that $D$ is a division quaternion algebra.
Let $F'$ be a non-Archimedean local field of characteristic zero, let $\psi'$ be an additive non-trivial character of $F'$, let $D'$ be a division quaternion algebra over $F'$, let $V'$ be another right $\epsilon$-Hermitian space of dimension $m$, and let $W'$ be another left $(-\epsilon)$-Hermitian space of dimension $n$.
Then, there exist
\begin{itemize}
\item a global field $\F$ and its places $v_1, v_2$ such that $\F_{v_1} = F, \F_{v_2} = F'$,
\item a division quaternion algebra $\D$ over $\F$ such that $\D_{v_1} = D$, $\D_{v_2} = D'$, and $\D_v$ is split for $v \not= v_1, v_2$,
\item a left $(-\epsilon)$-hermitian spaces $\underline{W}$ over $\D$ such that $\underline{W}_{v_1} = W, \underline{W}_{v_2} =W'$, 
\item a right $\epsilon$-hermitian space $\underline{V}$ over $\D$ such that $\underline{V}_{v_1} = V, \underline{V}_{v_2} = V'$, 
\item a Hecke character $\underline{\omega}$ of $\A^\times$ such that $\underline{\omega}_{v_1} = \omega, \underline{\omega}_{v_2} = 1_{F_{v_2}^\times}$,
\item a non-trivial additive character $\underline{\psi}$ of $\A/\F$ such that $\underline{\psi}_{v_1} = \psi_{a_1^2}, \underline{\psi}_{v_2} = \psi_{a_2^2}'$ for some $a_1 \in F^\times, a_2 \in {F'}^\times$.
\end{itemize} 
\end{lem} 

\begin{proof}
The existences of such $\F$ and $\D$ are well-known. The existences of such $\underline{W}$, $\underline{V}$ and $\underline{\psi}$ are due to the weak approximation. It remains to show the existence of $\underline{\omega}$. 
Let $\eta$ be the character of $\Gal(F^s/F)$ associated with $\omega$ via the local class field theory, and let $L$ be the Galois extension field of $F$ so that $\ker\eta = \Gal(F^s/L)$. Here, $F^s$ denotes the separable closure of $F$. Then there exists a Galois extension field $\mathbb{L}$ of $\F$ such that $\mathbb{L}_{w_1} = L$ and $\mathbb{L}_{w_2} = F'$ where $w_1$ (resp. $w_2$) is some place of $\mathbb{L}$ lying above $v_1$ (resp. $v_2$). Take a character $\widetilde{\eta}$ of $\Gal(\mathbb{L}/\F)$ so that $\widetilde{\eta}|_{\Gal(L/F)} = \eta$, and define $\underline{\omega}$ as the Hecke character of $\F$ associated with $\widetilde{\eta}$ via the global class field theory. Then, we have $\underline{\omega}_{v_1} = \omega$. Moreover, by $\mathbb{L}_{w_2} = F_{v_2}$, we have $\ker \widetilde{\eta} \supset \Gal({F'}^s/F')$ which implies that $\underline{\omega}_{v_2} = 1_{F_{v_2}^\times}$. Hence we have the lemma.
\end{proof}

Moreover, by using the Poincare series, we obtain a globalization lemma of representations.

\begin{lem}\label{glob rep}
Let $\F$ be a global field, let $G$ be a reductive group over $\F$, let $Z$ be the center of $G$, let $A$ be a maximal $\F$-split torus of $Z$, let $\underline{\chi}$ be a unitary character $A(\A)/A(\F)\rightarrow \C^\times$, let $S_0$ be a non-empty set of places of $\F$ such that all Archimedean places are contained in $S_0$, let $S$ be a finite set of  places of $\F$ such that $S_0 \cap S =\varnothing$. Suppose that $Z(\F_v)/A(\F_v)$ is compact for $v \in S$,  an irreducible supercuspidal representation $\pi_v$ of $G(\F_v)$ so that $\pi_v|_{A(\F_v)}$ coincides with $\underline{\chi_v}$  is given for each $v \in S$, and a compact open subgroup $K_v$ is given for each $v \not\in S_0 \cup S$.  
Then there is an irreducible cuspidal automorphic representation $\Pi$ of $G(\A)$ such that
\begin{itemize}
\item $\Pi|_{A(\A)}$ coincides with $\underline{\chi}$,
\item $\Pi_v \cong \pi_v$ for $v \in S$,
\item $\Pi_v$ possesses a non-zero $K_v$-fixed vector for $v \not\in S_0 \cup S$.
\end{itemize}
\end{lem}

\begin{proof}
Denote by $\omega_v$ the central character of $\pi_v$ for each $v\in S$. Since $\prod_{v\in S}Z(\F_v)/A(\F_v)$ is compact, we have that $(\prod_{v \in S}Z(\F_v))\cdot A(\A)$ is a closed subgroup of $Z(\A)$. Moreover, we have that $(\prod_{v \in S}Z(\F_v)) A(\A)/A(F)$ can be identified with a closed subgroup of $Z(\A)/A(\F)$. Consider a character $\widetilde{\underline{\chi}}$ on $(\prod_{v \in S}Z(\F_v)) A(\A)/A(F)$ given by
\[
\widetilde{\underline{\chi}}((z_v)_{v\in S}, a) = (\prod_{v\in S}\omega_v(z_v)) \cdot \underline{\chi}(a)
\]
for $z_v \in Z(\F_v)$ ($v \in S$) and $a \in A(\A)$. Then, this character can be extended to a character $\underline{\omega}$ on $Z(\A)/A(F)$. Hence, by \cite[Appendice I]{Hen84}, there exists an irreducible cuspidal automorphic representation $\Pi$ of $G(\A)$ so that $\Pi|_{Z(\A)} = \underline{\omega}$, $\Pi_v \cong \pi_v$ for $v \in S$ and $\Pi_v$ possesses a non-zero $K_v$-fixed vector for each $v \not\in S_0\cup S$. Thus we have the lemma.
\end{proof}

\subsection{Completion of the proof of Theorem \ref{gamma_main}}\label{ur_gamma}

Let $\pi$ be an irreducible representation of $G(W)$, and let $\omega$ be a character of $F^\times$. By Lemma \ref{mult3} and Corollary \ref{reducing}, we may assume that $\pi$ and $\sigma$ are irreducible supercuspidal representations. Take
\begin{itemize}
\item a global field $\F$ and its places $v_1, v_2$ such that $\F_{v_1} = F, \F_{v_2} = F$,
\item a division quaternion algebra $\D$ over $\F$ such that $\D_{v_1} = D$, $\D_{v_2} = D$, and $\D_v$ is split for $v \not= v_1, v_2$,
\item a left $(-\epsilon)$-hermitian spaces $\underline{W}$ over $\D$ such that $\underline{W}_{v_1} = W$, the dimension of the anisotropic kernel of $\underline{W}_{v_2}$ is $0$ or $1$, and $\fd(\underline{W}_{v_2}) \in \cO_{\F_{v_2}}^\times$, 
\item a right $\epsilon$-hermitian space $\underline{V}$ over $\D$ such that $\underline{V}_{v_1} = V$, the dimension of the anisotropic kernel of $\underline{V}_{v_2}$ is $0$ or $1$, and $\fd(\underline{V}_{v_2}) \in \cO_{\F_{v_2}}^\times$, 
\item a non-trivial additive character $\underline{\psi}$ of $\A/\F$ such that $\underline{\psi}_{v_1} = \psi_{a}$ for some $a \in F^\times$.
\end{itemize} 
Moreover, we can take a Hecke character $\underline{\omega}$ of $\A^\times$ such that $\underline{\omega}_{v_1} = \omega$ and $\underline{\omega}_{v_2} = 1_{F_{v_2}^\times}$ (Lemma \ref{Heckegl}\eqref{Hecke1}). Denote by $\{ \underline{V}_i\}_{i=0}^\infty$ the Witt tower containing $\underline{V}$. 
Let $K_{v_2}$ be the maximal compact subgroup fixing $0$ of the apartment $E$ of $G(\underline{W}_{v_2})$.
Then, by Lemma \ref{glob rep}, we can take an irreducible cuspidal automorphic representation $\Pi$ of $G(\underline{W})(\A)$ so that $\Pi_{v_1} = \pi$, and $\Pi_{v_2}$ possesses a non-zero $K_{v_2}$ fixed vector. Hence, by Lemma \ref {ga1_gamma_lemma}, we have

\begin{align}\label{eD(s)}
\notag e_D(s, V, W, \pi, \omega, \psi) 
&= \prod_{v\not=v_1} e_{\D_v}(s, \underline{V}, \underline{W}, \Pi, \underline{\omega}, \underline{\psi})^{-1} \\ 
&= e_{D}(s, \underline{V}_{v_2}, \underline{W}_{v_2}, \Pi_{v_2}, 1_{F_{v_2}^\times}, \underline{\psi}_{v_2})^{-1}.
\end{align}

Denote by $\psi'$ the localization $\underline{\psi}_{v_2}$ and by $W_0'$ the anisotropic kernel of $\underline{W}_{v_2}$. Then, we have $t(\Pi_{v_2}) = 0$ and $1_{G(W_0')} \in \cJ_0(\Pi_{v_2})$. Here, $1_{G(W_0')}$ is the trivial representation of $G(W_0')$. Hence, \eqref{eD(s)} is equal to
\[
e_{D}(s, (\underline{V}_p)_{v_2}, W_0', 1_{G(W_0')}, 1_{F_{v_2}^\times}, \psi')^{-1}
\]
where $p$ is a sufficiently large integer so that $\Theta_{\psi'}(1_{G(W_0')}, (\underline{V}_p)_{v_2})\not=0$. By the above observation, it only suffices to consider the cases where $n=\dim W = 0,1$ and $\pi = 1_{G(W)}$.

\begin{lem}
We denote by $1_{G(V)}$ (resp.~ $1_{G(W)}$) the trivial representation of $G(V)$ (resp.~ $G(W)$).
Suppose that $n=0$. Then we have $r(1_{G(W)}) = 0$ and $\theta_\psi(1_{G(W)}, V) = 1_{G(V)}$.
\end{lem}

For the rest of this subsection, we consider the case $n=1$.
In this case, we consider the accidental isomorphism:
\begin{align}\label{accidental_isom}
G(V) \cong {\rm SU}_E(2), \mbox{ and } G(W) \cong {\rm U}_E'(1).
\end{align}
Here, 
\begin{itemize}
\item $E$ is the quadratic unramified extension field of $F$ 
         associated with the quadratic character $\chi_W$ of $F^\times$,    
\item ${\rm SU}_E(2)$ is the special unitary group preserving the hermitian form
        \[
         ( \ , \ )_E\colon E^2 \times E^2 \rightarrow E\colon 
         (\begin{pmatrix} x_1 \\ x_2 \end{pmatrix}, \begin{pmatrix} y_1 \\ y_2 \end{pmatrix})
         \mapsto \overline{x_1}y_1 - \varpi_F \cdot \overline{x_2}y_2
         \]
        where $\overline{x_i}$ denotes the conjugate of $x_i$ with respect to $E/F$, 
\item ${\rm U}_E'(1)$ is the unitary group preserving the skew-hermitian form
         \[
         \langle \ , \ \rangle_E\colon E \times E \rightarrow E\colon x,y \mapsto x \alpha \overline{y}
         \]
         where $\alpha \in E$ is a non-zero trace zero element with $\ord_E(\alpha) = 0$.
\end{itemize}
In particular, for these groups, the L-parameters are defined.

\begin{prop}\label{n=m=1_gamma}
Suppose that $n=m=1$ and $\epsilon = 1$. Let $\pi$ be a non-trivial irreducible representation of $G(W)$, and let $\phi$ be its $L$-parameter. Then, the representation $\Theta_\psi(\pi,V)$ is non-zero irreducible, and has $L$-parameter 
\[
 (\phi \otimes \chi_V\chi_W) \oplus \chi_W.
\]
\end{prop}

\begin{proof}
By \cite[\S7]{Ike19}, the accidental isomorphisms \eqref{accidental_isom} are compatible with the local theta correspondences. We know the description of the local theta correspondence
\[
\Irr ({\rm U}_E'(1)) \rightarrow \Irr ({\rm U}_E(2))
\]
via $L$-parameters (\cite[Theorem 4.4]{GI16}). 
Therefore, we have the claim.
\end{proof}

By tracing the converse of the global argument at the beginning of this subsection, we obtain:

\begin{cor}\label{n=1,e=1}
Suppose that $n=1$ and $\epsilon = 1$. Denote by $\{V_i\}_{i=0}^\infty$ the Witt tower containing $V$. Then we have $e_D(s, V_p, W, 1_{G(W)},1_{F^\times},\psi) = 1$ for sufficiently large $p$.
\end{cor}

Similarly, by using the accidental isomorphism, we have:
\begin{lem}
Suppose that $n=1$ and $\epsilon=-1$. Denote by $\{V_i\}_{i=0}^\infty$ the Witt tower containing $V$. Then we have $e_D(s, V_p, W, 1_{G(W)},1_{F^\times},\psi) = 1$ for sufficiently large $p$.
\end{lem}

Hence, we complete the proof of Theorem \ref{gamma_main}.

%%%%%%%%%%%%%%%%%%%%%%%%%%%%%%%%%%   S15     %%%%%%%%%%%%%%%%%%%%%%%%%%%%%%%%%%

\section{
	    The local analogue of the Rallis inner product formula
	     }\label{loc Rallis inn prod formula}

In this section, we discuss the local analogue of the Rallis inner product formula following \cite{GI14}, and describe the relation between $\alpha_2(V,W)$ and $\alpha_3(V,W)$. 

We use the setting of \S\ref{set and not}, and we take a basis $\ue$ of $W$ as in \S\ref{basis for WV}. Suppose that $l=1$ and $\pi$ is an irreducible square-integrable representation of $G(W)$. Consider the map 
\[
\cP\colon \omega_\psi\otimes\overline{\omega_\psi}\otimes\overline{\omega_\psi}\otimes\omega_\psi\otimes \overline{\pi} \otimes\pi\otimes\pi\otimes\overline{\pi} \rightarrow \C
\]
defined by
\begin{align*}
&\cP(\phi_1, \phi_2, \phi_3, \phi_4; v_1, v_2, v_3, v_4) \\
&= \int_{G(V)} (\sigma(h)\theta(\phi_1, v_1), \theta(\phi_2, v_2)) \cdot
\overline{(\sigma(h)\theta(\phi_3, v_3), \theta(\phi_4, v_4))} \ dh.
\end{align*}
The integral defining $\cP$ converges absolutely (Lemma \ref{sqrrrrr}).
As in \cite[\S18]{GI14}, we compute $\cP$ in two ways. First, we have
\begin{align*}
&\cP(\phi_1 \ldots,\phi_4, v_1, \ldots, v_4) \notag \\
&= \frac{1}{\deg \sigma} \cdot(\theta(\phi_1, v_1), \theta(\phi_3, v_3))\cdot\overline{(\theta(\phi_2,v_2), \theta(\phi_4,v_4))} \notag \\
&= \frac{1}{\deg \sigma}\cdot Z(-\frac{1}{2}, F_{\phi_1\otimes\overline{\phi_3}}, \overline{\xi}_{v_1,,v_3}) \cdot \overline{Z(-\frac{1}{2}, F_{\phi_2\otimes\overline{\phi_4}}, \overline{\xi}_{v_2, v_4})}.
\end{align*}
Second, changing the order of integrals and using Proposition \ref{inner prods}, we have
\begin{align*}
&\cP(\phi_1 \ldots,\phi_4, v_1, \ldots, v_4) \notag \\
&=\int_{G(W)}\int_{G(W)} \left( \int_{G(V)} (\omega_\psi(gh)\phi_1, \phi_2)\overline{(\omega_\psi(g'h)\phi_3, \phi_4)} \: dh \right) \\
& \qquad (\pi(g)v_1, v_2) \overline{(\pi(g')v_3, v_4)} \:  dgdg'\\
&= |2|_F^{2mn}\cdot|N(R(\ue))|^{-m}\cdot\int_{G(W)}\int_{G(W)} \\
& \qquad \cI(\omega_\psi(g)\phi_1\otimes\omega_\psi(g')\phi_3, \phi_2\otimes\phi_4)\cdot (\pi(g)v_1, v_2) \overline{(\pi(g')v_3, v_4)} \: dgdg' \\
&=\alpha_2(V,W) \cdot |2|_F^{2mn}\cdot |N(R(\ue))|^{-m} \int_{G(W)}\int_{G(W)} \\
& \qquad \cE(\omega_\psi(g)\phi_1\otimes\omega_\psi(g')\phi_3, \phi_2\otimes\phi_4)\cdot (\pi(g)v_1, v_2) \overline{(\pi(g')v_3, v_4)} \: dgdg'. 
\end{align*}
Substituting the definition of $\cE$, the expression is equal to
\begin{align*}
&\alpha_2(V,W) \cdot |2|_F^{2mn}\cdot |N(R(\ue))|^{-m} \int_{G(W)}\int_{G(W)}\int_{G(W)} \\
& \quad F_{\phi_1\otimes\phi_3}^\dagger(\iota({g'}^{-1}g''g, 1))\cdot\overline{F_{\phi_2\otimes\phi_4}(\iota(g'', 1))}\cdot (\pi(g)v_1, v_2) \overline{(\pi(g')v_3, v_4)} \: dg''dgdg' \\
&=\alpha_2(V,W) \cdot |2|_F^{2mn}\cdot |N(R(\ue))|^{-m} \int_{G(W)}\int_{G(W)}\int_{G(W)} \\
& \quad F_{\phi_1\otimes\phi_3}^\dagger(\iota(g, 1))\cdot\overline{F_{\phi_2\otimes\phi_4}(\iota(g'', 1))}\cdot (\pi(g'g)v_1, \pi(g'')v_2) \overline{(\pi(g')v_3, v_4)} \: dg''dgdg' \\
&=\frac{\alpha_2(V,W)}{\deg\pi} \cdot |2|_F^{2mn}\cdot |N(R(\ue))|^{-m} \int_{G(W)}\int_{G(W)}\\
& \quad F_{\phi_1\otimes\phi_3}^\dagger(\iota(g, 1))\cdot\overline{F_{\phi_2\otimes\phi_4}(\iota(g'', 1))}\cdot (\pi(g)v_1, v_3)\cdot\overline{(\pi(g'')v_2, v_4)} \: dg''dg \\
&=\frac{\alpha_2(V,W)}{\deg \pi} \cdot |2|_F^{2mn}\cdot |N(R(\ue))|^{-m}\cdot Z(\frac{1}{2}, F_{\phi_1\otimes\overline{\phi}_3}^\dagger, \overline{\xi}_{v_1, v_3})\cdot \overline{Z(-\frac{1}{2}, F_{\phi_2\otimes\overline{\phi}_4}, \overline{\xi}_{v_2, v_4})}.
\end{align*}
The local functional equation of the doubling zeta integral says that
\begin{align*}
&Z(-\frac{1}{2}, F_{\phi_1\otimes\overline{\phi}_3}, \overline{\xi}_{v_1, v_3}) \\
&= \left.\left(c(s,\chi_V, A_0,\psi)R(s,\chi_V,A,\psi)^{-1}\gamma(s+\frac{1}{2}, \pi\times\chi_V, \psi)\right)\right|_{s=\frac{1}{2}} \\
&\ \ \times  c_\pi(-1)\cdot Z(\frac{1}{2}, F_{\phi_1\otimes\overline{\phi}_3}^\dagger, \overline{\xi}_{v_1,v_3}).
\end{align*}
By Theorem \ref{gamma_main} and Proposition \ref{c()}, we have
\begin{align*}
&c(s,\chi_V, A_0,\psi)R(s,\chi_V,A,\psi)^{-1}\gamma(s+\frac{1}{2}, \pi\times\chi_V, \psi) \\
&=e(G(W))\cdot|N(R(\ue))|^{-s}\cdot|2|^{-2ns+n(n-\frac{1}{2})}\cdot\chi_V^{-1}(4) \\
&\times \frac{\gamma(s+\frac{1}{2},1_{F^\times},\psi)}{\gamma(2s,1_{F^\times},\psi)}\cdot\prod_{i=1}^{n-1} \gamma(2s-2i,1_{F^\times},\psi)^{-1} \cdot\gamma^V(s+\frac{1}{2}, \sigma\times\chi_W, \psi) \\
&\times \begin{cases}\gamma(s-n+\frac{1}{2},\chi_V,\psi)^{-1} & -\epsilon = 1\\
\epsilon(\frac{1}{2}, \chi_W, \psi)^{-1} & -\epsilon = -1.\end{cases}
\end{align*}
Moreover, we have
\begin{align*}
\gamma^V(1, \sigma\times\chi_W, \psi) &= \gamma(1, \sigma^\vee\times\chi_W, \psi) \\
&= \gamma(0, \sigma\times\chi_W,\overline{\psi})^{-1} \\
&=\gamma(0, \sigma\times\chi_W,\psi)^{-1} 
\times \begin{cases}\chi_W(-1) & (\epsilon=1), \\
\chi_V(-1) & (\epsilon = -1). \end{cases}
\end{align*}
Summarizing the above discussions and substituting Proposition \ref{cent_char}, we obtain: 
\begin{prop}\label{alpha23}
Suppose that $l=1$ and $\pi$ is square-integrable. Then, we have
\begin{align*}
&\frac{\deg \pi}{\deg \sigma}\cdot  c_\sigma(-1)\cdot\gamma^V(0,\sigma\times\chi_W, \psi)^{-1} \\
&=\frac{1}{2}\cdot \alpha_2(V,W) \cdot e(G(W)) \cdot |2|_F^{2n\rho - n (n-\frac{1}{2})}\cdot |N(R(\ue))|^{-\rho}
\cdot \prod_{i=1}^{n-1}\frac{\zeta_F(2i)}{\zeta_F(1-2i)} \\
 &\quad\times\begin{cases} 
\chi_V(-1)^{n+1}\gamma(1 - n, \chi_V, \psi) & (-\epsilon=1), \\
\chi_W(-1)^{m+1}\epsilon(\frac{1}{2}, \chi_W, \psi) & (-\epsilon=-1).\end{cases}
\end{align*}
\end{prop}

Hence we obtain Lemma \ref{preliminaryFD}.
We write down the constant $\alpha_3(V,W)$ in the minimal cases, which proves a special case of Theorem \ref{fd theta1}.
\begin{prop}
\begin{enumerate}
\item In the case $\epsilon = -1$ and $V$ is anisotropic, we have
\[
\alpha_3(V,W) = \epsilon(\frac{1}{2}, \chi_V, \psi)^{-1}. 
\]
\item In the case $\epsilon = 1$ and either $V$ or $W$ is anisotropic, we have
\[
\alpha_3(V,W) = \frac{1}{2}\cdot\chi_W(-1)^m\cdot\epsilon(\frac{1}{2}, \chi_W, \psi)^{-1}.
\]
\end{enumerate}
\end{prop}

\begin{proof}
Recall that 
\[
\chi(-1)\cdot\epsilon(\frac{1}{2}, \chi, \psi) = \epsilon(\frac{1}{2}, \chi, \psi)^{-1}
\]
for a quadratic character $\chi$ of $F^\times$. Then, for the case $m=0$, one can verify this proposition directly. Otherwise, we obtain the claim by Proposition \ref{locSW1}.
\end{proof}

%%%%%%%%%%%%%%%%%%%%%%%%%%%%%%%%%     S 16     %%%%%%%%%%%%%%%%%%%%%%%%%%%%%%%%%%
\section{
            Plancherel measures 
           }\label{planc mes}

Our next goal is to prove Theorem \ref{fd theta1} completely. This will be done in \S\ref{ind_arg}, and the Plancherel measure has important role in the proof. 
In this section, we recall some formulas of the Plancherel measure, and we discuss how the Plancherel measure behaves under the theta correspondence. 

\subsection{
                Preliminaries
                }
Let $G$ be a reductive group over $F$, let $P$ be a parabolic subgroup of $G$, let $M$ be a Levi subgroup of $P$, and let $U$ be the unipotent radical of $P$. We denote by $X^*(M)$ the group of the algebraic characters of $M$, and by $E_\C^\vee$ the vector space $X^*(M)\otimes\C$.
For a finite length representation $\pi$ of $M$ and for  
\[
\eta = \sum_{i=1}^t\chi_i\otimes s_i \in E_\C^\vee
\]
we denote by $\pi\otimes \eta$ the representation given by
\[
[\pi\otimes \eta](g) \coloneqq   \pi(g) \prod_{i=1}^t|\chi_i(g)|^{s_i}
\]
for $g\in M$. Take a maxmal compact subgroup $K$ of $G$ so that $G = PK$. Then for $f \in \Ind_P^G(\pi)$, we define $f_\eta \in \Ind_P^G(\pi\otimes \eta)$ by
\[
f_\eta(muk) = \prod_{i=1}^t|\chi_i(m)|^{s_i} f(muk)
\]
for $m \in M, u \in U, k \in K$. Denote by $\overline{P}$ the opposite parabolic subgroup of $P$, and by $\overline{U}$ the unipotent radical of $\overline{P}$. It is known that for $f \in \Ind_P^G\pi$, the integral
\[
[J_{\overline{P}|P}^G(\pi\otimes \eta)f_\eta](g) = \int_{\overline{U}} f_\eta(\overline{u}g) \: d\overline{u}
\]
converges absolutely when $\eta$ is contained in a certain open subset of $E_\C^\vee$, and it admits a meromorphic continuation to the whole space $E_\C^\vee$ with at most finitely many poles (see \cite{Sha81}). Here, the measure $d\overline{u}$ is the normalized Haar measure as in \S\ref{Haar2}. Therefore we have an intertwining operator
\[
J_{\overline{P}|P}^G(\pi\otimes \eta)\colon \Ind_P^G(\pi\otimes \eta) \rightarrow \Ind_{\overline{P}}^G(\pi\otimes \eta)
\]
for almost all $\eta \in E_\C^\vee$. This operator will be abbreviated to $J_{\overline{P}|P}(\pi\otimes\eta)$ unless it occurs confusion. The map $\eta \mapsto J_{\overline{P}|P}(\pi\otimes\eta)$ is rational (see \cite[\S IV]{Wal03}).
Since $\Ind_P^G(\pi\otimes \eta)$ is irreducible for all $\eta$ in a certain Zariski open subset of $E_\C^\vee$ (\cite[Theor\'{e}me 3.2]{Sau97}), there exists a rational function $\mu(\eta, \pi)$ of $\eta$ satisfying
\[
J_{P|\overline{P}}(\pi\otimes \eta)\circ J_{\overline{P}|P}(\pi\otimes \eta) = \mu(\eta, \pi)^{-1}.
\]
It is called the Plancherel measure. 
\begin{lem}\label{positivity}
If $\pi$ is square-integrable, then we have $\mu(0, \pi) >0$.
\end{lem}
\begin{proof}
Let $( \ , \ )$ denotes both the unitary inner product on $\Ind_P^G(\pi)$ and that of $\Ind_{\overline{P}}^G(\pi)$ as in the proof of \cite[IV, (6)]{Wal03}. Then, we have
\begin{align*}
\mu(0, \pi)^{-1} (f,f) &= (f, J_{P|\overline{P}}(\pi)\circ J_{\overline{P}|P}(\pi)f) \\
&=(J_{\overline{P}|P}(\pi)f, J_{\overline{P}|P}(\pi)f)
\end{align*}
for all $f \in \Ind_P^G(\pi)$. Choosing $f \in \Ind_P^G(\pi)$ so that $J_{\overline{P}|P}(\pi)f \not = 0$, we have $\mu(0, \pi) > 0$. Thus we have the lemma. 
\end{proof}

Let $W' \subset W$ be $(-\epsilon)$-Hermitian spaces, and let $X,X^*$ be totally isotropic subspaces of $W$ so that $W = X + W'+ X^*$ and $X+X^*$ is the orthogonal complement of $W'$. Now we consider the case where $G=G(W)$, and $P = P(X)$. The restriction to $X + W'$ gives the identification $M \cong \GL(X) \times G(W')$. Then, for a finite length representation $\pi'$ of $G(W')$ and a finite length representation $\tau$ of $\GL(X)$, we abbreviate $\mu(N\otimes s, \pi'\boxtimes\tau)$ to $\mu(s, \pi'\boxtimes\tau)$. Here, $N$ denotes the reduced norm of ${\rm End}(X)$. 

\subsection{
		Jacquet-Langlands correspondences
		}

Let $F$ be a local field of characteristic zero, let $d, t$ be positive integers, and let $D$ be a central division algebra of $F$ of $[D:F] = d^2$. We denote by ${\rm Irr}_{\rm unit}(\GL_{dt}(F))$ (resp. ${\rm Irr}_{\rm unit}(\GL_t(D))$) the set of the isomorphism classes of unitary irreducible representations of $\GL_{dt}(F)$ (resp. $\GL_t(D)$). Then, the local Jacquet-Langlands correspondence provides the map
\[
|\JL_F| \colon {\rm Irr}_{\rm unit}(\GL_{dt}(F)) \rightarrow {\rm Irr}_{\rm unit}(\GL_t(D)) \cup \{0\}.
\]
Let $\F$ be a global field of characteristic zero, let $d,t$ be positive integers, let $\D$ be a central division algebra over $\F$ of $[\D:\F] = d^2$, and let $\Pi$ be a discrete series of $\GL_{dt}(\A)$ so that $|\JL_{\F_v}|(\Pi_v) \not= 0$. 
Then, the global Jacquet-Langlands correspondence provides a non-zero discrete series $|\JL_\F|(\Pi)$ of $\GL_t(\D_\A)$.
We do not explain the definition of the correspondences, but state some important properties, which are excerpts of the results of   \cite{DKV84}, \cite{Bad08}.

\begin{prop}\label{JLcorr}
\begin{enumerate}
\item If $d = 1$, then we have $|\JL_F|$ is the identity map ${\rm Id}$. \label{JL1}
\item If $\pi$ is an irreducible supercuspidal representation of $\GL_{dt}(F)$, then $|\JL_F|(\pi)$ is non-zero and spercuspidal.\label{JL2}
\item If $\pi$ is an irreducible square-integrable representation of $\GL_{dt}(F)$, then $|\JL_F|(\pi)$ is non-zero and square-integrable. \label{JL3}
\item For all irreducible square-integrable representation $\pi'$ of $\GL_t(D)$, there exists an irreducible square-integrable representation so that $\pi' = |\JL_F|(\pi)$.  \label{JL4}
\item If $\Pi$ be a discrete series of $\GL_{dt}(\D_\A)$ such that $|\JL_{\F_v}|(\Pi_v)\not=0$ for all $v$, then we have $|\JL_\F|(\Pi)_v = |\JL_{\F_v}|(\Pi_v)$.\label{JL5}
\item For all discrete series $\Pi'$ of $\GL_t(\D_\A)$, there exists a discrete series $\Pi$ of $\GL_{dt}(\A)$ so that $|\JL_{\F_v}|(\Pi_v)\not=0$ and $|\JL_\F|(\Pi) = \Pi'$. \label{JL6}
\end{enumerate}
\end{prop}

\begin{proof}
By \cite[Theorem 5.1]{Bad08}, we have the assertions \eqref{JL1}, \eqref{JL5} and \eqref{JL6}. The assertion \eqref{JL2} follows from \cite[\S3.1]{Bad08}. Finally, we have the assertions \eqref{JL3} and \eqref{JL4} by \cite[Theorem 2.2]{Bad08}.
\end{proof}

\subsection{
                Plancherel measures for inner forms of general linear groups
                }

In this subsection, we denote by $D'$ a central division algebra over $F$. Then there is a positive integer $d$ so that $[D':F] = d^2$.
Let $t_1$ and $t_2$ be positive integers, and let $t=t_1+t_2$. 
We consider the case where $M=\GL_{t_1}(D')\times \GL_{t_2}(D')$ and $G=\GL_t(D')$. Then, we have an identification $\C^2\cong E_\C^\vee$ by 
\[
(\eta_1, \eta_2) \mapsto {N}_1\otimes\eta_1 + {N}_2\otimes\eta_2
\]
where $N_i$ denotes the reduced norm of $\GL_{t_i}(D')$ for $i=1,2$ respectively.

\begin{prop}
Let $\rho_i$ be a square-integrable irreducible representation of $\GL_{dt_i}(F)$ for $i=1,2$, and let $\tau_i$ be the representation $|\JL_F|(\rho_i)$ of $\GL_{t_i}(D')$ for $i=1, 2$. Then we have
\[
\mu(\eta, \tau_1\boxtimes\tau_2) = \gamma(s_1-s_2, \rho_1\boxtimes\rho_2^\vee, \psi)
                                        \gamma(s_2-s_1, \rho_1^\vee\boxtimes\rho_2, \overline{\psi})
\]
for $\eta = (s_1, s_2) \in \C^2$.
\end{prop}

\begin{proof}
First, if we denote by $P$ an $F$-rational parabolic subgroup of $\GL_t(D')$ having the Levi subgroup $M$, then we have 
\[
J_{\overline{P}|P}((\tau_1\boxtimes\tau_2)\otimes\eta)\otimes|N|^{-\frac{1}{2}(s_1 + s_2)}
= J_{\overline{P}|P}((\tau_1\boxtimes\tau_2)\otimes(\eta-(\frac{1}{2}(s_1 + s_2), \frac{1}{2}(s_1 + s_2))).
\]
Thus we have
\[
\mu(\eta, \tau_1\boxtimes\tau_2) = \mu((\frac{1}{2}(s_1-s_2), \frac{1}{2}(s_2 - s_1)), \tau_1\boxtimes\tau_2).
\]
By \cite[Theorem 7.2]{AP05}, this is equal to 
\[
\mu((\frac{1}{2}(s_1-s_2), \frac{1}{2}(s_2 - s_1)), \rho_1\boxtimes\rho_2).
\]
We remark that our normalization of the Haar measure is different from that of \cite{AP05}. Since $\rho_1$ and $\rho_2$ are generic (\cite[Theorem 9.3]{Zel80}), we have
\[
\mu(\eta, \rho_1\boxtimes\rho_2) = \gamma(s_1-s_2, \rho_1\boxtimes\rho_2^\vee, \psi)
                                        \gamma(s_2-s_1, \rho_1^\vee\boxtimes\rho_2, \overline{\psi})
\]
for $s \in \C$ (\cite{Sha90}). Hence we have the proposition. \end{proof}

\begin{prop}\label{GLmult}
Let $\tau_i$ be an irreducible representation of $\GL_{t_i}(D')$ for $i=1, 2$, let $B_1$ be a $F$-rational parabolic subgroup pf $\GL_{t_1}(D')$, let $M_1$ be the Levi subgroup of $B_1$. The subgroup $M_1$ is of the form 
\[
\GL_{t_{11}}(D)\times\cdots\times\GL_{t_{1\lambda}}(D)
\]
with $t_{11} + \cdots + t_{1\lambda} = t_1$. Suppose that $\tau_1$ is embedded into  $\Ind_{B_1}^{\GL_{t_1}(D)}\sigma_1$ where
\[
\sigma_1 =\sigma_{11}|N_{11}|^{a_1} \boxtimes \cdots \boxtimes \sigma_{1\lambda}|N_{1\lambda}|^{a_{\lambda}}
\]
for some complex numbers $a_1,\ldots, a_{\lambda_1}$ and for some irreducible representations $\sigma_{11}, \ldots, \sigma_{1\lambda}$. 
Then we have
\[
\mu(\eta, \tau_1\boxtimes \tau_2) = \prod_{j=1}^r\mu(\eta + (a_j, 0)), \sigma_j\boxtimes \tau_2).
\]
\end{prop}

\begin{proof}
Let $S$ be the center of $M_1\times\GL_{t_2}(D')$, and let $P$ be a parabolic subgroup of $\GL_t(D)$ whose Levi subgoup is $\GL_{t_1}(D) \times \GL_{t_2}(D)$. We denote by $U$ (resp, $U_1$) the unipotent radical of $P$ (resp. $B_1$), by $\Delta_S(U)$ (resp. $\Delta_S(U_1)$) the set of the roots of $S$ whose root subspace is contained in $U$ (resp. $U_1$). For $\alpha \in \Delta_S(U)$, we denote by $S_\alpha$ the kernel of $\alpha$ in $S$, and by $M_\alpha$ the centralizer of $S_\alpha$ in $\GL_{t}(D)$. We may suppose that  the product $(B_1\times \GL_{t_2}(D'))\cdot U$ is a parabolic subgroup of $\GL_t(D)$, and we denote it by $B$. Finally, we denote by $P_\alpha$ the parabolic subgroup $M_\alpha\cdot B$. 
Then we have
\[
J_{\overline{P}|P}^{\GL_t(D)}(\tau_1\boxtimes\tau_2, \eta) = \prod_{\alpha \in \Delta_S(U)} \Ind_{P_\alpha}^{\GL_{t}(D)}(J_{\overline{B}|B}^{M_\alpha}(\sigma_1\boxtimes\sigma_2, \eta)).
\]
Hence, we have the formula
\[
\mu(\eta, \tau_1\boxtimes\tau_2) = \prod_{j = 1}^\lambda \mu(\eta, \sigma_{1j}|N_{1j}|^{a_j}|\boxtimes\tau_2),
\]
which implies the proposition.
\end{proof}

\subsection{
                Global Property
                }
                
In this subsection, we recall a global property of the Plancherel measure for inner forms of general linear group and quaternionic unitary groups.
Let $\F$ be a global field, and 
let $\D$ be a division quaternion algebra over $\F$.

\begin{lem}\label{QUQUlem}
Let $\underline{W}$ be a left $(-\epsilon)$-Hermitian space over $\D$, let $\underline{X}, \underline{X}^*$ be two left $\D$-vector spaces so that $\dim \underline{X} = \dim \underline{X}^*=r'$, and let $\underline{W'} = \underline{X} + \underline{W} + \underline{X}^*$ a $(-\epsilon)$-Hermitian space equipped with the $(-\epsilon)$-Hermitian form
\[
\langle \ , \ \rangle'\colon (\underline{X} + \underline{W} + \underline{X}^*) \times (\underline{X} + \underline{W} + \underline{X}^*) \rightarrow \D
\]
defined by 
\[
\langle (x_1,w_1,y_1), (x_2,w_2,y_2) \rangle' = x_1\cdot J_{r'} \cdot {}^t\!y_2^* + \langle w_1, w_2\rangle  -\epsilon y_1 \cdot J_{r'} \cdot {}^t\!x_2^*. 
\] 
Here, we recall that $J_{r'}$ is the anti-diagonal matrix defined in \S\ref{basis for WV}.
Then, $M = \GL_{r'}(\D)\times G(\underline{W})$ is a Levi subgroup of a maximal parabolic subgroup of $G(\underline{W'})$. Then, for an irreducible cuspidal automorphic representation $\Pi\boxtimes\Xi$ of $M(\A)$, we have
\begin{align}\label{QUQU}
\prod_{v \in S}\mu_v(s, \Pi_v\boxtimes \Xi_v) &= \frac{L^S(1-s, \Pi\boxtimes\Xi^\vee)}{L^S(s,\Pi^\vee\boxtimes\Xi)}\cdot\frac{L^S(1+s,\Pi^\vee\boxtimes\Xi)}{L^S(-s, \Pi\boxtimes\Xi^\vee)} \\ \notag
&\times \frac{L^S(1-2s,\Xi^\vee, \wedge^2)}{L^S(2s, \Xi, \wedge^2)}\cdot\frac{L^S(1+2s, \Xi,\wedge^2)}{L^S(-2s, \Xi^\vee,\wedge^2)}.
\end{align}
Here, $S$ is a finite set of places of $\F$ such that 
\begin{itemize}
\item $S$ contains all Archimedean places,
\item $\D_v'$ is split for $v \not\in S$, and
\item $G(\underline{W}_v)$, $\Pi_v$, $\Xi_v$ are unramified for $v \not\in S$,
\end{itemize}
and we denote 
\[
L^S(s, \Xi^\vee, \wedge^2) = \prod_{v\not\in S} L(s, \Xi_v^\vee, \wedge^2)
\]
where the right-hand side is an infinite product of the local exterior-square $L$-factor.  
\end{lem}

\begin{proof}
Let $P$ be a parabolic subgroup so that $M$ is the Levi subgroup of $P$, and let $f = \otimes_vf_v \in \Ind_{P(\A)}^{G(\A)}\Pi\boxtimes\Xi^\vee$. Consider the Eisenstein series 
\[
E_{P}[f](g) = \sum_{\gamma \in P(\F)\backslash G(\underline{W'})(\F)} f(\gamma g) 
\]
for  $g \in G(\underline{W'})(\A)$. If $v \not\in S$, then by \cite{Sha90} we have
\begin{align*}
&L(s,\Pi_v^\vee\boxtimes\Xi_v)L(-s, \Pi_v\boxtimes\Xi_v^\vee)L(2s, \Xi_v, \wedge^2)L(-2s, \Xi_v^\vee,\wedge^2)\cdot [J_{\overline{P}|P}\circ J_{P|\overline{P}}](f_v)\\
&=L(s,\Pi_v^\vee\boxtimes\Xi_v)L(-s, \Pi_v\boxtimes\Xi_v^\vee)L(2s, \Xi_v, \wedge^2)L(-2s, \Xi_v^\vee,\wedge^2)\cdot f_v.
\end{align*}
 Hence, by the functional equation of the Eisenstein series (\cite[p.216, (i), (ii)]{Lan76}), we have
\begin{align*}
&L^S(1-s, \Pi\boxtimes\Xi^\vee)L^S(1+s, \Pi^\vee\boxtimes\Xi) \\
&\times L^S(1-2s,\Xi^\vee, \wedge^2) L^S(1+2s, \Xi,\wedge^2) \prod_{v\in S}\mu_v(\mu_v, \Pi_v\boxtimes\Xi_v)^{-1}  E[f] \\
&= L^S(s, \Pi^\vee\boxtimes\Xi)L^S(-s, \Pi\boxtimes\Xi^\vee) \\
&\times L^S(2s, \Xi, \wedge^2)L^S(-2s, \Xi^\vee,\wedge^2)\cdot E[J_{\overline{P}|P}\circ J_{P|\overline{P}}(f)] \\
&= L^S(s, \Pi^\vee\boxtimes\Xi)L^S(-s, \Pi\boxtimes\Xi^\vee)L^S(2s, \Xi, \wedge^2)L^S(-2s, \Xi^\vee,\wedge^2) \cdot E[f],
\end{align*}
which implies the lemma.
\end{proof}

\subsection{
                The behavior of the Plancherel measures under the theta correspondence
                }\label{loc_theta_Planc}

Now we consider the Plancherel measures for quaternionic unitary groups. Let $V$ be an $m$-dimensional right $\epsilon$-Hermitian space, and let $W$ be an $n$-dimensional left $(-\epsilon)$-Hermitian space. Note that, in this section, we allow the case where $l\not=1$.

\begin{prop}\label{pl under theta}
Let $\pi$ be an irreducible representation of $G(W)$, let $\sigma\coloneqq   \theta_\psi(\pi; V)$ and let $\tau$ be an irreducible representation of $\GL(X)$. Then we have
\[
\frac{\mu(s, \pi\boxtimes\tau\chi_V)}{\mu(s, \sigma\boxtimes\tau\chi_W)}
= \gamma(s- \frac{l-1}{2}, \tau, \psi)\cdot \gamma(-s-\frac{l-1}{2}, \tau^\vee, \overline{\psi}).
\]
\end{prop}

The remaining part of this subsection is devoted to the proof of Proposition \ref{pl under theta}. Put
\[
u_D(s; W, V, X, \pi, \tau) = \frac{\mu(s, \pi\boxtimes\tau\chi_V)}{\mu(s, \sigma\boxtimes\tau\chi_W)}\gamma(s- \frac{l-1}{2}, \tau, \psi)^{-1} \gamma(-s-\frac{l-1}{2}, \tau^\vee, \overline{\psi})^{-1}.
\]
We will use global argument to prove Proposition \ref{pl under theta}, so that {\bf we allow $D$ to be split} in the rest of this section. We want to show $u_D(W,V,X,\pi,\tau)=1$ for all $D,W,V,X,\pi,\tau$.

\begin{lem}\label{pl loc}
Let $\{ W_i\}_{i \geq 0}$ be a Witt tower consisting of $(-\epsilon)$-Hermitian spaces and let $\{V_j\}_{j \geq 0}$ be a Witt tower consisting of $\epsilon$-Hermitian spaces. Suppose that $V = V_r$ and $W = W_t$.
\begin{enumerate}
\item \label{pl loc spl}
If $D$ is split, then we have
\[
u_D(s; W,V,X,\pi,\tau) = 1.
\]
\item \label{pl mult 2}
Suppose that $\pi$ is a subrepresentation of $\Ind_{P_{t',t}}^{G(W)} \pi' \boxtimes \rho_{s_0}\chi_V$ where $t'$ is an integer so that $\max\{t(\pi), r\} \leq t' \leq t$, $s_0 \in \C$, $\pi'$ is an irreducible representation of $G(W_{t'})$, and $\rho$ is an irreducible supercuspidal representation of $\GL_{t-t'}(D)$. Then, we have
\[
u_D(s; W, V, X, \pi, \tau) = u_D(s; W_{t'}, V_{r'}, X, \pi', \tau)
\]
where $r' = r-(t-t')$.
\item \label{pm mult 3}
Let $X', X''$ be two subspaces of $X$ so that $X = X' + X''$.
Suppose that $\tau$ is an irreducible subquotient of induced representation $\Ind_{P(X')}^{\GL(X)}\tau'\boxtimes\tau''$ where $\tau'$ (resp.~ $\tau''$) is an irreducible representation of $\GL(X')$ (resp.~ $\GL(X'')$). Then, we have
\[
u_D(s; W, V, X, \pi, \tau) = u_D(s; W, V, X', \pi, \tau')u_D(s; W,V,X'', \pi, \tau'').
\]
\item \label{pl loc 1st}
If $r(\pi)$ denotes the first occurrence index, then we have
\[
u_D(s; W,V_{r(\pi)}, X,\pi,\tau) = u_D(s; W,V,X,\pi,\tau).
\]
\end{enumerate}
\end{lem}

\begin{proof}
The claim \eqref{pl loc spl} is proved in \cite[Theorrem 12.1]{GI14}. 
Then, we prove \eqref{pl mult 2}. By \cite[Proposition B.3]{GI14}, we have
\begin{align*}
\mu(s, \tau\chi_V\otimes\pi) &= \mu((s,s_0), \tau\chi_V\boxtimes\rho\chi_V)\mu((s,-s_0), \tau\chi_V\boxtimes{\rho}^\vee\chi_V)\mu(s, \tau\boxtimes\pi')\\
&=\mu((s,s_0), \tau\boxtimes\rho)\mu((s,-s_0), \tau\boxtimes{\rho}^\vee)\mu(s, \tau\boxtimes\pi').
\end{align*}
Hence, by Corollary \ref{reducing} (with replacing $V$ and $W$, $\sigma$ and $\pi$), we have 
\[
\frac{\mu(s, \tau\chi_V\otimes\pi)}{\mu(s, \tau\chi_W\otimes\sigma)} 
= \frac{\mu(s, \tau\chi_V\otimes\pi')}{\mu(s, \tau\chi_W\otimes\sigma')}.
\]
Thus, we have \eqref{pl mult 2}. We prove \eqref{pm mult 3} similarly by using \cite[Lemma B.2]{GI14}. Finally, we prove \eqref{pl loc 1st}.
Put $t^\pi = r-r(\pi)$. Then, by using the local functional equation of the doubling $\gamma$-factor (\cite[Theorem 5.7(4)]{Kak20}), we have
\[
\mu(s, \sigma\boxtimes\tau\chi_W) 
= \mu(s, |N|^{\frac{l}{2}+t^\pi}\boxtimes\tau\chi_W)\cdot\mu(s, |N|^{-\frac{l}{2}-t^\pi}\boxtimes\tau\chi_W)\cdot\mu(s, \sigma'\boxtimes\tau\chi_W).
\]
By Proposition \ref{GLmult}, this is equal to
\begin{align*}
& \frac{\prod_{i=1}^{2t^\pi}\gamma(s+\frac{l}{2}+i-\frac{1}{2}, \tau^\vee\chi_W, \psi)}
            {\prod_{i=1}^{2t^\pi}\gamma(s+\frac{l}{2}+i+\frac{1}{2}, \tau^\vee\chi_W, \psi)}\cdot
     \frac{\prod_{i=1}^{2t^\pi}\gamma(s-\frac{l}{2}+\frac{1}{2}-i, \tau^\vee\chi_W, \psi)}        
            {\prod_{i=1}^{2t^\pi}\gamma(s-\frac{l}{2}+\frac{3}{2}-i, \tau^\vee\chi_W, \psi)} \\
& \ \ \times \mu(s, \sigma'\boxtimes\tau\chi_W) \\
&= \frac{\gamma(s+\frac{l+1}{2}, \tau^\vee\chi_W, \psi)}{\gamma(s+\frac{l_0+1}{2}, \tau^\vee\chi_W, \psi)}
\cdot\frac{\gamma(s-\frac{l_0-1}{2}, \tau^\vee\chi_W, \psi)}{\gamma(s-\frac{l-1}{2}, \tau^\vee\chi_W, \psi)}
\cdot \mu(s, \sigma'\boxtimes\tau\chi_W) \\
&=\frac{\gamma(-s-\frac{l_0-1}{2}, \tau\chi_W, \overline{\psi})}{\gamma(-s-\frac{l-1}{2}, \tau\chi_W, \overline{\psi})}\cdot\frac{\gamma(s-\frac{l_0-1}{2}, \tau^\vee\chi_W, \psi)}{\gamma(s-\frac{l-1}{2}, \tau^\vee\chi_W, \psi)}
\cdot \mu(s, \sigma'\boxtimes\tau\chi_W) \\
&=\frac{\gamma(-s-\frac{l_0-1}{2}, \tau\chi_W, \psi)}{\gamma(-s-\frac{l-1}{2}, \tau\chi_W, \psi)}\cdot\frac{\gamma(s-\frac{l_0-1}{2}, \tau^\vee\chi_W, \overline{\psi})}{\gamma(s-\frac{l-1}{2}, \tau^\vee\chi_W, \overline{\psi})}
\cdot \mu(s, \sigma'\boxtimes\tau\chi_W).
\end{align*}
Hence we have
\begin{align*}
u_D(s; W,V,X,\pi,\tau) 
&= \frac{\mu(s, \pi\boxtimes\tau\chi_V)}{\mu(s, \sigma\boxtimes\tau\chi_W)}\gamma(s- \frac{l-1}{2}, \tau, \psi)^{-1} \gamma(-s-\frac{l-1}{2}, \tau^\vee, \overline{\psi})^{-1} \\
&= \frac{\mu(s, \sigma'\boxtimes\tau\chi_W)}{\mu(s, \sigma\boxtimes\tau\chi_W)}\cdot\frac{\mu(s, \pi\boxtimes\tau\chi_V)}{\mu(s, \sigma'\boxtimes\tau\chi_W)} \\
& \ \ \times\gamma(s- \frac{l-1}{2}, \tau, \psi)^{-1} \gamma(-s-\frac{l-1}{2}, \tau^\vee, \overline{\psi})^{-1} \\
&= \frac{\mu(s, \pi\boxtimes\tau\chi_V)}{\mu(s, \sigma'\boxtimes\tau\chi_W)}\gamma(s- \frac{l_0-1}{2}, \tau, \psi)^{-1} \gamma(-s-\frac{l_0-1}{2}, \tau^\vee, \overline{\psi})^{-1} \\
&= u_D(s; W,V_{r(\pi)}, X,\pi,\tau).
\end{align*}
Thus we have \eqref{pl loc 1st}.
\end{proof}

Now we prove Proposition \ref{pl under theta}. By Corollary \ref{reducing} and Lemma \ref{pl loc} \eqref{pl mult 2}\eqref{pm mult 3}, we may assume that $\pi$, $\sigma$, and $\tau$ are supercuspidal.
Take 
\begin{itemize}
\item a global field $\F$ and two distinct places $v_1, v_2$ of $\F$ so that $\F_{v_1} = \F_{v_2} = F$, 
\item a non-trivial additive character $\underline{\psi}$ of the ring of adeles $\A$ of $\F$,
\item a division quaternion algebra $\D$ over $\F$ so that $\D_{v_1} = \D_{v_2} = D$ and $\D_v$ is split for all $v \not= v_1, v_2$,
\item an $\epsilon$-Hermitian space $\V$ over $\D$ so that $\V_{v_1} = \V_{v_2} =V$, 
\item a Witt tower $\{ \V_i\}_{i=0}^\infty$ containing $\V$,
\item a $-\epsilon$-Hermitian space $\bW$ over $\D$ so that $\bW_{v_1} = \bW_{v_2} = W$, 
\item an irreducible cuspidal automorphic representation $\Pi$ of $G(\bW)(\A)$ so that $\Pi_{v_1} = \pi$,
\item a vector space $\bX$ over $\D$ so that $\dim_\D\bX = \dim_DX$,
\item an irreducible cuspidal automorphic representation $\Xi$ of $\GL(\bX)(\A)$ so that $\Xi_{v_1} = \tau$,
\item a finite subset $S$ of places so that $v_1, v_2 \in S$, all Archimedean places are contained in $S$ and $\Pi_v$, $\Xi_v$ are unramified for all places $v\not\in S$.
\end{itemize}
Let $r(\Pi)$ be the first occurrence index of the theta correspondence of $\Pi$ to the Witt tower $\{\V_i\}_{i=0}^\infty$. Then, $\Theta_{\underline{\psi}}(\Pi, \V_{r(\Pi)})$ is a non-zero irreducible cuspidal automorphic representation. We denote by $\pi'$ (resp.~ $\tau'$) the representation $\Pi_{v_2}$ (resp.~ $\Xi_{v_2}$). Hence we have
\begin{align}\label{glpl}
&u_D(s; V,W,X,\pi,\tau)\cdot u_D(s; V,W,X, \pi', \tau') \\ 
& = u_{\D_{v_1}}(s; \V_{v_1}, \bW_{v_1}, \bX_{v_1}, \Pi_{v_1}, \Xi_{v_1})\cdot u_{\D_{v_2}}(s; \V_{v_2}, \bW_{v_2}, \bX_{v_2}, \Pi_{v_2}, \Xi_{v_2}) \nonumber\\
&= u_{\D_{v_1}}(s; (\V_{r(\Pi)})_{v_1}, \bW_{v_1}, \bX_{v_1}, \Pi_{v_1}, \Xi_{v_1})\cdot u_{\D_{v_2}}(s; (\V_{r(\Pi)})_{v_2}, \bW_{v_2}, \bX_{v_2}, \Pi_{v_2}, \Xi_{v_2}) \nonumber\\
& \ \ \times \prod_{v\not= v_1, v_2} u_{\D_{v}}(s; (\V_{r(\Pi)})_{v}, \bW_{v}, \bX_{v}, \Pi_{v}, \Xi_{v}) \nonumber\\ 
&= 1 \nonumber
\end{align}
Applying \eqref{glpl} when $\Pi$ and $\Xi$ are chosen so that $\pi' = \pi$ and $\tau' = \tau$, we have $u_D(s; V,W,X,\pi,\tau)^2 = 1$. Hence $u_D(s; V,W,X,\pi,\tau) = \pm1$. 
It remains to determine the signature. By Lemma \ref{pl loc} \eqref{pl loc 1st}, we may assume that $\sigma$ is also supercuspidal. Moreover, we may assume that $\tau$ is not of the form $\chi \circ N$ for any unramified character $\chi$ of $F^\times$, where $N$ denotes the reduced norm. Then, the Godement-Jacquet $L$-factor of $\tau$ is $1$ (\cite[Propositions 4.4, 5.11]{GJ72}). 
By Lemma \ref{positivity}, we have
\[
\mu(0, \pi\boxtimes\tau\chi_V) >0 \mbox{ and } \mu(0, \sigma\boxtimes\tau\chi_W) > 0.
\]
On the other hand, putting
\[
\epsilon(s+\frac{1}{2}, \tau, \psi) = a_\psi(\tau)\cdot q^{-\lambda s}
\]
with $a_\psi(\tau) \in \C^\times$ and $\lambda \in \Z$, we have 
\[
\epsilon(-s+\frac{1}{2}, \tau^\vee, \overline{\psi}) = a_\psi(\tau)^{-1}\cdot q^{\lambda s},
\]
and we have
\begin{align*}
&\gamma(-\frac{l-1}{2}, \tau, \psi)\gamma(-\frac{l-1}{2}, \tau^\vee, \psi) \\
&=a_\psi(\tau)q^{\lambda l/2}\cdot \frac{L(\frac{l+1}{2}, \tau^\vee)}{L(-\frac{l-1}{2}, \tau)} \cdot
   a_\psi(\tau)^{-1}q^{\lambda l/2}\cdot \frac{L(\frac{l+1}{2}, \tau)}{L(-\frac{l-1}{2}, \tau^\vee)} \\
&= q^{\lambda l} > 0.
\end{align*}
Thus, the signature of $u_D(s; V,W, X, \pi, \tau)$ turns out to be $1$. This completes the proof of Proposition \ref{pl under theta}.

%%%%%%%%%%%%%%%%%%%%%%%%%%%%%%%%%     S17      %%%%%%%%%%%%%%%%%%%%%%%%%%%%%%%%%%%

\section{
            Poles of Plancherel measures
            }\label{accidental isom}

In this section, we study the poles of the Plancherel measures, and we construct some irreducible supercuspidal representations whose Plancherel measures are not holomorphic on $\R_{>0}$. We will use this representation in the next section.

Let $F$ be a local field, let $D$ be a division quaternion algebra over $F$, and let $V$ be an $m$-dimensional $\epsilon$-Hermitian space. Denote by $V_0$ be the anisotropic kernel of $V$, and write $V=X+V_0+X^*$ where $X, X^*$ are totally isotropic subspaces so that $X+X^*$ is the orthogonal complement of $V_0$. Put $r=\dim_DX$. 

\begin{prop}\label{pole aru}
\begin{enumerate}
\item There exists an irreducible supercuspidal representation $\rho$ of $\GL_{2r}(F)$ such that the image of the $L$-parameter $\phi_\tau\colon\SL_2(\C)\times W_F\rightarrow \GL_{2r}(\C)$ is contained in $\Sp_{2r}(\C)$. \label{pole aru1}
\item Take an irreducible supercuspidal representation $\rho$ of $\GL_{2r}(F)$ as in \eqref{pole aru1}, and denote by $\tau$ the representation $|\JL_F|(\rho)$ of $\GL(X)$. Then, for an irreducible representation $\sigma$ of $G(V_0)$, the Plancherel measure $\mu(s, \sigma\boxtimes\tau)$ has at least one pole in $\R_{>0}$.\label{pole aru2}
\end{enumerate}
\end{prop}

This proposition is proved at the end of this section.
Before starting to prove it, we recall accidental isomorphisms to show the explicit formula of the Plancherel measure for some cases.
To use the global argument, we write down them in the global setting. Let $\F$ be a global field, let $\D$ be a division quaternion algebra over $\F$, and $\underline{V_0}$ be an anisotropic $\epsilon$-Hermitian spaces over $\D$. We denote by $\widetilde{G}(\underline{V_0})$ the similitude group of $\underline{V_0}$.
Then it is known that $\widetilde{G}(\underline{V_0})$ is isomorphic to a certain more familiar group. By $\E$ we mean the etale $\F$ algebra associated with the quadratic (or trivial) character $\chi_{\underline{V_0}}$ of the idele group of $\F$.
\begin{itemize}
\item Suppose that $\epsilon = 1$ and $m=1$. Then we have $\widetilde{G}(\underline{V_0}) = \D^\times$. If we denote by $\underline{V_0'}$ the two-dimensional symplectic space over $\F$, then $\widetilde{G}(\underline{V_0})$ is an inner form of ${\rm GSp}(\underline{V_0'}) = \GL_2(\F)$.
\item Suppose that $\epsilon=-1$ and $m=1$. Let $\underline{V_0'}$ be a two-dimensional quadratic space such that $\chi_{\underline{V_0'}} = \chi_{\underline{V_0}}$. Then we have $\widetilde{G}(\underline{V_0}) = {\rm GSO}(\underline{V_0'})$ which is isomorphic to $\E^\times$.
\item Suppose that $\epsilon=-1$ and $m=2$. If we denote by $\underline{V_0'}$ the isotropic four-dimensional quadratic space such that $\chi_{\underline{V_0'}} = \chi_{\underline{V_0}}$, then $\widetilde{G}(\underline{V_0})$ is an inner form of ${\rm GSO}(\underline{V_0'})$ which is isomorphic to $\GL_2(\E)\times \F^\times /\{(z, N_{\E/\F}(z)^{-1}) \mid z \in \E^\times\}$ (c.f. \cite[A.2]{GI11}). Thus, by using the non-commutative version of the Shapiro's lemma (c.f. \cite[Proposition 1.7]{PR94}), 
we have $\widetilde{G}(\underline{V_0}) \cong \mathbb{B}^\times\times\F^\times/\{ (z, N_{\E/\F}(z)^{-1}) \mid z \in \E^\times\}$ for some division quaternion algebra $\mathbb{B}$ over $\E$.
\item Finally, suppose that $\epsilon=-1$ and $m=3$. If we denote by $\underline{V_0'}$ the split six-dimensional quadratic space over $\F$, then $\widetilde{G}(\underline{V_0})$ is an inner form of ${\rm GSO}(\underline{V_0'})$ which is isomorphic to $\GL_4\times \GL_1^\times/\{(z,z^{-2})\mid z \in \GL_1\}$. Thus, we have $\widetilde{G}(\underline{V_0}) \cong \D_4^\times\times \F^\times/\{(z,z^{-2}) \mid z \in \F^\times\}$ for some central division algebra $\D_4$ with $[\D_4:\F]=16$.
\end{itemize}
Therefore, applying the Jacquet-Langlands correspondence, we have the following lemma:
\begin{lem}\label{RRRR}
Let $\widetilde{\Sigma}$ be a discrete series of $\widetilde{G}(\underline{V_0})(\A)$. Then, there is a discrete series $\widetilde{\mathcal{R}}$ of ${\rm GSO}(\underline{V_0'})(\A)$ (or ${\rm GSp}(\underline{V_0'})(\A)$) such that $\widetilde{\Sigma}_v$ and $\widetilde{\mathcal{R}}_v$ coincide for all place $v$ of $\F$.
\end{lem}

\begin{proof}
We write $\widetilde{G}(\underline{V_0})$ as a quotient ${\mathbb{B}'}^\times\times\GL_1/C_1$ where $\mathbb{B}'$ is a central division algebra of a finite extension field $\E'$ of $\F$ and $C_1$ is a central subgroup of ${\mathbb{B}'}^\times\times\GL_1$. Then, ${\rm GSO}(\underline{V_0})$ (or ${\rm GSp}(\underline{V_0})$) is of the form $\GL_d(\E')\times\GL_1/C_2$ where $d$ is the positive integer so that $d^2 = [\mathbb{B}':\mathbb{E}']$,  $C_2$ is a central subgroup of  $\GL_d(\E')\times\GL_1$. Then, $C_1$ is isomorphic to $C_2$ via the inner twist isomorphism. By Proposition \ref{JLcorr}, we have that there exists a discrete series $\widetilde{\mathcal{R}}'$ of $\GL_d(\E')\times\GL_1$ so that $|\JL_{\E'}|(\widetilde{\mathcal{R}}') = \widetilde{\Sigma}$. 
Since the weak approximation holds for $C_2$ in each case, we have $\widetilde{\mathcal{R}}'|_{C_2} = 1_{C_2}$ by  \eqref{JL1} and \eqref{JL5} of Proposition \ref{JLcorr}. Hence the representation $\widetilde{\mathcal{R}}$ of $\GL_d(\E')\times\GL_1/C_2$ yielded by  $\widetilde{\mathcal{R}}'$ satisfies the conditions of the lemma. Thus we have the claim.
\end{proof}

Let $F$ be a local field of characteristic $0$, let $D$ be the division quaternion algebra over $F$, and let $V$ be an $\epsilon$-Hermitian space over $D$. As in the global case, we denote $V_0'$
\[
\begin{cases} \mbox{the $2n_0$-dimensional equipped with the symplectic form} & (\epsilon = 1), \\
\mbox{the $2n_0$-dimensional quadratic space of $\chi_{V_0'} = \chi_{V_0}$} & (\epsilon = -1). \end{cases} 
\] 

\begin{lem}
Let $\sigma$ be an irreducible representation of $G(V_0)$, and let $\rho$ be an irreducible supercuspidal representation of $\GL_{2r}(F)$. Then, there exists a square-integrable representation $\sigma'$ of $\SO(V_0')$ (or $\Sp(V_0')$) such that 
\[
\mu(s, \sigma\boxtimes |\JL_F|(\rho)) = \frac{\gamma(s, {\sigma'}^\vee\boxtimes\rho, \psi)}{\gamma(1+s, {\sigma'}^\vee\boxtimes\rho, \psi)} \cdot \frac{\gamma(2s, \rho, \wedge^2, \psi)}{\gamma(1+2s, \rho, \wedge^2, \psi)}.
\]
Here, $\gamma(s, \rho, \wedge^2, \psi)$ is the Langlands-Shahidi $\gamma$-factor (\cite{Sha90}).
\end{lem}

\begin{proof}
We prove this lemma only with $\epsilon = -1$ for simplicity. 
Take
\begin{itemize}
\item a unitary irreducible representation $\widetilde{\sigma}$ of the similitude group $\widetilde{G}(V_0)$ so that $\widetilde{\sigma}|_{G(V_0)}$ contains $\sigma$,
\item a global field $\F$ and places $v_1, v_2$ of $\F$ such that $\F_{v_1} = \F_{v_2} = F$, 
\item a division quaternion algebra $\D$ over $\F$ such that $\D_{v_1} = \D_{v_2} = D$, and for all  place $v\not= v_1, v_2$, $\D_v$ is split,
\item an anisotropic $\epsilon$-Hermitian space $\V_0$ such that ${\V_0}_{v_1} = {\V_0}_{v_2} = V_0$, and for all place $v\not=v_1,v_2$, $G(\V_0)$ is quasi-split,
\item a $2n_0$-dimensional quadratic space $\V_0'$ such that $\chi_{\V_0'} = \chi_{\V_0}$,
\item a vector spaces $\bX$, $\bX^*$ over $\D$ such that $\dim_\D\bX = \dim_\D\bX^* = \dim_DX$.
\end{itemize}
We denote by $\widetilde{Z}_1$ the center of $\widetilde{G}(\V_0)$ and by $\widetilde{Z}_2$ the center of ${\rm GSO}(\V_0')$. Then $\widehat{Z}_1$ and $\widehat{Z}_2$ are isomorphic to each other. Denote by $\omega_{\widetilde{\sigma}}$ the central character of $\sigma$ and by $\omega_\rho$ the central character of $\rho$. Then, by Lemma \ref{Heckegl}\eqref{Hecke2}, there exists a character $\underline{\omega}$ of $Z_1(\A)/Z_1(\F)$ so that $\underline{\omega}_{v_1} = \underline{\omega}_{v_2} = \omega_{\widetilde{\sigma}}$, and there exists a Hecke character $\underline{\omega}'$ of $\A^\times$ so that $\underline{\omega}_{v_1}' = \underline{\omega}_{v_2}' = \omega_\rho$. We take an auxiliary non-Archimedean place $v_3 \not= v_1, v_2$. Then, by Lemma \ref{glob rep}, there exist 
\begin{itemize}
\item a cuspidal automorphic irreducible representation $\widetilde{\Sigma}$ of $\widetilde{G}(\V_0)(\A)$ so that $\widetilde{\Sigma}_{v_1} = \widetilde{\Sigma}_{v_2} = \widetilde{\sigma}$ and $\widetilde{\Sigma}_{v_3}$ is supercuspidal,
\item a cuspidal automorphic irreducible representation $\widetilde{\Xi}$ of $\GL_{2r}(\A)$ so that $\widetilde{\Xi}_{v_1} = \widetilde{\Xi}_{v_2} = \rho$ and $\widetilde{\Xi}_{v_3}$ is supercuspidal.
\end{itemize}
Take  a discrete series $\widetilde{\mathcal{R}}$ as in Lemma \ref{RRRR}. Then, $\widetilde{\mathcal{R}}$ is cuspidal since $\widetilde{\mathcal{R}}_{v_3} = \widetilde{\Sigma}_{v_3}$ is supercuspidal. By the product formula (Lemma \ref{QUQUlem}), we have
\begin{align*}
\mu(s, \sigma\boxtimes |\JL_{F}|(\rho))^2 
&= \mu(s, \widetilde{\sigma}\boxtimes|\JL_{F}|(\rho))^2 \\
&=\mu(s, \widetilde{\mathcal{R}}_{v_1}\boxtimes \Xi_{v_1})\mu(s, \widetilde{\mathcal{R}}_{v_2}\boxtimes \Xi_{v_2}) \\
&=\mu(s, \widetilde{\mathcal{R}}_{v_1}\boxtimes \Xi_{v_1})^2. 
\end{align*}
Thus, by the positivity (Lemma \ref{positivity}), we have
\[
\mu(s, \sigma\boxtimes |\JL_{F}|(\rho)) = \mu(s, \widetilde{\mathcal{R}}_{v_1}\boxtimes \Xi_{v_1}).
\]
Moreover, since  $\widetilde{\mathcal{R}}_{v_1}$ and $\rho$ are generic, this is equal to
\[
\frac{\gamma(s, {\widetilde{\mathcal{R}}_{v_1}}^\vee\boxtimes\rho, \psi)}{\gamma(1+s, {\widetilde{\mathcal{R}}_{v_1}}^\vee\boxtimes\rho, \psi)} \cdot \frac{\gamma(2s, \rho, \wedge^2, \psi)}{\gamma(1+2s, \rho, \wedge^2, \psi)}
\]
by \cite{Sha90}. Thus, putting $\sigma' = \widetilde{\mathcal{R}}_{v_1}$, we have the lemma.
\end{proof}

Now we prove Proposition \ref{pole aru}. \eqref{pole aru1} is a consequence of \cite[\S4]{Mie20}. We prove \eqref{pole aru2}. Let $\rho$ be an irreducible supercuspidal representation  so that the image of the $L$-parameter $\phi_\tau\colon\SL_2(\C)\times W_F\rightarrow \GL_{2r}(\C)$ is contained in $\Sp_{2r}(\C)$.
Then, by a result of Jiang, Nien and Qin \cite{GNQ10}, we can conclude that $\gamma(s, \rho, \wedge^2, \psi)$ has a pole at $s=1$. 
Let ${\rm Fr} \in W_F$ be a Frobenius element. Then, by \cite[Lemma3]{GR10}, $\phi_\rho({\rm Fr})$ has finite order, hence, $[\wedge^2\circ\phi_\rho]({\rm Fr})$ is a unitary operator. Thus all poles of $L(s, \rho, \wedge^2)$ lie in $\{\Re s = 0\}$, and we can conclude that $\gamma(s, \rho, \wedge^2, \psi)$ has a pole at $s=1$, and all poles of $\gamma(s, \rho, \wedge^2, \psi)$ lie in $\{\Re s = 1\}$. Hence, the ratio
\[
\frac{\gamma(2s,\rho,\wedge^2,\psi)}{\gamma(1+2s,\rho,\wedge^2,\psi)}
\]
has a pole at $s= \frac{1}{2}$. Put
\[
\mathcal{P} = \{ s_0 \geq \frac{3}{2}  \mid \mbox{ $\gamma(s, {\sigma'}^\vee\boxtimes\rho, \psi)$ has a pole at $s=s_0$}\}. 
\]
If $\mathcal{P} = \varnothing$, then $\mu(s, \sigma\boxtimes|\JL_F|(\rho))$ has a pole at $s=\frac{1}{2}$ since all zeros of $\gamma(s, {\sigma'}^\vee\boxtimes\rho, \psi)$ lie in $\{\Re s \leq 0\}$.
If $\mathcal{P} \not=\varnothing$, then the ratio $\gamma(s, {\sigma'}\boxtimes\rho, \psi)/\gamma(1+s, {\sigma'}^\vee\boxtimes\rho, \psi)$ has a pole at $s = \sup \cP$. Hence we finish the proof of Proposition \ref{pole aru}.

%%%%%%%%%%%%%%%%%%%%%%%%%%%%%%%%%%     S 18      %%%%%%%%%%%%%%%%%%%%%%%%%%%%%%%%

 \section{
            Induction argument
            }\label{ind_arg}

In this section, we prove the compatibility of $\alpha_3(V,W)$ with the induction on the dimensions of $V, W$ with $l=1$, which completes the proof of Theorem \ref{fd theta1}.
We emphasize that we allow $V$ or $W$ to be $0$.
Now, we explain more precisely. 
Let $V$ be an $m$-dimensional right $\epsilon$-Hermitian space, and let $W$ be an $n$-dimensional left $(-\epsilon)$-Hermitian space. We assume that $l = 2n - 2m -\epsilon = 1$. Consider
\begin{itemize}
\item an $\epsilon$-Hermitian space $V'$ containing $V$ and its totally isotropic subspaces $X, X^*$ so that $\dim_D X =\dim_DX^*=t$,  $X+V + X^*=V'$ and $X+X^*$ is the orthogonal complement of $V$, 
\item a $(-\epsilon)$-Hermitian space $W'$ containing $W$ and its totally isotropic subspaces $Y, Y^*$ so that $\dim_D Y =\dim_DY^*=t$,  $Y+W + Y^*=W'$ and $Y+Y^*$ is the orthogonal complement of $W$.
\end{itemize}
We put $n' = n + 2t$ and $m' = m + 2t$. Then, we will prove 
\begin{align}\label{goal18}
\alpha_3(V',W') = \alpha_3(V,W)
\end{align}
in this section. Let $Q$ (resp.~ $P$) be the maximal parabolic subgroup of $G(V')$ (resp.~ $G(W')$) preserving $X$ (resp.~ $Y$). Then, we can identify the Levi subgroup $L_Q$ (resp.~ $M_P$) of $Q$ (resp.~ $P$) with $\GL(X) \times G(V)$ (resp.~ $\GL(Y)\times G(W)$). 
Recall that $\alpha_3(V',W')$ and $\alpha_3(V,W)$ do not depend on the choices of the representations (see Proposition \ref{alpha23}). Thus, it suffice to compare $\deg \pi'/\deg\theta_\psi(\pi', V)$ with $\deg\pi/\deg\theta_\psi(\pi, V)$ for at least one pair $(\pi, \pi')$ of square-integrable representations $\pi$ of $G(W)$ and $\pi'$ of $G(W')$ so that both $\theta_\psi(\pi, V)$ and $\theta_\psi(\pi', V')$ are non-zero.  

\begin{prop}\label{ind1st}
Suppose that there are $s_0 >0$, an irreducible supercuspidal representation $\pi$ of $G(W)$,  an irreducible supercuspidal representation $\sigma$ of $G(V)$, and a non-trivial irreducible supercuspidal representation $\tau$ of $\GL(X) \cong \GL(Y)$ so that
\begin{itemize}
\item $\sigma \cong \theta_\psi(\pi, V)$,
\item $\Ind_P^{G(W')}\pi\boxtimes\tau_{s_0}\chi_V$ is reducible, and
\item $\Ind_Q^{G(V')}\sigma\boxtimes\tau_{s_0}\chi_W$ is reducible.
\end{itemize}  
Then, $\Ind_P^{G(W')}\pi\boxtimes\tau_{s_0}\chi_V$ and $\Ind_Q^{G(V')}\sigma\boxtimes\tau_{s_0}\chi_W$ have unique irreducible square-integrable representations $\pi'$ and $\sigma'$ respectively, and $\sigma'$ coincides with  $\theta_\psi(\pi', V')$. 
Moreover, we have $\alpha_3(V',W') = \alpha_3(V,W)$.
\end{prop}

We prove this proposition in the former part of this section.
Suppose that a quadruple $(s_0, \pi, \sigma, \tau)$ as in the proposition is given.
By Lemma \ref{sqrrrrr} and Proposition \ref{ttoowweerr}, the representation $\theta_\psi(\pi', V')$ is the unique square-integrable irreducible subquotient of $\Ind_Q^{G(V')}\sigma\boxtimes\tau_{s_0}\chi_W$, which is nothing other than $\sigma'$.
To prove the last assertion, we use the following proposition, which is due to a result of Heiermann \cite{Hei04}.

\begin{prop}\label{mlt_theta}
Let $s_0 >0$, let $\pi$ be an irreducible supercuspidal representation of $G(W)$ and let $\tau$ be a supercuspidal representation of $\GL_t(D)$. Suppose that $\mu(s, \pi\boxtimes\tau\chi_V)$ has a pole at $s=s_0$. Then we have the following:
\begin{enumerate}\renewcommand{\labelenumi}{(\arabic{enumi})}
\item The induced representation $\Ind_{Q}^{G(W')}\pi\boxtimes\tau_{s_0}\chi_V$ is reducible and it has a unique irreducible square-integrable constituent $\pi'$. Moreover we have, 
\begin{align*} 
\deg \pi' = &2 t\log q \cdot \deg \pi \deg \tau \cdot \Res_{s=s_0}\mu(s, \pi\boxtimes\tau\chi_V) \\
&\times \gamma(G(W')/P) \cdot \frac{|K_{M_{P}}|}{|K_{G(W')}|}\cdot |U_P\cap K_{W'}|\cdot|\overline{U_P}\cap K_{W'}|.
\end{align*}
Here, $\gamma(G(W')/P)$ is the constant defined by
\[
\gamma(G(W')/P) = \int_{\overline{U}} \delta_{P}(\overline{u}) \: d\overline{u}
\]
where $\overline{U}$ is the unipotent radical of the opposite parabolic subgroup $\overline{P}$ of $P$, and $f^\circ$ is the unique $K_{W'}$-invariant section of the representation $\Ind_{P}^{G(W')} \delta_P^{\frac{1}{2}}$ induced by the square root of the modular character $\delta_P$ so that $f^\circ(1) = 1$. \label{111111111}
\item  The induced representation $\Ind_{Q}^{G(V')} \sigma\boxtimes\tau_{s_0}\chi_V$ is also reducible, and it has a unique irreducible square-integrable constituent $\sigma'$. Moreover we have  
\begin{align*}
\deg \sigma' = &2 t\log q \cdot \deg \sigma \deg \tau \cdot \Res_{s=s_0}\mu(s, \sigma\boxtimes\tau\chi_W) \\
&\times \gamma(G(V')/Q) \cdot \frac{|K_{L_Q}|}{|K_{G(V')}|}\cdot |U_Q\cap K_{V'}|\cdot|\overline{U_Q}\cap K_{V'}|.
\end{align*}
Here, $\gamma(G(V')/Q)$ is the constant defined similarly as in \eqref{111111111}.
\end{enumerate}
\end{prop}

\begin{proof}
This proposition is obtained by the proof of [GI14, Proposition 20.4].
\end{proof}

Now, take $\pi$ as in Proposition \ref{ind_quadraple}, and put $\sigma = \theta(\pi,V)$. 
Then, by Proposition \ref{mlt_theta},  we have
\begin{align*}
\frac{\deg \pi'}{\deg \sigma'} & = \frac{\deg\pi}{\deg\sigma} \cdot \frac{\Res_{s=s_0}\mu(s, \pi\boxtimes{\chi_V})}{\Res_{s=s_0}\mu(s, \sigma\boxtimes{\chi_W})}\cdot \frac{\gamma(G(W)/P)}{\gamma(G(V)/Q)}\cdot\frac{|K_{G(V')}||K_{M_{P}}|}{|K_{G(W')}||K_{L_Q}|} \\
& = \frac{\deg\pi}{\deg\sigma}\cdot \gamma(s_0-\frac{l-1}{2}, | \ |^{s_0}, \psi)\gamma(-s_0-\frac{l-1}{2}, | \ |^{s_0}, \overline{\psi}) \\
& \ \times \frac{|U_P\cap K_{W'}|\cdot|\overline{U_P}\cap K_{W'}|}
{|U_Q\cap K_{V'}|\cdot|\overline{U_Q}\cap K_{V'}|}\cdot\frac{|\cB_{V'}^+||\cB_{M_P}^+|}{|\cB_{W'}^+||\cB_{L_Q}^+|} \\
& \ \times \frac{\prod_{\alpha \in \Sigma_{\rm red}(\overline{P})}[X_\alpha \cap K_{W'}:X_\alpha\cap\cB_{W'}^+]^{-1}}{\prod_{\beta \in \Sigma_{\rm red}(\overline{Q})}[X_\beta \cap K_{V'}:X_\beta\cap\cB_{V'}^+]^{-1}}.
\end{align*}
Here, we denote by $\cB^+$ by the pro-unipotent radical of $\cB$, by $\Sigma_{\rm red}(\overline{P})$ (resp.~ $\Sigma_{\rm red}(\overline{Q})$) the set of positive reduced root with respect to the opposite parabolic subgroup $\overline{P}$ (resp.~ $\overline{Q}$) of $P$ (resp.~ $Q$), and by $X_\alpha$ (resp.~ $X_\beta$) the root subrgroup associated with $\alpha \in \Sigma_{\rm red}(\overline{P})$ (resp.~ $\beta \in \Sigma_{\rm red}(\overline{Q})$).

\begin{lem}
We have
\[
\frac{|\cB_{V'}^+||\cB_{M_P}^+|}{|\cB_{W'}^+||\cB_{L_Q}^+|} = q^{2t}.
\]
\end{lem}

\begin{proof}
Since $|\cB_{M_P}^+| = |\cB_{W}^+||\cB_{\GL_r(D)}^+|$ and $|\cB_{L_Q}^+| = |\cB_{V}^+||\cB_{\GL_r(D)}^+|$, we have
\begin{align*}
\frac{|\cB_{V'}^+||\cB_{M_P}^+|}{|\cB_{W'}^+||\cB_{L_Q}^+|}
&= \frac{|\cB_{V'}^+||\cB_{W}^+|}{|\cB_{W'}^+||\cB_{V}^+|} \\
&=\begin{cases}
q^{({n'}^2-n^2)-({m'}^2-m'-m^2+m) - \frac{1}{2}(a_{V'}' - a_V')} & (-\epsilon = 1), \\
q^{({n'}^2-n' - n^2+n) - ({m'}^2-m^2) + \frac{1}{2}(a_{W'}'-a_W')} & (-\epsilon = -1)
\end{cases}
\end{align*}
where 
\[
a_W' = \begin{cases} 0 & (\chi_W \mbox{ is unramified}), \\ -1 & (\chi_W\mbox{ is ramified}).
\end{cases}
\]
One can show that both coincide with $q^{2t}$. Hence we have the lemma.
\end{proof}

Moreover, we have
\begin{align*}
\frac{\prod_{\alpha \in \Sigma_{\rm red}(\overline{P})}[X_\alpha \cap K_{W'}:X_\alpha\cap\cB_{W'}^+]^{-1}}{\prod_{\beta \in \Sigma_{\rm red}(\overline{Q})}[X_\beta \cap K_{V'}:X_\beta\cap\cB_{V'}^+]^{-1}} &=q^{-2(n_0-m_0)t} \\
& = q^{-(1+\epsilon)t},
\end{align*}
and
\begin{align*}
\frac{|U_P\cap K_{W'}|\cdot|\overline{U_P}\cap K_{W'}|}
{|U_Q\cap K_{V'}|\cdot|\overline{U_Q}\cap K_{V'}|}
&= q^{-2(nt+\frac{1}{2}t(t-\epsilon))} \cdot q^{2(mt+\frac{1}{2}t(t+\epsilon))} \\
& = q^{-2(n-m) +2\epsilon t}\\
& = q^{-(1-\epsilon)t}.
\end{align*}
Hence we have
\begin{align*}
\frac{\deg \pi'}{\deg \sigma'}
&=\frac{\deg\pi}{\deg\sigma}\cdot
\gamma(s_0-\frac{l-1}{2}, \tau, \psi)\gamma(-s_0-\frac{l-1}{2}, \tau^\vee, \overline{\psi})\\
&=\frac{\deg\pi}{\deg\sigma}\cdot
\gamma(s_0, \tau, \psi)\gamma(-s_0, \tau^\vee, \overline{\psi})
\end{align*}
since $l=1$. Thus we have Proposition \ref{ind1st}.

Now, we prove the existence of the quadruple $(s_0, \pi, \sigma, \tau)$ as in Proposition \ref{ind1st} when either $V$ or $W$ is anisotropic.

\begin{prop}\label{ind_quadraple}
Suppose that $V$ is anisotropic. Then, there exists an irreducible supercuspidal representation $\pi$ of $G(W)$ such that $\Theta_\psi(\pi,V) \not=0$.
\end{prop}

\begin{proof}
We use the following see-saw diagram to prove this:
\[
\xymatrix{
G(V^\Box) \ar@{-}[rd]\ar@{-}[d] & G(W) \times G(W) \ar@{-}[d]\ar@{-}[ld]\\
G(V)\times G(V) & \Delta G(W) 
}.
\]
Let $\sigma$ be an irreducible representation of $G(V)$. Since $G(V)$ is anisotropic, $\sigma$ is supercuspidal, and then we have $\Theta_\psi(\sigma, V)$ is irreducible (\cite[p.~ 69]{MVW87}). Then Theorem \ref{Howe duality} \eqref{hd3} implies that $\Theta_\psi(\sigma, V)^\vee \cong \Theta_{\overline{\psi}}(\sigma^\vee, V)$. Hence, by the see-saw property we have
\begin{align*}
&\Theta_\psi(\sigma,W) \not=0 \Leftrightarrow {\rm Hom}_{\Delta G(W)}(\Theta_\psi(\sigma,W)\otimes \Theta_\psi(\sigma,W)^\vee, 1_{G(W)}) \not= 0\\
&\Leftrightarrow\sigma \boxtimes \sigma^\vee \mbox{ appears as a quotient of }\Theta_\psi(1_{G(W)}, V^\Box)|_{G(V)\times G(V)}.
\end{align*}
In the case where $W$ is anisotropic, the proposition is clear by the above observation.
Then suppose that $W$ is isotropic. This only occurs in the case where $\epsilon = -1$. Thus we have $\chi_W=1$. Hence, we have the isomorphism
\[
R_s\colon I^V(s,1_{F^\times}) \rightarrow \cS(G(V))
\]
by $[R_sf_s](g) = f_s(\iota(g,1))$ for $f_s \in I^V(s,1_{F^\times})$ and $g \in G(V)$.

\begin{lem}
For $u \in U(V^\tru)$ there is a unique element $g_u \in G(V)$ such that $\iota(g_u,1) \in P(V^\tru)\tau u$ for some $p \in P(V^\tru)$. Moreover, $u \mapsto g_u$ gives a homeomorphism
\[
U(V^\tru) \rightarrow G(V)\setminus\{1\}.
\]
\end{lem}

By this lemma, if we take a non-zero function $\varphi \in \cS(G(V))$ so that $\overline{\supp (\varphi)}\not\ni1$ and $\varphi(g) \geq 0$ for all $g \in G(V)$, then the integral defining $M(s,1_{F^\times})(R_s^{-1}\varphi)$ converges and $M(s,1_{F^\times})(R_s^{-1}\varphi) \not=0$ for all $s \in \C$.
On the other hand, if we denote by $W_i$ the $i$-dimensional $(-\epsilon)$-Hermitian space with $\chi_{W_i}=1_{F^\times}$ and by $l_i$ the integer $2i-2m-\epsilon$, then we have 
\[
\Theta_\psi(1_{W_i}, V^\Box) = \ker M(-\frac{l_i}{2},1_{F^\times})
\]
for $i=0, \ldots, n-1$ by [Yam11, Theorems 1.3, 1.4]. Thus, we have proved that 
\[
\sum_{i=0}^{n-1} R_{-l_i/2}(\Theta_\psi(1_{G(W_i0}, V^\Box)) \subsetneqq \cS(G(V)).
\]
Hence, there is an irreducible representation $\sigma$ of $G(V)$ such that $n^+(\sigma) \geq n$ and $n^-(\sigma) \geq n+1$. Since we have assumed $l=1$, the conservation relation (Proposition \ref{cons.rel}) says that $n^+(\sigma) + n^-(\sigma) = 2n+1$. Thus, we have $n^+(\sigma) =n$, and we have the lemma by putting $\pi = \Theta_\psi(\sigma, W)$.
\end{proof}

\begin{prop}\label{ind_q2}
Suppose that $W$ is anisotropic and $V$ is isotropic. 
Then, there is an irreducible representation $\pi$ of $G(W)$ such that $\theta_\psi(\pi, V)$ is non-zero supercuspidal.
\end{prop}

\begin{proof}
The situation in this proposition occurs only in the case where $\epsilon = 1$, $\dim W = 3$, $\dim V = 2$ and $\chi_W = \chi_V =1_{F^\times}$.
Then, as in \S\ref{accidental isom}, we have the accidental isomorphism 
\[
\widetilde{G}(W) \cong D_4^\times \times F^\times / \{(a, a^{-2}) \mid a \in F^\times\}
\]
where $D_4$ is a central division algebra of $F$ so that $[D_4:F] = 16$.
Now, we denote by $\pi_0$ an irreducible representation of $D_4^\times$ obtained  as follows:
let $\pi_1$ be an irreducible supercuspidal representation of $\GL_4(F)$ so that the image of its $L$-parameter is contained in $\Sp_4(\C)\times W_F$ (see \cite[\S4]{Mie20}). 
Then we denote by $\pi_0$ the irreducible representation of $D_4^\times$ associated with $\pi_1$ by the Jacquet-Langlands correspondence.
Since the central character of $\pi_0$ is trivial, we have the irreducible representation $\pi_0\boxtimes1_{F^\times}$ of $D_4^\times\times F^\times/\{ (a, a^{-2}) \mid a \in F^\times \}$. 
We may regard it as a representation of $\widetilde{G}(W)$ by the accidental isomorphism.
We denote by $\pi$ an irreducible component of the restriction of $\pi_0\boxtimes1$ to $G(W)$. 
Then, the square exterior $\gamma$-factor $\gamma(s, \phi_{\pi_0}, \wedge^2, \psi)$ has a pole at $s=1$ (see \cite{GNQ10}). Hence we have $\Theta_\psi(\pi, V)\not=0$  (see \cite[Theorem 6.1]{GT14}, and \cite[Proposition 3.3]{GT14}). Moreover, since $\pi\not=1_{G(W)}$, we have $m(\pi) > 0$. This forces that $m(\pi) = 2$, and $\theta_\psi(\pi, V)$ is supercuspidal. Hence we have the proposition.
\end{proof}

\begin{cor}\label{goalgoal}
There exist $(s_0, \pi, \sigma, \tau)$ as in Proposition \ref{ind1st} when either $V$ or $W$ is anisotropic.
\end{cor}

\begin{proof}
Take an irreducible supercuspidal representation $\rho$ of $\GL_2(F)$ so that the image of the L-parameter $\phi_\tau$ is contained in $\Sp_{2r}$ (Proposition \ref{pole aru} \eqref{pole aru1}). Moreover, we put $\tau = |\JL_F(\rho)|$. 

Suppose first that $V$ is anisotropic. Take $\pi$ as Proposition \ref{ind_quadraple}, and put $\sigma = \theta_\psi(\pi, V)$. Then, there exists a positive real number $s_0$ so that $\mu(s, \pi\boxtimes\tau\chi_V)$ has a pole at $s=s_0$ (Proposition \ref{pole aru}\eqref{pole aru2}).
Since the Godement-Jacquet L-factor $L(s, \tau)$ is equal to $1$, we have that $\mu(s, \sigma\boxtimes\tau\chi_W)$ also has a pole at $s = s_0$ by Proposition \ref{pl under theta}. This implies that the quadruple $(s_0, \pi, \sigma, \tau)$ satisfies the assumption of Proposition \ref{ind1st}.

Then, suppose that $W$ is anisotropic and $V$ is isotropic.
Take $\pi$ as in  Proposition \ref{ind_q2}, and put $\sigma = \theta_\psi(\pi, V)$. Then, by Proposition \ref{pole aru}, $\mu(s, \pi\boxtimes\tau\chi_V)$ has a pole at a positive real number $s_0$. Then, we have that $\mu(s, \sigma\boxtimes\tau\chi_W)$ also has a pole at $s = s_0$ as in the above discussion. Hence, the quadruple $(s_0, \pi, \sigma, \tau)$ satisfies the assumption \ref{ind1st}. Hence we have the corollary.
\end{proof}

Corollary \ref{goalgoal} completes the proof of \eqref{goal18}, and we finish the proof of Theorem \ref{fd theta1}.

%%%%%%%%%%%%%%%%%%%%%%%%%%%%%%       S19       %%%%%%%%%%%%%%%%%%%%%%%%%%%%%%%%%%%

\section{
            Determination of $\alpha_1$ and $\alpha_2$
            }\label{det alpha}

In this section, we complete the formulas of $\alpha_1(W)$ and $\alpha_2(V,W)$ even when both $V$ and $W$ are isotropic. 
Let $V$ be an $m$-dimensional right $\epsilon$-Hermitian space, and let $W$ be an $n$-dimensional left $(-\epsilon)$-Hermitian space. 
We assume that $2n - 2m -\epsilon = 1$. 
Take a basis $\ue = (e_1, \ldots, e_n)$ for $W$. 
%Note that in this section, we do not suppose that $R(\ue)$ is of the form \eqref{formform}. 
First, we have:

\begin{thm}\label{alpha2}
\begin{align*}
\alpha_2(V,W) &= |2|_F^{-2n\rho + n(n-\frac{1}{2})}\cdot |N(R(\ue))|^\rho\cdot\prod_{i=1}^{n-1}\frac{\zeta_F(1-2i)}{\zeta_F(2i)}\\
&\times\begin{cases}
2\chi_V(-1)^n\gamma(1-n, \chi_V, \psi)^{-1}\epsilon(\frac{1}{2}, \chi_V, \psi) & (-\epsilon = 1), \\
\epsilon(\frac{1}{2}, \chi_W, \psi) & (-\epsilon = -1).
\end{cases}
\end{align*}
\end{thm}

\begin{proof}
There is at least one irreducible square-integrable representation $\pi$ of $G(W)$ such that $\Theta_\psi(\pi, V)\not=0$ (this has been proved in \S\ref{ind_arg} by replacing $V$ with $V'$).
Then, comparing the formula of $\alpha_3(V,W)$ of Proposition \ref{alpha23} with its definition in Theorem \ref{fd theta1}, we obtain
\begin{align*}
&\frac{1}{2}\cdot \alpha_2(V,W) \cdot e(G(W)) \cdot |2|_F^{2n\rho - n (n-\frac{1}{2})}\cdot |N(R(\ue))|^{-\rho}
\cdot \prod_{i=1}^{n-1}\frac{\zeta_F(2i)}{\zeta_F(1-2i)} \\
 &\times\begin{cases} 
\chi_V(-1)^{n+1}\gamma(1 - n, \chi_V, \psi) & (-\epsilon=1), \\
\chi_W(-1)^{m+1}\epsilon(\frac{1}{2}, \chi_W, \psi) & (-\epsilon=-1)\end{cases}\\
&=\begin{cases}
\chi_V(-1)\epsilon(\frac{1}{2}, \chi_V, \psi) & (-\epsilon = 1), \\
\frac{1}{2}\cdot\chi_W(-1)^{m+1}\epsilon(\frac{1}{2}, \chi_W, \psi) & (-\epsilon = -1).
\end{cases}
\end{align*}
Hence, we have the claim.
\end{proof} 

Suppose that $-\epsilon = -1$. We denote by $W^{u}$ a $(-\epsilon)$-Hermitian space so that $\dim W^{u} = n$ and $W^{(u)}$ possesses a basis $\ue^{(u)}$ with $R(\ue^{(u)}) \in \GL_n(\cO_D)$. Then, by Theorem \ref{alpha2}, we  have:

\begin{cor}
\[
\alpha_2(V,W)=|N(R(\ue))|^\rho\cdot\alpha_2(V,W^{u}).
\]
\end{cor}

\begin{proof}
Since $|N(R(\ue^{(u)})| = 1$, the claim follows from Theorem \ref{alpha2}.  
\end{proof}

On the other hand, we may identify ${W^u}^\Box$ with $W^\Box$ by identifying $e_i'$ with ${e_i^{(u)}}'$ for $i=1, \ldots, 2n$. Then we can compare $\cI^{W^{u}}$ with $\cI^{W}$:

\begin{lem}\label{aaaaa}
For $\phi, \phi' \in \cS(V\otimes W^\trd) = \cS(V\otimes {W^u}^\trd)$, we have
\[
\cI^W(\phi,\phi') = \cI^{W^{u}}(\phi, \phi').
\]
\end{lem}

\begin{proof}
By writing down the definitions, we have the equation.
\end{proof}

Therefore, we have
\begin{align*}
\frac{\alpha_1(W)}{\alpha_1(W^{u})} 
&= \frac{\cE^{W^{u}}(\phi, \phi')}{\cE^{W}(\phi, \phi')} \\
& = \frac{\alpha_2(V,W^{u})}{\alpha_2(V,W)} \\
&= |N(R(\ue))|^{-\rho}.
\end{align*}
Thus, we have a general formula of $\alpha_1(W)$:
\begin{prop}\label{alpha_1 comp}
In the case $-\epsilon = -1$, we have
\[
\alpha_1(W) = |2|_F^{2n\rho}\cdot |N(R(\ue))|^{-\rho}\cdot q^{-(2\lfloor \frac{n}{2}\rfloor\lceil \frac{n}{2} \rceil - \lfloor\frac{n}{2} \rfloor)}\cdot\prod_{i=1}^n(1+q^{-(2i-1)}).
\]
\end{prop}

\begin{proof}
We already have a formula of $\alpha_1(W)$ either when $n_0=0$ or $n_0=1$ and $\chi_W$ is unramified (Proposition \ref{alpha1 min}). Hence, we have the proposition by Lemma \ref{aaaaa}.
\end{proof}

%%%%%%%%%%%%%%%%%%%%%%%%%%%%%%%%   LLC    %%%%%%%%%%%%%%%%%%%%%%%%%%%%%%%%%%%

\section{
           The formal degree conjecture
            }\label{FDCFDCFDC}

In this section, we explain the refined version of formal degree conjecture (\cite{GR10}).
Moreover, we give another proof of Theorem \ref{fd theta1} assuming the local Langlands correspondence and the formal degree conjecture.

%%%%%%%%%%%%%%%%%%%%%%%%%%%%%%%%%%%%%%%%%%%%%%%%%%%%%%%%%%%%%%%%%%%%%%%%%%%%
%%%%%%%%%%%%%%%%%%%%%%%%%%%%%%%%%%%%%%%%%%%%%%%%%%%%%%%%%%%%%%%%%%%%%%%%%%%%

Let $F$ be a non-Archimedean local field of characteristic $0$, and let $G$ be a connected reductive group over $F$. Let ${\widehat{G}}$ be the Langlands dual group of $G$, let ${}^L\!G$ be the L-group of $G$. By an L-parameter we mean the $\widehat{G}$-conjugacy class of L-homomorphisms 
\[
\phi\colon W_F \times \SL_2(\C) \rightarrow {}^L\!G
\]
whose images are not contained in any irrelevant parabolic subgroup of ${}^L\!G$ (\cite{Bor79}). 
In this section, We assume the Langlands correspondence, that is, we can associate an L-parameter with an irreducible representation of $G$. 
But we explain it more precisely for quaternionic unitary groups to clarify what we assume.  

\begin{hyp}\label{Lpar}
For an irreducible tempered representation $\pi$ of $G(W)$, there is a tempered $L$-parameter $\phi$ of $G(W)$ satisfying the following two properties.
\begin{itemize}
\item For an irreducible supercuspidal representation $\tau$ of $\GL_{2r}(F)$ and $s \in \C$ we have
\[
\mu(s, \pi\boxtimes|\JL_F|(\tau)) = \frac{\gamma(s, \std\circ\phi^\vee\otimes\phi_\tau, \psi)}{\gamma(1+s,  \std\circ\phi^\vee\otimes\phi_\tau, \psi)}\cdot\frac{\gamma(2s, \tau, \wedge^2, \psi)}{\gamma(1+2s, \tau, \wedge^2, \psi)}
\]
where $\phi_\tau$ is the L-parameter of $\tau$.
\item For a character $\chi$ of $F^\times$ we have
\[
\gamma^W(s, \pi\boxtimes\chi, \psi) = \gamma(s, \std\circ\phi\otimes\chi, \psi)
\]
where the left-hand side is the $\gamma$-factor defined in \S\ref{doubling and gamma}, and the right-hand side is the standard gamma factor.
\end{itemize}
\end{hyp}

It is known that the equality of the Plancherel measure characterizes $\phi$ uniquely (\cite[Lemma 12.3]{GS12}, \cite[p.~652]{GI14}). Now we consider a general connected reductive group $G$ again and we discuss the internal structure of the L-packet $\Pi_\phi(G(F))$.
We denote by $Z(\widehat{G})$ the center of $\widehat{G}$, by $\widehat{G}_{\rm ad}$ the quotient $\widehat{G}/Z(\widehat{G})$, and by $\widehat{G}_{\rm sc}$ the simply connected covering group of $\widehat{G}_{\rm ad}$. Moreover, we denote by $C_\phi$ be the centralizer of the image $\Im \phi$ of $\phi$ in $\widehat{G}$,  by $\Gamma$ the Galois group of $F^{\rm s}/F$ where $F^{\rm s}$ is the separable closure of $F$, by $S_\phi$ the quotient $C_\phi/Z(\widehat{G})^\Gamma \subset \widehat{G}_{\rm ad}$, and by $\widetilde{S}_\phi$ the preimage of $S_\phi$ by $\widehat{G}_{\rm sc}\rightarrow \widehat{G}_{\rm ad}$. We call the component group of $\widetilde{S}_\phi$ the S-group of $\phi$, and we denote it by $\widetilde{\cS}_\phi$.
We choose a character $\zeta_G$ of $Z(\widehat{G}_{\rm sc})$ which is associated with $G$ via the composition of the maps
\[
{\rm Hom}(Z(\widehat{G}_{\rm sc}), \C^\times)\rightarrow {\rm Hom}(Z(\widehat{G}_{\rm sc})^\Gamma, \C^\times) \rightarrow \{ \mbox{Inner forms of $G^{\rm qs}$}\}
\]
where $G^{\rm qs}$ is the quasi-split inner form of $G$, the first map is the restriction, and the second map is the isomorphism of Kottwitz \cite{Kot84}. Let $A$ be the maximal split torus of the center of $G$. Then $\widehat{G/A}$ is a subset of $\widehat{G}$. We denote by $C_\phi'$ the intersection $C_\phi\cap \widehat{G/A}$ and by ${\rm Irr}(\widetilde{\cS}_\phi, \zeta_G)$ the set of irreducible representations $\rho$ of $\widetilde{\cS}_\phi$ so that 
\[
{\rm Hom}_{Z(\widehat{G}_{\rm sc})}(\zeta_G, \rho)\not=0.
\]

\begin{conj}\label{GRFDC}
Let $\phi$ be a tempered L-parameter. Then there is a bijection 
\begin{align}\label{LLC2}
\Pi_\phi(G) \rightarrow {\rm Irr}(\widetilde{\cS}_\phi, \zeta_G)
\end{align}
such that for any square-integrable representation $\pi \in \Pi_\phi(G)$ we have
\[
\deg(\pi) = \zeta_\pi \frac{\dim\eta_\pi}{\# C_\phi'}\gamma(0, \pi, {\rm Ad}, \psi)
\]
where ${\rm Ad}\colon {}^L\!G \rightarrow \GL({\rm Lie}(\widehat{G}_{\rm ad}))$ is the adjoint representation, and 
\[
\zeta_\pi = \frac{|\gamma(0, {\rm St}, {\rm Ad}, \psi)|}{\gamma(0, {\rm St}, {\rm Ad}, \psi)}\frac{\epsilon(\frac{1}{2}, {\rm St}, {\rm Ad}, \psi)}{\epsilon(\frac{1}{2}, \pi, {\rm Ad}, \psi)} \in \{\pm 1\}
\]
where ${\rm St}$ is the Steinberg representation of $G(W)$.
\end{conj}

We denote by $\eta_\pi$ the image of $\pi$ via the map \eqref{LLC2}. 

\begin{rem}
It is expected that the characters of the irreducible representations belonging to $\Pi_\phi(G(F))$ satisfy linear equations called the ``endoscopic character relations''. (See for example \cite{Kal16}.) But we do not discuss them in this paper.
\end{rem}

Now we deduce some properties of Langlands parameters and local theta correspondence. Assume that $l=1$. For a unitary character $\chi$ of $F^\times$, we also denote by $\chi$ the corresponding character of the Weil group $W_F$ via the local class field theory. 
For a non-negative integer $k$, let $Q_k$ be the $k$-dimensional quadratic space over $\C$. 
Then, we have $\widehat{G(V)} = \SO(Q_M)$ and $\widehat{G(W)} = \SO(Q_{M + 1})$ where $M = 2m + (1+\epsilon)/2$. Fix an isometric embedding $u\colon Q_M \rightarrow Q_{M+1}$. Then, $u$ yields the embedding $\xi_0\colon {\rm O}(Q_M) \rightarrow \SO(Q_{M+1})$. Moreover, we fix an element $\varepsilon \in {\rm O}(Q_M)$ so that $\det(\varepsilon) = -1$. We extends the embedding $\xi_0|_{\SO(Q_M)}$ of the dual groups to an L-embedding from ${}^L\!G(V)$ into ${}^L\!G(W)$ by
\[
\xi(w, g) = (w, g \xi_0(\varepsilon)^{a_V(w)}) \quad (w \in W_F, g \in \SO(M, \C))
\]
where $a_V(w) = (1 - \chi_V(w))/2$ for $w \in W_F$.

\begin{prop}\label{weak prasad}
Assume that Hypothesis \ref{Lpar} holds and that $l=1$.
\begin{enumerate}
\item For an irreducible tempered representation $\pi$ of $G(W)$, $\theta_\psi(\pi,V)$ is non-zero if and only if $\std\circ\phi_\pi$ contains $\chi_V$ as representations of $W_F\times \SL_2(\C)$. \label{L1}
\item For an irreducible tempered representation $\sigma$ of $G(V)$, $\theta_\psi(\sigma, W)$ is non-zero if and only if $\std\circ\phi_\sigma$ does not contains $\chi_W$ as representations of $W_F\times\SL_2(\C)$.\label{L2}
\item For an L-parameter $\phi$ of $G(W)$, by $\vartheta(\phi)$ we mean the set of L-parameters $\phi'$ of $G(V)$ so that $\xi\circ\phi' = \phi$ as L-parameters of $G(W)$. Then, the local theta correspondence defines a bijection
\[
\theta(-, V)\colon \Pi_\phi(G(W)) \rightarrow \bigcup_{\phi' \in \vartheta(\phi)}\Pi_{\phi'}(G(V))
\]
if $\phi$ is tempered and $\vartheta(\phi) \not=\varnothing$. \label{L3}
\end{enumerate}
\end{prop}

\begin{proof}
Let $\sigma$ be an irreducible tempered representation of $G(V)$ so that $\theta_\psi(\sigma, W)\not=0$. Put $\pi = \theta_\psi(\sigma, W)$. Then as in \cite[Theorem C.5]{GI14}, we have 
\[
\std\circ\phi_\pi\otimes \chi_V = (\std\circ\phi_\sigma\otimes\chi_W)\oplus 1_{W_F}
\]
as representations of $W_F\times \SL_2(\C)$. This proves the ``only if'' part of \eqref{L1}.
The ``if'' part of the property \eqref{L1} is obtained by the similar argument to \cite[Theorem 6.2]{HKS96} of Harris, Kudla and Sweet. Let $\pi$ be an irreducible tempered representation of $G(W)$ such that $\std\circ\phi_\pi$ contains $\chi_V$. Then the standard L-function $L(s, \pi\boxtimes\chi_V)$ has a pole at $s = 0$. Since $\pi$ is tempered, $L(s, \pi\boxtimes\chi_V)$ does not have a pole in $\Re s <0$ (\cite{Yam11}). Hence we have
\[
Z(M(s, \chi_V)f^{(s)}, \xi_\pi) \sim \frac{L(-s+\frac{1}{2}, \pi\boxtimes\chi_V)}{\zeta(2s - 1)}\cdot\frac{Z(f^{(s)}, \xi_\pi)}{L(s+\frac{1}{2}, \pi\boxtimes\chi_V)}.
\]
Here, by $f_1 \sim f_2$ we mean $f_1/f_2$ is holomorphic at $s = \frac{1}{2}$. Hence, by \cite[Theorem 5.2]{Yam11} we can conclude that $Z(-, \xi_\pi)$ is non-zero on $\Theta_\psi(1_{G(V)}, W^\Box)\subset I^W(-\frac{1}{2}, \chi_V)$. Then, by considering the see-saw diagram
\[
\xymatrix{
G(W^\Box) \ar@{-}[rd]\ar@{-}[d] & G(V) \times G(V) \ar@{-}[d]\ar@{-}[ld]\\
G(W)\times G(W) & \Delta G(V) 
},
\]
we have
\[
\Hom_{\Delta G(V)}(\Theta_\psi(\pi, V)\boxtimes\Theta_{\overline{\psi}}(\pi^\vee, V), 1_V) = \Hom_{G(V)\times G(V)}(\Theta_\psi(1_V, W^\Box), \pi\boxtimes\pi^\vee)\not=0.
\]
Thus, we have $\theta_\psi(\pi, V)\not=0$. 
By combining the property \eqref{L1} with the conservation relation (Proposition \ref{cons.rel}), we have the property \eqref{L2}. The property \eqref{L3} is a consequence of \eqref{L1} and \eqref{L2}.
\end{proof}

%%%%%%% Revision
In the rest of this section, we assume Hypothesis \ref{Lpar} and that Conjecture \ref{GRFDC} is true.

\begin{prop}\label{dim irr S}
Let $\pi$ be a square-integrable irreducible representation of $G(W)$, and let $\sigma = \theta_\psi(\pi, V)$. Suppose that $\sigma\not=0$. Then we have
\[
\frac{\dim \eta_\pi}{\dim \eta_\sigma} =  \begin{cases} 1 & (\epsilon = 1), \\ 1 & (\epsilon = -1, \phi_\sigma^\varepsilon \not\cong \phi_\sigma), \\ 2 & (\epsilon = -1, \phi_\sigma^\varepsilon \cong \phi_\sigma).\end{cases}
\]
\end{prop}

\begin{proof}
Let $\phi_\pi$ be the L-parameter of $\pi$. Then, it is known that all representations in ${\rm Irr}(\widetilde{\cS}_{\phi_\sigma}, \zeta_{G(W)})$ have the same dimension $[\widetilde{\cS}_{\phi_\pi}: \widetilde{Z}_{\phi_\pi}]^{\frac{1}{2}}$ where $\widetilde{Z}_{\phi_\pi}$ denotes the center of $\widetilde{\cS}_{\phi_\pi}$ (\cite[Lemma 9.2.2]{Art13}). 
First, we claim that 
\begin{align}\label{Sgrpxi}
C_{\phi_\pi} \subset C_{\phi_\pi} \cdot \xi(g_1)  \mbox{ for some } g_1 \in {\rm O}(Q_M), \det(g_1) = -1.
\end{align} 
Let $g \in C_{\phi_\pi}$. Since $\std\circ\phi_\pi = \std\circ\phi_\sigma\otimes \chi_W\chi_V + \chi_V\boxtimes 1_{{\rm SL}_2(\C)}$, and $\std\circ\phi_\sigma\otimes \chi_W\chi_V$ does not contains $\chi_V\boxtimes 1_{{\rm SL}_2(\C)}$ (Proposition \ref{weak prasad}), the action of $g$ on $Q_{M+1}$ preserves the subspace $u(Q_M)$. Hence we have $g \in {\rm O}(Q_M)$. Thus, if $g'$ is also a element if $C_{\phi_\pi}$, then we have $g'g^{-1} \in C_{\phi_\pi}$. This implies the claim \eqref{Sgrpxi}.

Suppose that $\epsilon = 1$. Then, we have $C_{\phi_\pi} \supset \xi(C_{\pi_\sigma}) \times \{\pm 1\}$. By \eqref{Sgrpxi}, we have $C_{\phi_\pi} = \xi(C_{\pi_\sigma}) \times \{\pm 1\}$, and we have $[\widetilde{\cS}_{\phi_\pi}:\xi(\widetilde{\cS}_{\phi_\sigma})] = [\widetilde{Z}_{\phi_\pi}:\widetilde{Z}_{\phi_\sigma}] = 2$.
Thus we have $\dim \eta_\pi = \dim \eta_\sigma$.

Suppose that $\epsilon = -1$ and $\phi_\sigma^\varepsilon \not\cong\phi_\sigma$. Then, there is no element $g \in \SO(Q_M)$ so that $\xi(g \varepsilon) \in C_{\phi_\pi}$. Then, by \eqref{Sgrpxi}, we have that $\xi$ is a bijection between $C_{\phi_\sigma}$ and $C_{\phi_\pi}$.
Thus we have $\widetilde{\cS}_{\phi_\sigma} \cong \widetilde{\cS}_{\phi_\pi}$ and $\dim \eta_\pi = \dim \eta_\sigma$.

Finally, suppose that $\epsilon = -1$ and $\phi_\sigma^\varepsilon \cong\phi_\sigma$. Then, there exists an element $g \in \SO(Q_M)$ so that $\xi(g \varepsilon) \in C_{\phi_\pi}$. Hence we have $[\widetilde{\cS}_{\phi_\pi}:\xi(\widetilde{\cS}_{\phi_\sigma})] = 2$ by \eqref{Sgrpxi}. Then, we have 
\[
\eta_\pi \subset \Ind_{\xi(\widetilde{\cS}_{\phi_\sigma})}^{\widetilde{\cS}_{\phi_\pi}} \eta
\]
for some irreducible representation $\eta$ of $\xi(\widetilde{\cS}_{\phi_\sigma})$. Thus we have $\dim \eta_\pi \leq 2\dim \eta = 2 \dim\eta_\sigma$. Besides, since the action of $g\varepsilon$ on $Z({\rm Spin}(Q_M))$ is non-trivial, we have $[\widetilde{Z}_{\phi_\sigma}:\widetilde{Z}_{\phi_\pi}] > 1$, which implies $\dim \eta_\pi \geq 2 \dim \eta_\sigma$. Thus we have $\dim \eta_\pi  = 2\dim \eta_\sigma$. Therefore, we prove the proposition.
\end{proof}

Now we give the alternative proof of Theorem \ref{fd theta1}. By the proof of Proposition \ref{dim irr S}, we have
\[
\frac{\dim\eta_\pi}{\dim\eta_\sigma}\cdot \frac{\# C_{\phi_\sigma}}{ \# C_{\phi_\pi}} = \begin{cases} \frac{1}{2} & (\epsilon = 1),\\ 1 & (\epsilon = -1). \end{cases}
\]
Moreover, as representations of $W_F\times \SL_2(\C)$, we have
\begin{align}\label{ad_Lpar}
{\rm Ad}\circ\phi_\pi = {\rm Ad}\circ\phi_\sigma \oplus (\std \circ \phi_\sigma) \otimes \chi_W.
\end{align}
By the substitution of the formulas of the formal degrees (Conjecture \ref{GRFDC}), we have
\begin{align*}
&\frac{\deg\pi}{\deg\sigma}\cdot c_\sigma(-1)\cdot\gamma(0, \sigma\times\chi_W, \psi)^{-1} \\
&=\zeta_\pi\zeta_\sigma^{-1}\cdot  c_\sigma(-1) \cdot \frac{\dim\eta_\pi}{\dim\eta_\sigma}\cdot \frac{\# C_{\phi_\sigma}}{ \# C_{\phi_\pi}}\cdot \frac{\gamma(0, \pi, {\rm Ad}, \psi)}{\gamma(0, \sigma, {\rm Ad}, \psi)} \cdot \gamma^V(0, \sigma\times\chi_W, \psi)^{-1} \\
&=\begin{cases}
\frac{1}{2}  c_\sigma(-1)\zeta_\pi\zeta_\sigma^{-1} & (\epsilon = 1), \\  c_\sigma(-1)\zeta_\pi\zeta_\sigma^{-1} & (\epsilon = -1). 
\end{cases}
\end{align*}
It remains to show that 
\[
\zeta_\pi\zeta_\sigma^{-1} =  c_\sigma(-1)\chi_W(-1)^n\cdot\epsilon(\frac{1}{2}, \chi_V\chi_W, \psi)^{-1}.
\]
It is known that 
\[
{\rm Ad}\circ \phi_{\rm St} = \oplus_{d\geq1} E_d' \otimes r_{2d-1}
\]
as $W_F\times\SL_2(\C)$-modules (\cite[\S3.3]{GR10}). Here $E_d'$ is the $W_F$-modules obtained by the action by $\Gamma$ on the sub-module of homogeneous elements of degree $d$ of $E'$ (see \S\ref{Haar1}), and $r_{2d-1}$ is the unique $2d - 1$-dimensional irreducible representation of $\SL_2(\C)$. Then, by using the formula of the structure of the graded module $E'$ (see the proof of Proposition \ref{Iwa_vol}), we have 
\[
\gamma(0, {\rm St}, {\rm Ad}, \psi) = \epsilon(\frac{1}{2}, {\rm St}, {\rm Ad}, \psi) \cdot |\gamma(0, {\rm St}, {\rm Ad}, \psi)|,
\]
and thus we have $\zeta_\pi = \epsilon(\frac{1}{2}, \pi, {\rm Ad}, \psi)^{-1}$. Since $\sigma$ is square-irreducible, the $L$-factor $L(s, \sigma \boxtimes\chi_W)$ does not have a pole at $s = 1/2$. Hence, by \eqref{ad_Lpar}, we have
\[
\zeta_\pi\zeta_{\sigma}^{-1} = \epsilon(\frac{1}{2}, \sigma\boxtimes\chi_W, \psi)^{-1}.
\]
Moreover, by \cite[Proposition 8.2]{Kak20}, this is equal to
\[
 c_\sigma(-1)\chi_W(-1)^m\epsilon(\frac{1}{2}, \chi_V\chi_W, \psi)^{-1}.
\]
Thus, we complete the proof of Theorem \ref{fd theta1} admitting that Hypothesis \ref{Lpar} and Conjecture \ref{GRFDC} hold.

%%%%%%%%%%%%%%%%%%%%%%%%%%%%%%%%   S21     %%%%%%%%%%%%%%%%%%%%%%%%%%%%%%%%%%%

\section{
            Formal degree conjecture for the non-split inner forms of $\Sp_4$, ${\rm GSp}_4$
            }\label{FDC}

The local Langlands correspondence for the non-split inner forms of ${\rm GSp}_4$ and $\Sp_4$ has been constructed by Gan and Tantono \cite{GT14} and Choiy \cite{Cho17}. 
Note that one can show the equation of the Plancherel measure by Proposition \ref{pl under theta} and accidental isomorphisms. Hence Hypothesis \ref{Lpar} is true in these cases, and we have a bijection $\pi\mapsto\eta$. Thus, the refined formal degree conjecture for these groups can be stated unconditionally, and it suffices to show that the bijection $\pi\mapsto \eta$ satisfies the formula in Conjecture \ref{GRFDC}. In this section, we prove this as an application of Theorem \ref{fd theta1}.
We denote by $G_{1,1}$, $H_{1,1}$, and $H_{3,0}$ the isometry groups of
\begin{itemize}
\item the two-dimensional Hermitian space $W$ with $\chi_W=1_{F^\times}$, 
\item the two-dimensional skew-Hermitian space $W$ with $\chi_W=1_{F^\times}$, 
\item the three-dimensional skew-Hermitian space $W$ with $\chi_W = 1_{F^\times}$
\end{itemize}
respectively. We also denote by $\widetilde{G}_{1,1}$, $\widetilde{H}_{1,1}$, and $\widetilde{H}_{3,0}$ their similitude groups respectively.  In this section, we assume that $G$ is one of $G_{1,1}, H_{1,1}, H_{3,0}$, and we assume that $\widetilde{G}$ is the corresponding similitude group. We denote by $\p\colon \widehat{\widetilde{G}} \rightarrow \widehat{G}$ the projection of \cite[Theorem 8.1]{Lab85}.
Let $\widetilde{\phi}$ be an $L$-parameter for $\widetilde{G}$. We denote by $\phi\colon W_F\times\SL_2(\C)\rightarrow {}^L\!G$ the $L$-parameter given by the composition $\p\circ\widetilde{\phi}$. 
According to \cite[\S7.3]{Cho17}, the $L$-parameter $\phi$ of $\widetilde{G}_{1,1}$ is classified into one of the following ``Case I-(a), Case I-(b), Case II, Case III'';
\begin{itemize}
\item \underline{Case I-(a)}: the parameter $\widetilde{\phi}$ comes from that of  $\widetilde{H}_{1,1}$, and the cardinality of the $L$-packet $\Pi_{\widetilde{\phi}}$ is equal to $2$, and the action of $\Hom(W_F, \C^1)$ is not transitive;
\item \underline{Case I-(b)}: the parameter $\widetilde{\phi}$ comes from that of $\widetilde{H}_{1,1}$, and the cardinality of the $L$-packet $\Pi_{\widetilde{\phi}}$ is equal to $2$, and the action of $\Hom(W_F, \C^1)$ is transitive;
\item \underline{Case II}: the parameter $\widetilde{\phi}$ comes from that of $\widetilde{H}_{1,1}$, and  the cardinality of the $L$-packet $\Pi_{\widetilde{\phi}}$ is equal to $1$;
\item \underline{Case III}: the parameter $\widetilde{\phi}$ comes from that of $\widetilde{H}_{3,0}$, and the cardinality of the $L$-packet $\Pi_{\widetilde{\phi}}$ is equal to $1$.
\end{itemize}
Denote by $X(\widetilde{\phi})$ the stabilizer
\[
\{ a \in H^1(W_F, \widehat{{\rm GL}_1}) \mid a \widetilde{\phi} = \widetilde{\phi} \mbox{ as $L$-parameters }\}.
\]
Then we have an exact sequence
\[
\cS_{\widetilde{\phi}} \rightarrow \cS_\phi \rightarrow X(\widetilde{\phi}) \rightarrow 1.
\]
In the case $\phi$ is a tempered parameter, the first map is injective (\cite[Lemma4.9]{Cho19}). 

\subsection{
                Restriction of representations from $\widetilde{G}$ to $G$
                }\label{inn Sp_4 GSp_4}

It is known that such restriction problems have much information on Langlands parameters for $G$. We only use the following lemma:
\begin{lem}\label{res and dim}
Let $\widetilde{\pi}$ be an irreducible representation of $\widetilde{G}$. Then, we have a decomposition 
\[
\widetilde{\pi}|_{G} = (\bigoplus_{i=1}^t \pi_i)^{\oplus k}
\]
where $\pi_1, \ldots, \pi_t$ are irreducible representations of $G$ and
\[
k= \begin{cases}\frac{1}{2}\dim\eta & (G = G_{1,1} \mbox{ and $\widetilde{\pi}$ has the $L$-parameter of Case I-(b)}), \\ \dim \eta & (\mbox{otherwise}).\end{cases}
\]
\end{lem} 

\begin{proof}
It is obtained by \cite[Theorems 5.1, 6.1, 7.5]{Cho17}.
\end{proof}

In this paper, we need this lemma to prove the following two lemmas.

\begin{lem}\label{simili isom}
Let $\pi$ be a square-integrable irreducible representation of $G$, let $(\phi, \eta)$ be its Langlands parameter, let $\widetilde{\pi}$ be an irreducible representation of $\widetilde{G}$ so that its restriction $\widetilde{\pi}|_{G}$ to $G$ contains $\pi$, and let $(\widetilde{\phi}, \widetilde{\eta})$ be the Langlands parameter of $\widetilde{\pi}$.
Then, we have 
\[
\deg \widetilde{\pi} = \frac{\dim \widetilde{\eta}}{\dim \eta} \cdot \frac{\# C_\phi}{\# C_{\widetilde{\phi}}'}\cdot\deg\pi, \mbox{ and } {\rm Ad}\circ \widetilde{\phi} = {\rm Ad}\circ \phi.
\] 
\end{lem}

\begin{proof}
Put
\[
X(\widetilde{\pi}) = \{ \chi \in \Hom(F^\times, \C^\times)\mid (\chi\circ\lambda)\widetilde{\pi}  \cong \widetilde{\pi}\}.
\]
Then the reciprocity map of the local class field theory induces an embedding $X(\widetilde{\pi})\rightarrow X(\widetilde{\phi})$. Moreover, we have 
\[
[X(\widetilde{\phi}): X(\widetilde{\pi})] =  \begin{cases} 2 & (G = G_{1,1} \mbox{ and $\widetilde{\pi}$ has the $L$-parameter of Case I-(b)}), \\ 1 & (\mbox{otherwise}).\end{cases}
\]
Hence, by \cite[Lemma 13.2]{GI14} and by Lemma \ref{res and dim}, we have
\begin{align*}
\deg \pi &= \frac{\# Z'(\widehat{\widetilde{G}})}{\# Z(\widehat{G})} \cdot \frac{k}{\# X(\widetilde{\pi})}\cdot \deg \widetilde{\pi} \\
&= \frac{\# Z'(\widehat{\widetilde{G}})}{\# Z(\widehat{G})} \cdot\frac{\dim \eta\cdot \#\cS_{\widetilde{\phi}}}{\#\cS_{\phi}} \cdot \deg \widetilde{\pi} \\
&= \frac{\dim \eta\cdot \# C_{\widetilde{\phi}}'}{\# C_{\phi}} \cdot \deg \widetilde{\pi}. 
\end{align*}
Moreover, since the projection $\p\colon \widehat{\widetilde{G}} \rightarrow \widehat{G}$ factors through the adjoint map ${\rm Ad}$, we have 
\begin{align*}
{\rm Ad}\circ \widetilde{\phi} &= {\rm Ad}\circ\p\circ\phi \\
&= {\rm Ad}\circ \phi.
\end{align*} 
Hence, we have the lemma.
\end{proof}

\begin{lem}\label{bbbbb}
Let $\pi$ be a square-integrable irreducible representation of $G_{1,1}$, 
and let $\sigma$ be an irreducible representation of either $H_{1,1}$ or $H_{3,0}$ associated with $\pi$ by the local theta correspondence. We assume that $\sigma\not=0$. We denote by $(\phi_\pi, \eta_\phi)$ (resp.~ $(\phi_\sigma, \eta_\sigma)$) the Langlands parameter associated with $\pi$ (resp.~ $\sigma$). Then, we have
\begin{align}\label{bbbbb1}
\frac{\dim \eta_\sigma}{\dim \eta_\pi}  = \begin{cases} \frac{1}{2} & (\mbox{$\pi$ has the $L$-parameter of Case I-(b)}), \\ 1 & (\mbox{otherwise}) \end{cases}
\end{align}
and we have
\[
\frac{\# C_{\phi_\sigma}'}{\# C_{\phi_\pi}'} = \begin{cases} \frac{1}{2} & (\mbox{$\pi$ has the $L$-parameter of Case I-(b), III}), \\ 1 & (\mbox{otherwise}). \end{cases}
\]
\end{lem}

\begin{proof}
Note that discrete series parameters are of Case I and Case III. By \cite[Proposition 3.3]{GT14} and Lemma \ref{res and dim}, we have \eqref{bbbbb1}. The remaining equality is obtained by the case-by-case discussion in \cite[p.~ 1867--1874]{Cho17}. 
\end{proof}

\subsection{
		   Refined formal degree conjecture
		   }

In this section, we discuss the refined formal degree conjecture \cite[Conjecture 7.1]{GR10}.
We first prove it for inner forms of $\GL_N$. Note that $\#C_\phi' = N$ if $\phi$ is a discrete parameter for $\GL_N$.
\begin{lem}
Let $G$ be an inner form of $\GL_N$, and let $\pi$ be a square-integrable irreducible representation of $G$. Then, we have
\[
\deg(\pi) =  c_\pi(-1)^{N-1}\cdot\frac{1}{N} \cdot \gamma(0, \pi, {\rm Ad}, \psi).
\]
Here, ${\rm Ad}$ is the adjoint representation of ${}^L\!G$ on $\mathfrak{sl}_N(\C)$.
\end{lem}

\begin{proof}
By \cite[\S3.1]{HII08}, we have 
\[
\deg(\pi) = \frac{1}{N}\cdot|\gamma(0, \rho, {\rm Ad}, \psi)|.
\]
Take a square-integrable representation $\rho$ of $\GL_{N}(F)$ so that $\pi = |\JL_F|(\rho)$.  Then, by \cite[Proposition 14.1]{GI14}, we have
\begin{align*}
\frac{\gamma(0, \pi, {\rm Ad}, \psi)}{|\gamma(0, \pi, {\rm Ad}, \psi)|} 
&= \frac{\gamma(0, \rho, {\rm Ad}, \psi)}{|\gamma(0, \rho, {\rm Ad}, \psi)|} \\
&= \omega_\rho(-1)^{N-1}\\
&= \omega_{\pi}(-1)^{N-1}.
\end{align*}
Thus, by the positivity of $\deg \pi$, we have the lemma. 
\end{proof}

Let $\widetilde{G}$ be one of $G_{1,1}$, $H_{1,1}$, $H_{3,0}$, $\widetilde{G}_{1,1}$, $\widetilde{H}_{1,1}$, and $\widetilde{H}_{3,0}$. 
Then the refined formal degree conjecture for $\widetilde{G}$ is true: 

\begin{thm}\label{Sp^*}
Let $\pi$ be a square-integrable irreducible representation of $\widetilde{G}$, and let $(\phi,\eta)$ be its Langlands parameter.
Then we have
\[
\deg \pi =  c_\pi(-1) \cdot \frac{\dim \eta}{\# C_\phi'} \cdot \gamma(0, {\rm Ad}\circ\phi, \psi).
\]
\end{thm}

\begin{proof}
When $G'$ is either $\widetilde{H}_{1,1}$ or $\widetilde{H}_{3,0}$, we have the claim because of the accidental isomorphisms
\begin{align*}
\widetilde{H}_{1,1} &= D^\times \times \GL_2(F) / \{ (t, t^{-1}\cdot I_2) \mid t \in F^\times\}, \\
\widetilde{H}_{3,0} &= D_4^\times \times F^\times/ \{ (t, t^{-2}) \mid t \in F^\times\}
\end{align*}
as in \S\ref{accidental isom}. Here, $D_4$ is a central division algebra of $F$ with $[D_4:F] = 16$.
Hence, we have the claim for $H_{1,1}$ and $H_{3,0}$ by Lemma \ref{simili isom}. 
When $G' = G_{1,1}$, we have the claim by Theorem \ref{fd theta1}, the equation \eqref{ad_Lpar} and Lemma \ref{bbbbb}. Hence, we also have the claim for $\widetilde{G}_{1,1}$. Thus we have the theorem. 
\end{proof}

%%%%%%%%%%%%%%%%%%%%%%%%%%%%%%%   Appendix   %%%%%%%%%%%%%%%%%%%%%%%%%%%%%%%%%

\section{
	     Appendix: an explicit formula of zeta integrals
	     }\label{App}

In \cite[Proposition 8.3]{Kak20}, the author computed the doubling zeta integral of right $K({\ue'}^\Box)$-invariant sections. 
However, the formula does not tell us about the constant term and a certain multiplier polynomial factor $S(T)$. 
In this section, we complete the formula by applying the formula of $\alpha_1(W)$. 
Note that there are errors in \cite[Proposition 8.3]{Kak20}. We also point out and correct them. 
In this section, we assume {\bf the} {\bf residue} {\bf characteristic} {\bf of} {\bf $F$} {\bf is} {\bf not} {\bf $2$}. 
We note that the results in this Appendix are not used in this paper but had been used in the previous version. Actually, we can prove Proposition \ref{alpha1 min} by them if we assume $q\not|2$. 

Fix a basis $\underline{e}$ of $W$ as in \S\ref{basis for WV}. 
We denote by $\ue_0$ the basis $e_{r+1}, \ldots, e_{r+n_0}$ for $W_0$. Moreover, we may assume that
\[
R_0 = R(\ue_0) = \begin{cases}1 & (-\epsilon = 1, n_0=1), \\
                          \alpha & (-\epsilon = -1, n_0=1), \\
                          \varpi_D^{-1}& (-\epsilon = -1, n_0=1), \\
                          \diag(\varpi_D^{-1}, \alpha\varpi_D^{-1}) & (-\epsilon = -1, n_0=2 \mbox{ with $\chi_W$ unramified}), \\
                          \diag(\alpha, \varpi_D^{-1}) & (-\epsilon = -1, n_0=2 \mbox{ with $\chi_W$ ramified}),\\
                          \diag(\alpha, \varpi_D^{-1}, \beta^{-1}) & (-\epsilon = -1, n_0=3).
        \end{cases}
\]
Here, $\alpha$ is defined in \S\ref{quaternion}, and $\beta$ is an element of $D$ so that $\ord_D(\beta) = -1$, $T_D(\beta) = 0$ and $\beta^2 = \alpha^{2}\varpi_D^{2}$.
We recall that we have put $n_0 = \dim W_0$ and $r = \frac{n-n_0}{2}$.  
By this basis, we regard $G(W)$ as a subgroup of $\GL_n(D)$. Then, put 
\[
C_1 \coloneqq   \{ g \in G(W)\cap \GL_n(\cO_D) \mid (g-1)R({\ue}) \in {\rm M}_n(\cO_D)\}.
\]
Note that $C_1$ is an open compact subgroup of $G(W)$.
Let $X_i$ be a subspace of $X$ spanned by $e_1, \ldots, e_i$. We denote by $\ff$ the flag  
\[
\ff: 0 = X_0 \subsetneqq X_1 \subsetneqq \cdots \subsetneqq X_r = X,
\]
and by $B$ the minimal parabolic subgroup preserving $\ff$. 

\begin{prop}\label{Iwasawa C_1}
We have $G(W) = B \cdot C_1$.
\end{prop}

\begin{proof}
We use the setting and the notation of \S\ref{BTtheory} in the proof of this proposition.
By the result of Bruhat and Tits \cite[Th\'{e}or\`{e}me (5.1.3)]{BT72} and that of Heines and Rapoport \cite[Appendix, Proposition 8]{PR08}, we have the decomposition
\[
G(W) = B \cdot N_{G(W)}(S)\cdot\cB.
\]
Since $B \supset Z_{G(W)}(S)$, we can take a representative system $w_1, \ldots, w_t$ for $B\backslash (B\cdot N_{G(W)}(S))$ so that $w_i \in C_1$ for $i=1, \ldots, t$. 
Moreover, $X_{a,0} \subset C_1$ for $a \in \Phi^+$ and $X_{a,\frac{1}{2}} \subset C_1$ for $a \in \Phi^-$. Hence, by Lemma \ref{decomp iwahori}, we have
\begin{align*}
B \cdot N_{G(W)}(S) \cdot \cB 
&= \bigcup_{i=1}^t B \cdot w_i \cdot Z_{G(W)}(S)_1 \cdot  \prod_{a \in \Phi^+}X_{a,0} \cdot \prod_{a \in \Phi^-}X_{a,\frac{1}{2}} \\
&= \bigcup_{i=1}^t B\cdot Z_{G(W)}(S)_1 \cdot w_i \cdot \prod_{a \in \Phi^+}X_{a,0} \cdot \prod_{a \in \Phi^-}X_{a,\frac{1}{2}} \\
& \subset B \cdot C_1.
\end{align*}
Thus we have the proposition.
\end{proof}

Let $\sigma_0$ be the trivial representation of $G(W_0)$, let $s_i$ be a complex number for $i=1, \ldots, r$, let $\sigma_i$ be the character $| \ |^{s_i}$ of $\GL_1(D)$ for $i=1, \ldots, r$. Then, $\sigma = \otimes_{i=0}^r\sigma_i$ is a character of the Levi subgroup of $B$. Let $\pi$ be an irreducible subquotient representation of $\Ind_B^{G(W)}(\sigma)$ having a non-zero $C_1$-fixed vector. Then, we have the following formula of a zeta integral with a certain section and a matrix coefficient:

\begin{prop}\label{zeta formula2}
Let $f_s^\circ\in I(s,1_{F^\times})^{K({\ue'}^\Box)}$ be a non-zero $K({\ue'}^\Box)$-invariant section with $f_s^\circ(1) = 1$, let $\xi^\circ$ be the $C_1$-fixed matrix coefficient of $\pi$. Then, we have 
\[
Z(f_s^\circ,\xi^\circ) = |C_1|\cdot\frac{S(q^{-s})}{d^W(s)}\prod_{i=0}^{r}L^{W_i}(s+\frac{1}{2},\sigma_i)
\]
for some self-reciprocal monic polynomial $S(T)$ of degree
\[
f_W
=\begin{cases}
1 & (-\epsilon = -1, n_0=2, \chi_W \mbox{ is unramified}),\\
0 & (\mbox{otherwise}).
\end{cases}
\] 
Here we set
\[
d^W(s)=
\begin{cases}
\zeta_F(s+n+\frac{1}{2})\prod_{i=1}^{\lfloor n/2 \rfloor}\zeta_F(2s+2n+1-4i) 
& (-\epsilon = 1), \\
\prod_{i=1}^{\lceil n/2 \rceil} \zeta_F(2s+2n+3-4i) & (-\epsilon = -1).
\end{cases}
\]
Note that if $n_0=0$, then $L^{W_0}(s, 1_{W_0}\times1)$ denotes
\[
\begin{cases}
\zeta_F(s) & (-\epsilon = 1),\\
1 & (-\epsilon=-1).
\end{cases}
\]
\end{prop}

Note that we will determine $S(T)$ and $|C_1|$ later (Propositions \ref{sxsx} and \ref{c1c1}).

\begin{rem}\label{correction1}
Proposition \ref{zeta formula2} differs from \cite[Proposition 8.3]{Kak20} at the definition of $f_W$ in the case $n_0=3$ and the definition of $L^{W_0}(s, 1_{G(W_0)}\times1_{F^\times})$ in the case $n_0=0$, $-\epsilon = 1$.
The former is caused by an error in the computation of the $\gamma$-factor, which is modified by \eqref{teisei1}. And the latter is caused by a typo. 
\end{rem}

\begin{proof}
We can deform the doubling zeta integral to the summation 
\[
Z^W(f_s^\circ,\xi^\circ) = \int_{C_1}\xi^\circ(g) \: dg + \int_{G(W) - C_1} f_s^\circ((g,1))\xi^\circ(g) \: dg.
\]
If $s_0$ is a sufficiently large real number so that $Z^W(f_s^\circ, \xi^\circ)$ converges absolutely, then, by \cite[Lemma 8.4]{Kak20}, we have
\begin{align*}
\left| \int_{G(W)-C_1}f_s^\circ((g,1))\xi^\circ(g) \right|
& \leq \int_{G(W)-C_1}|\Delta((g,1))|^{s-s_0}|f_{s_0}^\circ((g,1))\xi^\circ(g)| \: dg \\
& \leq q^{-(\Re s - s_0)}\int_{G(W)}|f_{s_0}^\circ((g,1))\xi^\circ(g)|\: dg
\end{align*}
for $\Re s > s_0$. Thus we have
\begin{align}\label{limit zeta}
\lim_{\Re s\rightarrow\infty} Z^W(f_s^\circ,\xi^\circ) = |C_1|.
\end{align}
Put
\[
\Xi(q^{-s}) \coloneqq   \frac{Z^W(f_s^\circ, \xi^\circ)}{\prod_{i=0}^rL^{W_i}(s+\frac{1}{2}, \sigma_i\times1_{F^\times})}.
\]
Then, by the ``g.c.d property'' (\cite[Theorem 5.2]{Yam14} and \cite[Lemma 6.1]{Yam14}) implies that $\Xi(q^{-s})$ is a polynomial in $q^{-s}$ and $q^{s}$. Moreover, by \eqref{limit zeta}, it is a polynomial of $q^{-s}$ with the constant term $|C_1|$. Put $D(q^{-s}) \coloneqq   d^W(s)$. Once we prove the equation
\begin{align}\label{goalzeta}
\Xi(q^{-s})D(q^s) = (q^{-s})^{f_W}\cdot \Xi(q^s) D(q^{-s}),
\end{align}
one can deduce that 
\[
\Xi(q^{-s}) = |C_1|\cdot S(q^{-s})D(q^{-s})
\]
for some monic self-reciprocal monic polynomial of degree $f_W$ 
 since $q^{-ts}D(q^s)$ is a polynomial of $q^{-s}$ which is coprime to $D(q^{-s})$ for sufficiently large $t$, which proves the proposition.

In the following, we prove the equation \eqref{goalzeta}. By the definition of the $\gamma$-factor, we have
\begin{align*}
&c(s, 1_{F^\times}, A, \psi)^{-1}R(s, 1_{F^\times}, A, \psi) \cdot Z^W(M(s,1_{F^\times})f_s^\circ, \xi^\circ) \\
&= c_\pi(-1)\cdot \gamma^W(s+\frac{1}{2}, \pi\boxtimes1_{F^\times}, \psi) Z^W(f_s^\circ, \xi^\circ).
\end{align*}
By comparing this with the equation
\begin{align*}
&c(s, 1_{F^\times}, A, \psi)^{-1}R(s,1_{F^\times},A,\psi)M^*(s,1_{F^\times},A,\psi)f_s^\circ  \\
&= q^{-n's}|N(R(\ue))|^{-s}\epsilon(\frac{1}{2}, \chi_{W}, \psi)\cdot \frac{D(q^{-s})}{D(q^s)}f_{-s}^\circ
\end{align*}
where
\[
n' = \begin{cases}2\lceil\frac{n}{2}\rceil & (-\epsilon = 1), \\
                       2\lfloor \frac{n}{2}\rfloor & (-\epsilon = -1), \end{cases}
\]
we obtain
\begin{align*}
\Xi(q^{-s}) D(q^s) &= D(q^{-s})\Xi(q^s) \\
&\times|N(R(\ue))|^{-s}q^{-n's}\cdot\frac{\epsilon(\frac{1}{2},\chi_W, \psi)}{\gamma(s+\frac{1}{2}, \pi\boxtimes1_{F^\times}, \psi)}\cdot\frac{\prod_{i=0}^rL^{W_i}(-s+\frac{1}{2}, \sigma_i^\vee\times 1_{F^\times})}{\prod_{i=0}^rL^{W_i}(s+\frac{1}{2}, \sigma_i\times 1_{F^\times})}.
\end{align*}

Moreover, by Lemma \ref{triv gamma}, we have
\begin{align}\label{teisei1}
\gamma^W(s+\frac{1}{2}, \pi\times1_{F^\times}, \psi) = q^{-\lambda s}\cdot \epsilon^W(\frac{1}{2}, \chi_W, \psi) \prod_{i=0}^r\frac{L^{W_i}(-s+\frac{1}{2}, \sigma_i^\vee\times1_{F^\times})}{L^{W_i}(s+\frac{1}{2}, \sigma\times1_{F^\times})}
\end{align}
where
\[
\lambda = 
\begin{cases}
2\lceil \frac{n}{2} \rceil & (-\epsilon = 1), \\  
2\lfloor \frac{n}{2} \rfloor & (-\epsilon= -1, \ n \not\equiv 3 \mbox{ mod } 4, \ \chi_W \mbox{ is unramified}), \\
2\lfloor \frac{n}{2} \rfloor + 1& (\mbox{otherwise}). 
\end{cases}      
\]
Therefore, 
\begin{align*}
\Xi(q^{-s}) D(q^s) &= D(q^{-s})\Xi(q^s) \cdot q^{-(n'-\lambda)s}\cdot |N(R(\ue))|^{-s} \\
&=D(q^{-s})\Xi(q^s)\cdot (q^{-s})^{f_W}.
\end{align*}
Hence we have the equation \eqref{goalzeta}, and we have the proposition.
\end{proof}

For the polynomial $S(T)$, we have the following:

\begin{prop}\label{sxsx}
We have
\begin{align*}\label{alpha, S, ippan}
S(T) =\begin{cases} T^2 + (q^{\frac{1}{2}} + q^{-\frac{1}{2}})T + 1 & (-\epsilon = -1, \  n_0=2, \ \chi_W \mbox{ is unramified}), \\
1 & (\mbox{ otherwise}).\end{cases}
\end{align*}
\end{prop}

\begin{proof}
We have $f_W =0$ in the cases other than  $-\epsilon = -1$, $n_0=2$, and $\chi_W$ is unramified.
Thus the proposition is clear for the second case. 
Consider the case $n=n_0=2$ and $\chi_W$ is unramified. 
Since $G(W)$ is compact, $Z(f_s^\circ, \xi^\circ)$ is a polynomial in $q^{-s}$. In other words, 
\[
S(q^{-s})\frac{\zeta_F(s+\frac{3}{2})L(s+\frac{1}{2}, \chi_W)}{\zeta_F(2s+3)}
\]
is a polynomial. 
Thus, we can conclude that $(1+q^{-\frac{1}{2}}T)$ divides $S(T)$. 
Such a self-reciprocal polynomial is only $(1+q^{-\frac{1}{2}}T)(1+q^{\frac{1}{2}}T)$. Hence we have
\[
S(T) = T^2 +(q^{\frac{1}{2}} + q^{-\frac{1}{2}})T+ 1.
\]
Now, suppose that $-\epsilon = -1$, $n>n_0=2$, and $\chi_W$ is unramified.
We recall a certain intertwining operator associated with the parabolic induction.
Let $Q(X^\Box)$ be the parabolic subgroup of $G(W^\Box)$ preserving $X^\Box$, 
let $U(X^\Box)$ be the unipotent radical of $Q(X^\Box)$, let $M$ be the Levi-subgroup of $Q(X^\Box)$,
and let $I^X(s,1_{F^\times})$ be the space of smooth functions $f$ on $\GL(X^\Box)$ satisfying
\[
f(pg) = |N(p|_{X^\tru})|^{-(s + r)} |N(p|_{X^\trd})|^{s + r} f(g)
\]
for $p \in P'(X^\tru)$ and $g \in \GL(X^\Box)$. Here, we denote by $P'(X^\tru)$ the parabolic subgroup of $\GL(X^\Box)$ preserving $X^\tru$, by $p|_{X^\tru}$ (resp.~ $p|_{X^\trd}$) the restriction of $p$ to $X^\tru$ (resp.~ $X^\trd$), and by $N$ the reduced norm of ${\rm End}(X^\tru)$ (resp.~ ${\rm End}(X^\trd)$). For a coefficient $\xi$ of an irreducible representation of $\GL(X)$ and a section $f \in I(s, 1_{F^\times})$, we define the doubling zeta integral by
\[
Z^X(f, \xi) = \int_{\GL(X)}f(\iota_X(g,1))\xi(g) \: dg
\]
where $\iota_X\colon \GL(X)\times \GL(X)\rightarrow \GL(X^\Box)$ is the embedding induced by the natural action of $\GL(X)\times\GL(X)$ on $X^\Box$.
Then, there is an intertwining map
\[
\Psi(s)\colon I^W(s,1_{F^\times})\rightarrow \Ind_{Q(X^\Box)}^{G(W^\Box)}(I^X(s,1_{F^\times})\otimes I^{W_0}(s,1_{F^\times})\otimes|\Delta_{(X, W_0):W}|)
\colon f_s \mapsto (g \mapsto [\Phi(s)f_s]_g)
\]
(see \cite[Proposition 4.1]{Yam14}). Although we omit the definition, we note the relation
\[
[\Phi(s)f_s^\circ]_e = J(s) {f_s'}^\circ\otimes{f_s''}^\circ
\]
where ${f_s'}^\circ$ (resp.~ ${f_s''}^\circ$) is the unique $\GL_r(\cO_D)$-invariant section of $I^X(s,1_{F^\times})$ 
(resp.~ the unique $K({\ue_0'}^\Box)$-invariant section of $I^{W_0}(s,1_{F^\times})$) so that ${f_s'}^\circ(1) = 1$ (resp.~ ${f_s''}^\circ(1)=1$), and
\[
J(s) = \int_{U(X^\Box) \cap Q(W^\tru) \backslash U(X^\Box)}f_s^\circ(u) \: du.
\]
Moreover, by Proposition \ref{Iwasawa C_1}, we have
\begin{align*}
Z^W(f_s^\circ, \xi^\circ)
&=|C_1|\int_{Q} f_s^\circ((g,1)) \: dg\\
&=|C_1|\int_{M} [\Psi(s)f_s^\circ]((m,1)) \: dm \\
&=|C_1|J(s) Z^{W_0}({f_s'}^\circ, {\xi'}^\circ) Z^X({f_s''}^\circ, {\xi''}^\circ)\\
&=|C_1|J(s)S(q^{-s}) \frac{L^{W_0}(s+\frac{1}{2}, 1_{G(W)}\times 1_{F^\times})}{d^{W_0}(s)} 
   \cdot\frac{L^X(s+\frac{1}{2}, \sigma)}{d^X(s)} \\
&=|C_1|S^{W_0}(q^{-s}) \frac{J(s)}{d^{W_0}(s)d^X(s)}L^W(s+\frac{1}{2}, 1_{G(W)}\times 1_{F^\times}).
\end{align*}
Thus, we obtain
\[
S^W(q^{-s}) = S^{W_0}(q^{-s}) \times J(s) \frac{d^W(s)}{d^{W_0}(s)d^X(s)}.
\]
However, since $J(s)$ does not have a pole in $\Re s >-1$ (\cite[Lemma 5.1]{Yam14}) and $d^W(s), d^{W_0}(s), d^X(s)$ has neither a pole nor a zero at $s=\pi \sqrt{-1} \pm\frac{1}{2}$, we can conclude that $S^W(X)$ is diveded by $(1 + q^{\pm \frac{1}{2}}T)$. 
Thus, we have $S^W(T) = S^{W_0}(T)$. Hence, we finish the proof of the proposition.
\end{proof}

Finally, by the formula of $\alpha_1(W)$ (Proposition \ref{alpha_1 comp}), we can determine the volume $|C_1|$ of $C_1$:

\begin{prop}\label{c1c1}
\begin{enumerate}
\item In the case $-\epsilon = 1$, we have
\[
|C_1| = q^{-2\lfloor n/2 \rfloor \lceil n/2 \rceil - \lceil n/2 \rceil} \prod_{i=1}^{\lfloor n/2 \rfloor} (1+ q^{-(2i-1)})(1-q^{-2i}).
\]
\item In the case $-\epsilon = -1$, we have
\begin{align*}
|C_1|&= |N(R(\ue))|^{-\rho} q^{-(2\lfloor n/2 \rfloor \lceil n/2 \rceil - \lfloor n/2 \rfloor)} \\
&\times\begin{cases}
\prod_{i=1}^{\lfloor n/2 \rfloor}(1+q^{-(2i-1)})\prod_{i=1}^{\lfloor n/2 \rfloor}(1-q^{-2i}) 
& (n_0=0), \\
\prod_{i=1}^{\lceil n/2 \rceil}(1+q^{-(2i-1)})\prod_{i=1}^{\lfloor n/2 \rfloor}(1-q^{-2i}) 
& (n_0=1, \chi_W:\mbox{unramified}), \\
\prod_{i=1}^{\lfloor n/2 \rfloor}(1+q^{-(2i-1)})\prod_{i=1}^{\lfloor n/2 \rfloor}(1-q^{-2i}) 
& (n_0=1, \chi_W:\mbox{ramified}), \\
\prod_{i=1}^{\lfloor n/2 \rfloor-1}(1+q^{-(2i-1)})\prod_{i=1}^{\lfloor n/2 \rfloor-1}(1-q^{-2i}) 
& (n_0=2, \chi_W: \mbox{unramified}), \\
\prod_{i=1}^{\lfloor n/2 \rfloor}(1+q^{-(2i-1)})\prod_{i=1}^{\lfloor n/2 \rfloor-1}(1-q^{-2i})
&(n_0=2, \chi_W:\mbox{ramified}), \\
\prod_{i=1}^{\lfloor n/2 \rfloor}(1+q^{-(2i-1)})\prod_{i=1}^{\lfloor n/2 \rfloor-1}(1-q^{-2i})
& (n_0=3). 
\end{cases}
\end{align*}
\end{enumerate} 
\end{prop}

Proposition \ref{sxsx} and Proposition \ref{c1c1} give a completion of the formula in Proposition \ref{zeta formula2}.

%%%%%%%%%%%%%%%%%%%%%%%%%%%%%%%%%%  reference   %%%%%%%%%%%%%%%%%%%%%%%%%%%%%%%

\bibliographystyle{alpha}
\bibliography{fd.bib}

\begin{thebibliography}{MVW87}

\bibitem[AP05]{AP05}
Anne-Marie Aubert and Roger Plymen.
\newblock Plancherel measure for {${\rm GL}(n,F)$} and {${\rm GL}(m,D)$}:
  explicit formulas and {B}ernstein decomposition.
\newblock {\em J. Number Theory}, 112(1):26--66, 2005.

\bibitem[Art13]{Art13}
James Arthur.
\newblock {\em The endoscopic classification of representations}, volume~61 of
  {\em American Mathematical Society Colloquium Publications}.
\newblock American Mathematical Society, Providence, RI, 2013.
\newblock Orthogonal and symplectic groups.

\bibitem[Bad08]{Bad08}
Alexandru~Ioan Badulescu.
\newblock Global {J}acquet-{L}anglands correspondence, multiplicity one and
  classification of automorphic representations.
\newblock {\em Invent. Math.}, 172(2):383--438, 2008.
\newblock With an appendix by Neven Grbac.

\bibitem[Bor79]{Bor79}
A.~Borel.
\newblock Automorphic {$L$}-functions.
\newblock In {\em Automorphic forms, representations and {$L$}-functions
  ({P}roc. {S}ympos. {P}ure {M}ath., {O}regon {S}tate {U}niv., {C}orvallis,
  {O}re., 1977), {P}art 2}, Proc. Sympos. Pure Math., XXXIII, pages 27--61.
  Amer. Math. Soc., Providence, R.I., 1979.

\bibitem[BP18]{BP18}
Raphael Beuzart-Plessis.
\newblock Plancherel formula for $\mathrm{GL}_n({F})\backslash
  \mathrm{GL}_n({E})$ and applications to the {I}chino-{I}keda and formal
  degree conjectures for unitary groups, 2018.

\bibitem[BT72]{BT72}
F.~Bruhat and J.~Tits.
\newblock Groupes r\'{e}ductifs sur un corps local.
\newblock {\em Inst. Hautes \'{E}tudes Sci. Publ. Math.}, (41):5--251, 1972.

\bibitem[BZ76]{BZ76}
I.~N. Bern\v{s}te\u{\i}n and A.~V. Zelevinski\u{\i}.
\newblock Representations of the group {$GL(n,F),$} where {$F$} is a local
  non-{A}rchimedean field.
\newblock {\em Uspehi Mat. Nauk}, 31(3(189)):5--70, 1976.

\bibitem[BZ77]{BZ77}
I.~N. Bernstein and A.~V. Zelevinsky.
\newblock Induced representations of reductive {$p$}-adic groups. {I}.
\newblock {\em Ann. Sci. \'{E}cole Norm. Sup. (4)}, 10(4):441--472, 1977.

\bibitem[Cas80]{Cas80}
W.~Casselman.
\newblock The unramified principal series of {$p$}-adic groups. {I}. {T}he
  spherical function.
\newblock {\em Compositio Math.}, 40(3):387--406, 1980.

\bibitem[Cas89]{Cas89}
W.~Casselman.
\newblock Canonical extensions of {H}arish-{C}handra modules to representations
  of {$G$}.
\newblock {\em Canad. J. Math.}, 41(3):385--438, 1989.

\bibitem[Cho17]{Cho17}
Kwangho Choiy.
\newblock The local {L}anglands conjecture for the {$p$}-adic inner form of
  {$\rm Sp_4$}.
\newblock {\em Int. Math. Res. Not. IMRN}, (6):1830--1889, 2017.

\bibitem[Cho19]{Cho19}
Kwangho Choiy.
\newblock On multiplicity in restriction of tempered representations of
  {$p$}-adic groups.
\newblock {\em Math. Z.}, 291(1-2):449--471, 2019.

\bibitem[DKV84]{DKV84}
P.~Deligne, D.~Kazhdan, and M.-F. Vign\'{e}ras.
\newblock Repr\'{e}sentations des alg\`ebres centrales simples {$p$}-adiques.
\newblock In {\em Representations of reductive groups over a local field},
  Travaux en Cours, pages 33--117. Hermann, Paris, 1984.

\bibitem[GG99]{GG99}
Benedict~H. Gross and Wee~Teck Gan.
\newblock Haar measure and the {A}rtin conductor.
\newblock {\em Trans. Amer. Math. Soc.}, 351(4):1691--1704, 1999.

\bibitem[GI11]{GI11}
Wee~Teck Gan and Atsushi Ichino.
\newblock On endoscopy and the refined gross^^e2^^80^^93prasad conjecture for
  (so5, so4).
\newblock {\em Journal of the Institute of Mathematics of Jussieu},
  10(2):235^^e2^^80^^93324, 2011.

\bibitem[GI14]{GI14}
Wee~Teck Gan and Atsushi Ichino.
\newblock Formal degrees and local theta correspondence.
\newblock {\em Invent. Math.}, 195(3):509--672, 2014.

\bibitem[GI16]{GI16}
Wee~Teck Gan and Atsushi Ichino.
\newblock The {G}ross-{P}rasad conjecture and local theta correspondence.
\newblock {\em Invent. Math.}, 206(3):705--799, 2016.

\bibitem[GJ72]{GJ72}
Roger Godement and Herv\'e Jacquet.
\newblock {\em Zeta functions of simple algebras}.
\newblock Lecture Notes in Mathematics, Vol. 260. Springer-Verlag, Berlin-New
  York, 1972.

\bibitem[GR10]{GR10}
Benedict~H. Gross and Mark Reeder.
\newblock Arithmetic invariants of discrete {L}anglands parameters.
\newblock {\em Duke Math. J.}, 154(3):431--508, 2010.

\bibitem[Gro97]{Gro97}
Benedict~H. Gross.
\newblock On the motive of a reductive group.
\newblock {\em Invent. Math.}, 130(2):287--313, 1997.

\bibitem[GS12]{GS12}
Wee~Teck Gan and Gordan Savin.
\newblock Representations of metaplectic groups {I}: epsilon dichotomy and
  local {L}anglands correspondence.
\newblock {\em Compos. Math.}, 148(6):1655--1694, 2012.

\bibitem[GS17]{GS17}
Wee~Teck Gan and Binyong Sun.
\newblock The {H}owe duality conjecture: quaternionic case.
\newblock In {\em Representation theory, number theory, and invariant theory},
  volume 323 of {\em Progr. Math.}, pages 175--192. Birkh\"{a}user/Springer,
  Cham, 2017.

\bibitem[GT14]{GT14}
Wee~Teck Gan and Welly Tantono.
\newblock The local {L}anglands conjecture for {$\rm GSp(4)$}, {II}: {T}he case
  of inner forms.
\newblock {\em Amer. J. Math.}, 136(3):761--805, 2014.

\bibitem[GT16]{GT16}
Wee~Teck Gan and Shuichiro Takeda.
\newblock A proof of the {H}owe duality conjecture.
\newblock {\em J. Amer. Math. Soc.}, 29(2):473--493, 2016.

\bibitem[Han11]{Han11}
Marcela Hanzer.
\newblock Rank one reducibility for unitary groups.
\newblock {\em Glas. Mat. Ser. III}, 46(66)(1):121--148, 2011.

\bibitem[Hei04]{Hei04}
Volker Heiermann.
\newblock D\'{e}composition spectrale et repr\'{e}sentations sp\'{e}ciales d'un
  groupe r\'{e}ductif {$p$}-adique.
\newblock {\em J. Inst. Math. Jussieu}, 3(3):327--395, 2004.

\bibitem[Hen84]{Hen84}
Guy Henniart.
\newblock La conjecture de {L}anglands locale pour {${\rm GL}(3)$}.
\newblock {\em M\'em. Soc. Math. France (N.S.)}, (11-12):186, 1984.

\bibitem[HII08]{HII08}
Kaoru Hiraga, Atsushi Ichino, and Tamotsu Ikeda.
\newblock Formal degrees and adjoint {$\gamma$}-factors.
\newblock {\em J. Amer. Math. Soc.}, 21(1):283--304, 2008.

\bibitem[HKS96]{HKS96}
Michael Harris, Stephen~S. Kudla, and William~J. Sweet.
\newblock Theta dichotomy for unitary groups.
\newblock {\em J. Amer. Math. Soc.}, 9(4):941--1004, 1996.

\bibitem[Ike19]{Ike19}
Yasuhiko Ikematsu.
\newblock Local theta lift for {$p$}-adic unitary dual pairs {$\rm
  U(2)\times\rm U(1)$} and {$\rm U(2)\times\rm U(3)$}.
\newblock {\em Kyoto J. Math.}, 59(4):1075--1110, 2019.

\bibitem[ILM17]{ILM17}
Atsushi Ichino, Erez Lapid, and Zhengyu Mao.
\newblock On the formal degrees of square-integrable representations of odd
  special orthogonal and metaplectic groups.
\newblock {\em Duke Math. J.}, 166(7):1301--1348, 2017.

\bibitem[JNQ10]{GNQ10}
Dihua Jiang, Chufeng Nien, and Yujun Qin.
\newblock Symplectic supercuspidal representations of {${\rm GL}(2n)$} over
  {$p$}-adic fields.
\newblock {\em Pacific J. Math.}, 245(2):273--313, 2010.

\bibitem[Kak20]{Kak20}
Hirotaka Kakuhama.
\newblock On the local factors of irreducible representations of quaternionic
  unitary groups.
\newblock {\em Manuscripta Math.}, 163(1-2):57--86, 2020.

\bibitem[Kal16]{Kal16}
Tasho Kaletha.
\newblock Rigid inner forms of real and {$p$}-adic groups.
\newblock {\em Ann. of Math. (2)}, 184(2):559--632, 2016.

\bibitem[Kar79]{Kar79}
Martin~L. Karel.
\newblock Functional equations of {W}hittaker functions on {$p$}-adic groups.
\newblock {\em Amer. J. Math.}, 101(6):1303--1325, 1979.

\bibitem[Kot84]{Kot84}
Robert~E. Kottwitz.
\newblock Stable trace formula: cuspidal tempered terms.
\newblock {\em Duke Math. J.}, 51(3):611--650, 1984.

\bibitem[Kot97]{Kot97}
Robert~E. Kottwitz.
\newblock Isocrystals with additional structure. {II}.
\newblock {\em Compositio Math.}, 109(3):255--339, 1997.

\bibitem[Kud94]{Kud94}
Stephen~S. Kudla.
\newblock Splitting metaplectic covers of dual reductive pairs.
\newblock {\em Israel J. Math.}, 87(1-3):361--401, 1994.

\bibitem[Lab85]{Lab85}
J.-P. Labesse.
\newblock Cohomologie, {$L$}-groupes et fonctorialit\'{e}.
\newblock {\em Compositio Math.}, 55(2):163--184, 1985.

\bibitem[Lan76]{Lan76}
Robert~P. Langlands.
\newblock {\em On the functional equations satisfied by {E}isenstein series}.
\newblock Lecture Notes in Mathematics, Vol. 544. Springer-Verlag, Berlin-New
  York, 1976.

\bibitem[Li89]{Li89}
Jian-Shu Li.
\newblock Singular unitary representations of classical groups.
\newblock {\em Invent. Math.}, 97(2):237--255, 1989.

\bibitem[LR05]{LR05}
Erez~M. Lapid and Stephen Rallis.
\newblock On the local factors of representations of classical groups.
\newblock In {\em Automorphic representations, {$L$}-functions and
  applications: progress and prospects}, volume~11 of {\em Ohio State Univ.
  Math. Res. Inst. Publ.}, pages 309--359. de Gruyter, Berlin, 2005.

\bibitem[Mie20]{Mie20}
Yoichi Mieda.
\newblock Parity of the {L}anglands parameters of conjugate self-dual
  representations of {${\rm GL}(n)$} and the local {J}acquet-{L}anglands
  correspondence.
\newblock {\em J. Inst. Math. Jussieu}, 19(6):2017--2043, 2020.

\bibitem[Mui06]{Mui06}
Goran Mui\'{c}.
\newblock On the structure of the full lift for the {H}owe correspondence of
  {$({\rm Sp}(n),{\rm O}(V))$} for rank-one reducibilities.
\newblock {\em Canad. Math. Bull.}, 49(4):578--591, 2006.

\bibitem[MVW87]{MVW87}
Colette M{\oe}glin, Marie-France Vign\'{e}ras, and Jean-Loup Waldspurger.
\newblock {\em Correspondances de {H}owe sur un corps {$p$}-adique}, volume
  1291 of {\em Lecture Notes in Mathematics}.
\newblock Springer-Verlag, Berlin, 1987.

\bibitem[Per81]{Pe80}
Patrice Perrin.
\newblock Repr\'{e}sentations de {S}chr\"{o}dinger, indice de {M}aslov et
  groupe metaplectique.
\newblock In {\em Noncommutative harmonic analysis and {L}ie groups
  ({M}arseille, 1980)}, volume 880 of {\em Lecture Notes in Math.}, pages
  370--407. Springer, Berlin-New York, 1981.

\bibitem[PR]{PR94}
V.P. Platonov and A.S Rapinchuk.
\newblock {\em Algebraic groups and number theory. Pure and Applied
  Mathematics}.
\newblock Pure and Applied Mathematics.

\bibitem[PR08]{PR08}
G.~Pappas and M.~Rapoport.
\newblock Twisted loop groups and their affine flag varieties.
\newblock {\em Adv. Math.}, 219(1):118--198, 2008.
\newblock With an appendix by T. Haines and Rapoport.

\bibitem[Rag02]{Rag02}
A.~Raghuram.
\newblock On representations of {$p$}-adic {${\rm GL}_2(\mathscr{D})$}.
\newblock {\em Pacific J. Math.}, 206(2):451--464, 2002.

\bibitem[RR93]{Rao93}
R.~Ranga~Rao.
\newblock On some explicit formulas in the theory of {W}eil representation.
\newblock {\em Pacific J. Math.}, 157(2):335--371, 1993.

\bibitem[Sau97]{Sau97}
Fran\c{c}ois Sauvageot.
\newblock Principe de densit\'{e} pour les groupes r\'{e}ductifs.
\newblock {\em Compositio Math.}, 108(2):151--184, 1997.

\bibitem[Sch85]{Sch85}
Winfried Scharlau.
\newblock {\em Quadratic and {H}ermitian forms}, volume 270 of {\em Grundlehren
  der Mathematischen Wissenschaften [Fundamental Principles of Mathematical
  Sciences]}.
\newblock Springer-Verlag, Berlin, 1985.

\bibitem[Sha81]{Sha81}
Freydoon Shahidi.
\newblock On certain {$L$}-functions.
\newblock {\em Amer. J. Math.}, 103(2):297--355, 1981.

\bibitem[Sha90]{Sha90}
Freydoon Shahidi.
\newblock A proof of {L}anglands' conjecture on {P}lancherel measures;
  complementary series for {$p$}-adic groups.
\newblock {\em Ann. of Math. (2)}, 132(2):273--330, 1990.

\bibitem[Shi99]{Shi99}
Goro Shimura.
\newblock Some exact formulas on quaternion unitary groups.
\newblock {\em J. Reine Angew. Math.}, 509:67--102, 1999.

\bibitem[SZ15]{SZ15}
Binyong Sun and Chen-Bo Zhu.
\newblock Conservation relations for local theta correspondence.
\newblock {\em J. Amer. Math. Soc.}, 28(4):939--983, 2015.

\bibitem[Wal90]{Wal90}
J.-L. Waldspurger.
\newblock D\'{e}monstration d'une conjecture de dualit\'{e} de {H}owe dans le
  cas {$p$}-adique, {$p\neq 2$}.
\newblock In {\em Festschrift in honor of {I}. {I}. {P}iatetski-{S}hapiro on
  the occasion of his sixtieth birthday, {P}art {I} ({R}amat {A}viv, 1989)},
  volume~2 of {\em Israel Math. Conf. Proc.}, pages 267--324. Weizmann,
  Jerusalem, 1990.

\bibitem[Wal03]{Wal03}
J.-L. Waldspurger.
\newblock La formule de {P}lancherel pour les groupes {$p$}-adiques (d'apr\`es
  {H}arish-{C}handra).
\newblock {\em J. Inst. Math. Jussieu}, 2(2):235--333, 2003.

\bibitem[Yam11]{Yam11}
Shunsuke Yamana.
\newblock Degenerate principal series representations for quaternionic unitary
  groups.
\newblock {\em Israel J. Math.}, 185:77--124, 2011.

\bibitem[Yam14]{Yam14}
Shunsuke Yamana.
\newblock L-functions and theta correspondence for classical groups.
\newblock {\em Invent. Math.}, 196(3):651--732, 2014.

\bibitem[Zel80]{Zel80}
A.~V. Zelevinsky.
\newblock Induced representations of reductive {$p$}-adic groups. {II}. {O}n
  irreducible representations of {${\rm GL}(n)$}.
\newblock {\em Ann. Sci. \'{E}cole Norm. Sup. (4)}, 13(2):165--210, 1980.

\end{thebibliography}

\end{document}